\documentclass[a4paper,10pt,reqno]{amsart}
\usepackage{amssymb,latexsym,amsmath,amsfonts}
\usepackage{amsthm}
\usepackage[mathscr]{eucal}
\usepackage{rotating}
\usepackage{multirow}

\numberwithin{equation}{section}
\newtheorem{theorem}{Theorem}
\newtheorem{lemma}{Lemma}
\newtheorem{proposition}{Proposition}

\theoremstyle{definition}

\newtheorem{example}[theorem]{Example}

\theoremstyle{remark}

\begin{document}

\title[Decay rates of non--hyperbolic unbounded delay FDEs] 
{Exact and memory--dependent decay rates to the non--hyperbolic equilibrium of differential equations with unbounded delay and maximum
functional}

\author{John A. D. Appleby}
\address{School of Mathematical
Sciences, Dublin City University, Glasnevin, Dublin 9, Ireland}
\email{john.appleby@dcu.ie} 

\thanks{John Appleby gratefully acknowledges the support of a SQuaRE activity entitled ``Stochastic stabilisation of limit-cycle dynamics in ecology and neuroscience'' funded by the American Institute of Mathematics. 
} 
\subjclass{39K20}
\keywords{functional differential equation; unbounded delay; asymptotic stability; Halanay's
inequality; equations with maxima.}
\date{1 July 2016}

\begin{abstract}
In this paper, we obtain the exact rates of decay to the non--hyperbolic equilibrium of the solution of a functional differential equation
with maxima and unbounded delay. We study the convergence rates for both locally and globally stable solutions.
We also give examples showing how the rate of growth of decay of solutions depends on the rate of growth of the unbounded delay
as well as the nonlinearity local to the equilibrium.
\end{abstract}

\maketitle

\section{Introduction}
A large literature has developed in the past decades concerning the rate of decay to equilibrium in delay differential equations with unbounded delay. Some representative papers include Krisztin~\cite{Kris85,Kris90}, Kato~\cite{Kato}, Diblik~\cite{Diblik},
Cermak~\cite{Cermak:04a}, and Haddock and Krisztin~\cite{HaddKris84, HaddKris86}. In the last two papers in particular, the rate of convergence is considered for equations in which the leading order space behaviour at the equilibrium is of smaller than linear order. The results can also be applied to stochastic delay equations in both the nonlinear (Appleby and Rodkina~\cite{jaar}) and linear case (Appleby~\cite{3}).  

An especially interesting equation which has received much
attention is one with proportional delay, called the pantograph
equation: fundamental work on the asymptotic behaviour dates back
to Kato and McLeod~\cite{KatoMcL}, Fox et.~al.~\cite{Foxetal},
Ockendon and Tayler~\cite{OckTay}. Complex--valued and finite
dimensional treatments were considered by Carr and
Dyson~\cite{CarrD,CarrD2}, while more recent treatments and
generalisations include Iserles~\cite{Iser0} and
Makay and Terjeki~\cite{MakTer}.

Another category of functional differential equation are those
with maximum functionals on the righthand side. Inspiration for
the study of these equations may be traced to work of
Halanay~\cite{Halanay:66}. The asymptotic behaviour of solutions
of such Halanay type inequalities is considered by e.g., Baker and
Tang~\cite{BakerTang96}, Mohamad and Gopalsamy~\cite{MohGop}, Liz and Trofimchuk~\cite{LizTrom}, 
Ivanov, Liz and Trofimcuk~\cite{LizTromIv} and Liz, Ivanov and Ferreiro~\cite{LizIvFerr} for equations with
finite memory.


In this paper, we consider non--hyperbolic equations (as in e.g.,~\cite{HaddKris86}) with unbounded delay (as in e.g.,~\cite{KatoMcL}) 
as well as equations with max--type functionals (as in e.g.,~\cite{LizTrom}). In particular, we give a complete characterisation of the rate of convergence of the solution $x$ of
the delay differential equation
\begin{equation} \label{eq.introdde}
x'(t)=-ag(x(t))+bg(x(t-\tau(t))), \quad t>0; \quad x(t)=\psi(t), \quad t\leq 0
\end{equation}
and the functional differential equation
\begin{equation}  \label{eq.intromaxfde}
x'(t)=-ag(x(t))+b\max_{t-\tau(t)\leq s\leq t} g(x(s)), \quad t>0;, \quad x(t)=\psi(t), \quad t\leq 0
\end{equation}
to zero as $t\to\infty$. We are interested in equations in which $g(0)=0$ but $g'(0)=0$, so that the equilibrium
solution $x(t)=0$ for all $t\geq 0$ which arises from the initial condition $\psi(t)=0$ for all $t\leq 0$ is non--hyperbolic.
In order to confine attention to a class of equations, we assume that $g$ is regularly varying with index $\beta>1$.
Of course we ask that $g$ is increasing and in $C^1$ on an interval $(0,\delta)$. The condition $a>b>0$ is natural if
we require solutions to be positive and for (at least) solutions with small initial conditions $\psi$ to obey $x(t)\to 0$
as $t\to\infty$.

Granted that a solution obeys $x(t)\to 0$ as $t\to\infty$, we are able to determine the convergence rate of both equation \eqref{eq.introdde} and \eqref{eq.intromaxfde}. This rate can be related to the rate of decay to zero of the solution of the related ordinary differential equation
\begin{equation}  \label{eq.introode}
y'(t)=-(a-b)g(y(t)), \quad t>0; \quad y(0)>0.
\end{equation}
By unifying hypotheses used to prove results under slightly different conditions on the nonlinear function $g$, our subsidiary results can be consolidated to give the following main theorem.
\begin{itemize}
\item[(a)] If $\tau(t)/t\to 0$ as $t\to\infty$ and $y$ is the solution of \eqref{eq.introode} then
$x(t)/y(t)\to 1$ as $t\to\infty$
\item[(b)] If $\tau(t)/t\to q\in(0,1)$ with $a>b(1-q)^{-\beta/(\beta-1)}$, and $y$ is the solution of \eqref{eq.introode}, then
\[
1<\liminf_{t\to\infty} \frac{x(t)}{y(t)}\leq \limsup_{t\to\infty} \frac{x(t)}{y(t)}<+\infty,
\]
We conjecture that $x(t)/y(t)\to \Lambda>1$ as $t\to\infty$.
\item[(c)] If $\tau(t)/t\to q\in(0,1)$ with $a<b(1-q)^{-\beta/(\beta-1)}$, and $y$ is the solution of \eqref{eq.introode},
then $x(t)/y(t)\to \infty$ as $t\to\infty$ and moreover
\[
\lim_{t\to\infty} \frac{\log x(t)}{\log t}=-\frac{1}{\beta}\frac{1}{\log(1/(1-q))}\log\left(\frac{a}{b}\right)<0.
\]
\item[(d)] If there exists an auxiliary function $\sigma$ such that $\sigma(t)/t\to\infty$ (which implies $\tau(t)/t\to 1$ as $t\to\infty$),
\[
\int_0^t \frac{1}{\sigma(s)} \,ds\to\infty, \quad \int_{t-\tau(t)}^t \frac{1}{\sigma(s)}\,ds \to 1 \quad\text{as $t\to\infty$}
\]
then $x$ obeys
\[
\lim_{t\to\infty} \frac{\log x(t)}{\int_0^t \frac{1}{\sigma(s)}\,ds} = -\frac{1}{\beta}\log\left(\frac{a}{b}\right).
\]
\end{itemize}
The result in (d) generalises to the nonlinear setting results in Krisztin, \v{C}ermak etc., adapting the approach in \cite{AppBuck:07a} 
used to obtain sharp asymptotic estimates for linear equations. 
Since the equation does not have infinite memory (so we cannot have $\tau(t)/t\to q>1$ as $t\to\infty$),
the results (a)--(d) can reasonably be said to provide a quite complete picture of the relationship between the rate of
convergence, the strength of the nonlinearity $g$, and the rate of growth of the unbounded delay $\tau$ for the class of
nonlinearities considered.

The results show that the rate of convergence is dependent on the delay: while the rate of growth of the delay $\tau$ is less than some critical
rate, the solution inherits the rate of decay of \eqref{eq.introode} exactly. Once the delay grows more rapidly than a critical rate, the
solution no longer inherits the rate of convergence of solutions of \eqref{eq.introode}. We are able not only to identify the critical growth
rate of the delay at which this happens, but also to determine the exact convergence rate of the solution of e.g., \eqref{eq.introdde} whether
there is ``slowly'' growing or ``rapidly'' growing delay. As a by--product, the results also cover the case of bounded delay.
We use comparison--type arguments inspired especially by Appleby and Patterson~\cite{appleby_pattersonEJQTDE,appleby_patterson} which deal with non--hyperbolic ordinary 
and stochastic differential equations, and Appleby and Buckwar~\cite{AppBuck:07a} which 
deals with equations of the form \eqref{eq.introdde} and \eqref{eq.intromaxfde} with $g(x)=x$. 

\section{Notation and Statement of the Problem}

\subsection{Notation}
We recall that a function $f$ is \emph{regularly varying at infinity with exponent
$\alpha\in\mathbb{R}$} if
\[
\lim_{t\to\infty} \frac{f(\lambda t)}{f(t)} =
\lambda^\alpha, \quad \text{for all $\lambda>0$}.
\]
For such a function, we write $f\in \text{RV}_{\infty}(\alpha)$.

A function $f$ is \emph{regularly varying at zero with exponent $\alpha\in\mathbb{R}$} if
\[
\lim_{t\to 0^+} \frac{f(\lambda t)}{f(t)} =
\lambda^\alpha, \quad \text{for all $\lambda>0$}.
\]
For such a function, we write $f\in \text{RV}_{0}(\alpha)$. The exploitation of properties of regularly varying functions in studying asymptotic properties of ordinary and functional differential equations is an active field of research. 
Recent research themes in this direction are recorded in monographs  such as \cite{maric2000regular} and \cite{rehakrv} and all properties of regularly varying functions employed in this paper can be found in the classic text \cite{BGT}. A highly selective list of the properties of regular variation that we have found useful appear in the introduction of~\cite{appleby_pattersonEJQTDE}, a work which concerns ordinary differential equations.  

\subsection{Statement of the problem}
Let $\psi\in C([-\bar{\tau},0];\mathbb{R}^+)$. In this paper we consider the asymptotic behaviour of the solutions of
\begin{equation}\label{T1.c}
  \begin{array}{rcl}
    x'(t) & = & - a g(x(t)) + bg(x(t-\tau(t)),\quad t\geq 0 \\
    x(t) & = & \psi(t),\quad t \in [-\bar{\tau},0]
  \end{array}
\end{equation}
and
\begin{equation}\label{T1.d}
  \begin{array}{rcl}
    x'(t) & = & - a g(x(t)) + b\sup_{t-\tau(t)\leq s\leq t} g(x(s)),\quad t\geq 0 \\
    x(t) & = & \psi(t),\quad t \in [-\bar{\tau},0].
  \end{array}
\end{equation}
We assume that
\begin{equation}\label{eq:agtb}
a > b> 0.
\end{equation}
We assume that the function $g$ obeys
\begin{subequations}\label{eq:g}
\begin{gather}
\label{eq:g1}g(0)=0; \\
\label{eq:g2}\text{$g:[0,\infty)\to[0,\infty)$ is increasing}\\
\label{eq:g3}\text{$g\in C^1((0,\infty);(0,\infty))$};\\
\label{eq:g4}
\text{There exists $\beta>1$ such that $g\in \text{RV}_0(\beta)$}.
\end{gather}
\end{subequations}
The last hypothesis implies that $g'(0)=0$. Define the function $G$ by
\begin{equation} \label{def.G}
G(x)=\int_x^1 \frac{1}{g(u)}\,du, \quad x>0.
\end{equation}
Then $G$ is in $C^1(0,\infty)$ (by \eqref{eq:g3}), is decreasing on $(0,\infty)$ (by \eqref{eq:g2}) and by virtue of the fact that $g'(0)=0$
we have
\begin{equation} \label{eq.Gtoinfty}
\lim_{x\to 0^+} G(x)=+\infty.
\end{equation}
$\tau$ is assumed to satisfy
\begin{subequations}\label{eq:tau}
\begin{gather}\label{T1.0}
  \tau \quad \mbox{is a continuous non-negative function on }
  [0,\infty),\\
\label{T1.a}
\text{There exists a finite $\bar{\tau}> 0$  such that $-\bar{\tau} = \inf_{t\geq 0} t-\tau(t)$.}
\end{gather}
\end{subequations}
Under these conditions there is a unique continuous solution $x$ of \eqref{T1.c} and of \eqref{T1.d} on $[-\bar{\tau},\infty)$. These solutions are
moreover guaranteed to be positive on $[-\bar{\tau},\infty)$.
\section{Discussion of Main Results}
In this section we state, motivate, and discuss results giving the exact rate of decay to zero of solutions
of \eqref{T1.c} and \eqref{T1.d}. The main tool employed is a type of comparison argument.

In the case when the rate of growth of the delay is ``fast'', and the equation has long memory, an important auxiliary function $\sigma$ is introduced which enables the asymptotic behaviour to be determined. Motivation for the method of proof, and the role of the auxiliary function $\sigma$ is given in the following section, along with easily applicable corollaries of the main results. The ease of applicability of these results relies upon being able to determine the appropriate auxiliary function $\sigma$, often as a function asymptotic to $\tau$.

In the case when the rate of growth of the delay is ``slow'' (or the equation has bounded delay), the function $\sigma$ is not required. In this case
we show that the solution of the delay differential equation converges to zero at exactly the same rate as the ordinary differential equation
$x'(t)=-(a-b)g(x(t))$ for $t>0$.

The results of the following theorems are given in Section~\ref{sec.proofs}. In this section we concentrate on stating the main general results, and discuss the role and necessity of the hypotheses on $g$, $\tau$ and the auxiliary function $\sigma$. The implications of the conclusions of the general results are also explored here.

First, we have that solutions of \eqref{T1.c} and of \eqref{T1.d} are uniformly bounded. This is used in later theorems to
show that solutions of \eqref{T1.c} and \eqref{T1.d} tend to zero as $t\to\infty$ and to determine the rate of convergence.
\begin{theorem}\label{thm3.0}
Let $\tau$ be a continuous and non-negative function such that $-\overline{\tau}=\inf_{t\geq 0} t-\tau(t)$.
Let $a>b>0$ and $g$ satisfy \eqref{eq:g1},  \eqref{eq:g3} and suppose $\psi\in
C([-\overline{\tau},0];(0,\infty))$.
\begin{itemize}
\item[(a)]
The solution of
\eqref{T1.c} viz.,
\begin{align*}
    x'(t) &=  - ag(x(t)) + bg(x(t-\tau(t)),\quad t\geq 0 \\
    x(t) &=  \psi(t),\quad t \in [-\bar{\tau},0]
\end{align*}
obeys
\begin{equation}  \label{eq.xbounded}
0<x(t)\leq \max_{-\bar{\tau}\leq s\leq 0}\psi(s), \quad t\geq -\bar{\tau}.
\end{equation}
If moreover $g$ obeys \eqref{eq:g2}, then $x(t)\to 0$ as $t\to\infty$.
\item[(b)] The solution of \eqref{T1.d} viz.,
\begin{align*}
    x'(t) &= -ag(x(t))+b\max_{t-\tau(t)\leq s\leq t} g(x(s)) ,\quad t\geq 0 \\
    x(t) &=  \psi(t),\quad t \in [-\bar{\tau},0]
\end{align*}
obeys \eqref{eq.xbounded}. If moreover $g$ obeys \eqref{eq:g2}, then $x(t)\to 0$ as $t\to\infty$.
\end{itemize}
\end{theorem}
We are interested in solutions of \eqref{T1.c} which tend to zero as $t\to\infty$. In order to guarantee this we assume that
\begin{equation} \label{eq.tminustautoinfty}
\lim_{t\to\infty} \{t-\tau(t)\}=\infty.
\end{equation}
An assumption of this type is reasonable; indeed if $x(t)\to 0$ as $t\to\infty$, we require that
\begin{equation} \label{eq.tminustautoinftysup}
\limsup_{t\to\infty} \{t-\tau(t)\}=\infty.
\end{equation}
To see this, suppose to the contrary that $\limsup_{t\to\infty} t-\tau(t)= \tau_1<+\infty$. Therefore as $t\mapsto t-\tau(t)$ is continuous, there exists $\tau^\ast$ such that $t-\tau(t)\leq \tau^\ast$ for all $t\geq 0$. By \eqref{T1.a} we have $-\bar{\tau}\leq t-\tau(t)\leq \tau^\ast$ for all
$t\geq 0$. Then for all $t\geq 0$ we have
\[
0<x_1:=\min_{s\in[-\bar{\tau},\tau^\ast]} x(s)\leq x(t-\tau(t))\leq \max_{s\in[-\bar{\tau},\tau^\ast]} x(s) =:x_2<+\infty.
\]
Next define $G(t):=bg(x(t-\tau(t)))$ for $t\geq 0$. Define
\[
0<g_1:=\min_{x\in[x_1,x_2]} g(x)\leq \max_{x\in[x_1,x_2]} g(x)=:g_2<+\infty.
\]
Then $bg_1\leq G(t)\leq bg_2$ for all $t\geq 0$. Therefore, as
$x'(t)=-ag(x(t))+G(t)$ for $t>0$ and $x(t)\to 0$ as $t\to\infty$, we have
\[
\liminf_{t\to\infty} x'(t)=\liminf_{t\to\infty} -ag(x(t))+G(t) = \liminf_{t\to\infty} G(t) \geq bg_1>0.
\]
Therefore $x(t)\to\infty$ as $t\to\infty$, a contradiction, and so \eqref{eq.tminustautoinftysup} must hold.

We notice that if $\psi(t)=\psi(0)>0$ for $t\in[-\bar{\tau},0]$ and $a=b$, then the solution of \eqref{T1.c} is $x(t)=\psi(0)>0$
for all $t\geq -\bar{\tau}$. Similarly, if $\psi(t)=\psi(0)>0$ for $t\in[-\bar{\tau},0]$ and $a=b$, then the solution of \eqref{T1.d} is $x(t)=\psi(0)>0$ for all $t\geq -\bar{\tau}$. These examples shows that the assumption $a>b$ cannot be relaxed if solutions of both \eqref{T1.c} and \eqref{T1.d} are to tend to zero for all initial conditions.
\subsection{General results}
We start by making some assumptions on $g$:
\begin{subequations}\label{eq:galt4}
\begin{gather}
\label{eq:g1alt4}g(0)=0; \\
\label{eq:g2alt4}
g(x)>0 \quad x>0; \\
\label{eq:g3alt4}\text{There is $\delta_1>0$ such that $g$ is increasing on $(0,\delta_1)$};\\
\label{eq:g4alt4}
\text{There is $\gamma\geq 1$ such that $g\circ G^{-1}\in \text{RV}_\infty(-\gamma)$}.
\end{gather}
\end{subequations}
We now state our main result for slowly growing (or bounded) delay.
\begin{theorem} \label{thm.odedecaygen}
Let $\tau$ be a continuous and non--negative function such that $-\overline{\tau}=\inf_{t\geq 0} t-\tau(t)$ and which obeys
\eqref{eq.tminustautoinfty}. Suppose also that there is $q\in[0,1)$ such that
\begin{equation}\label{eq.tauasytoqt}
\lim_{t\to\infty} \frac{\tau(t)}{t}=q.
\end{equation}
Suppose that $a>b>0$ in such a manner that
\begin{equation} \label{eq.abqsmallgamma}
a>b\left(\frac{1}{1-q}\right)^{\gamma}>0.
\end{equation}
Let $g$ satisfy \eqref{eq:galt4} and suppose $\psi\in
C([-\overline{\tau},0];(0,\infty))$. If the solution of
\eqref{T1.c} viz.,
\begin{align*}
    x'(t) &=  - ag(x(t)) + bg(x(t-\tau(t)),\quad t\geq 0 \\
    x(t) &=  \psi(t),\quad t \in [-\bar{\tau},0]
\end{align*}
obeys $x(t)\to 0$ as $t\to\infty$, and $G$ is defined by \eqref{def.G} then
\begin{equation}  \label{eq.odedecaygen}
0<\liminf_{t\to\infty} \frac{G(x(t))}{t} \leq \limsup_{t\to\infty} \frac{G(x(t))}{t}<+\infty.
\end{equation}
\end{theorem}
Our next general result deals with the case when $q\in(0,1)$ is so large that it does not satisfy \eqref{eq.abqbigsuperhyp}.
We modify the hypotheses on $g$ slightly in this case:
\begin{subequations}\label{eq:galt6}
\begin{gather}
\label{eq:g1alt6}g(0)=0; \\
\label{eq:g2alt6}
g(x)>0 \quad x>0; \\
\label{eq:g3alt6}\text{There is $\delta_1>0$ such that $g\in C^1(0,\delta_1)$, with $g'(x)>0$ for $x\in(0,\delta_1)$};\\
\label{eq:g4alt6}
\text{There is $\gamma\geq 1$ such that $g'\circ g^{-1}\in \text{RV}_0(1/\gamma)$}.
\end{gather}
\end{subequations}
It turns out that the hypothesis \eqref{eq:g4alt6} often implies \eqref{eq:g4alt4}.
\begin{theorem}  \label{thm3.1superhypq}
Let $\tau$ be a continuous and non--negative function such that $-\overline{\tau}=\inf_{t\geq 0} t-\tau(t)$ and which obeys
\eqref{eq.tminustautoinfty}. Suppose also that $\tau$ obeys \eqref{eq.tauasytoqt} for some $q\in(0,1)$, that $a>b>0$ and moreover that
\begin{equation} \label{eq.abqbigsuperhyp}
a<b\left(\frac{1}{1-q}\right)^\gamma.
\end{equation}
Suppose $g$ satisfies \eqref{eq:galt6} and suppose $\psi\in
C([-\overline{\tau},0];(0,\infty))$. If the solution $x$ of
\eqref{T1.c} obeys $x(t)\to 0$ as $t\to\infty$ then
\begin{equation}  \label{eq.pantodecaysuperhyp}
\lim_{t\to\infty} \frac{\log g(x(t))}{\log t} = -\frac{1}{\log(1/(1-q))}\log\left(\frac{a}{b}\right).
\end{equation}
\end{theorem}
To deal with rapidly growing delay, the following result is employed; it uses the same hypotheses on $g$ as Theorem~\ref{thm3.1superhypq}.
\begin{theorem}  \label{thm3.1gen}
Let $\tau$ be a continuous and non--negative function such that $-\overline{\tau}=\inf_{t\geq 0} t-\tau(t)$ and which obeys
\eqref{eq.tminustautoinfty}. Suppose also
\begin{gather}  \label{eq.t1}
\text{$\sigma$ is a non-negative, continuous function on
$[-\overline{\tau},\infty)$},\\
\label{eq.t2}
\lim_{t\to\infty} \int_0^t \frac{1}{\sigma(s)}\,ds =\infty,
\quad\lim_{t\to\infty}\sigma(t)=\infty, \\
\label{eq.t3}
\lim_{t\to\infty} \int_{t-\tau(t)}^t \frac{1}{\sigma(s)}\,ds = 1, \\
\label{eq.t4}
\lim_{t\to\infty} \frac{\sigma(t)}{t}=\infty.
\end{gather}
Let $a>b>0$ and $g$ satisfy \eqref{eq:galt6} and suppose $\psi\in
C([-\overline{\tau},0];(0,\infty))$. If the solution $x$ of \eqref{T1.c} obeys $x(t)\to 0$ as $t\to\infty$ , then
\begin{equation}  \label{Cgen}
\lim_{t\to\infty} \frac{\log g(x(t))}{\int_0^t
\frac{1}{\sigma(s)}\,ds} = -\log\left(\frac{a}{b}\right).
\end{equation}
\end{theorem}
\section{Slowly Growing and Proportional Delay for Equations with Regularly Varying Coefficient}
\subsection{Slowly growing delay}
Apart from the positivity of $g$, which guarantees positive solutions, we require conditions on $g$ local to the equilibrium $0$ in order to determine the rate of convergence of solutions in the case when the delay grows sublinearly.

Theorem~\ref{thm.odedecaygen} can be applied to equations with coefficients in $\text{RV}_0(\beta)$; in the first instance we consider the case where $\tau(t)/t\to 0$ as $t\to\infty$. The hypotheses on $g$ become:
\begin{subequations}\label{eq:galt2}
\begin{gather}
\label{eq:g1alt2}g(0)=0; \\
\label{eq:g2alt2}
g(x)>0 \quad x>0; \\
\label{eq:g3alt2}\text{There is $\delta_1>0$ such that $g$ is increasing on $(0,\delta_1)$};\\
\label{eq:g4alt2}
\text{There is $\beta>1$ such that }
g\in \text{RV}_0(\beta). 
\end{gather}
\end{subequations}
We now state our main result for slowly growing (or bounded) delay.
\begin{theorem} \label{thm.odedecay}
Let $\tau$ be a continuous and non--negative function such that $-\overline{\tau}=\inf_{t\geq 0} t-\tau(t)$ and which obeys
\eqref{eq.tminustautoinfty}. Suppose also
\begin{equation}\label{eq.tauttto0}
\lim_{t\to\infty} \frac{\tau(t)}{t}=0.
\end{equation}
Let $a>b>0$ and $g$ satisfy \eqref{eq:galt2} and suppose $\psi\in
C([-\overline{\tau},0];(0,\infty))$. If the solution of
\eqref{T1.c} viz.,
\begin{align*}
    x'(t) &=  - ag(x(t)) + bg(x(t-\tau(t)),\quad t\geq 0 \\
    x(t) &=  \psi(t),\quad t \in [-\bar{\tau},0]
\end{align*}
obeys $x(t)\to 0$ as $t\to\infty$, and $G$ is defined by \eqref{def.G} then
\begin{equation}  \label{eq.odedecay}
\lim_{t\to\infty} \frac{G(x(t))}{t} = a-b.
\end{equation}
Moreover
\begin{equation} \label{eq.odedecayuserv}
\lim_{t\to\infty} \frac{x(t)}{G^{-1}(t)}=(a-b)^{-1/(\beta-1)}.
\end{equation}
\end{theorem}
The limit in \eqref{eq.odedecayuserv} is a direct consequence of the fact that $G^{-1}\in \text{RV}_\infty(-1/(\beta-1))$ and 
\eqref{eq.odedecay}.

The result shows that when the delay $\tau$ grows sublinearly (or is bounded), converging solutions of \eqref{T1.c} have the same asymptotic
behaviour as the non--delay differential equation
\begin{equation} \label{eq.ode}
y'(t)=-(a-b)g(y(t)), \quad t>0; \quad y(0)=x_0>0,
\end{equation}
because $y(t)\to 0$ as $t\to\infty$ and the hypothesis $g\in\text{RV}_0(\beta)$ implies
\begin{equation} \label{eq.odeasymptotics}
\lim_{t\to\infty} \frac{y(t)}{G^{-1}(t)}=(a-b)^{-1/(\beta-1)}.
\end{equation}
Therefore, if $y$ is the solution of \eqref{eq.ode} we have
\[
\lim_{t\to\infty} \frac{x(t)}{y(t)}=1.
\]
\subsection{Proportional delay and asymptotic behaviour equivalent to non--delay case}
If the delay grows proportionately to $t$ in the sense that \eqref{eq.tauasytoqt} holds for some $q\in(0,1)$
the rate of decay of \eqref{T1.c} is not the same as \eqref{eq.ode}. We can prove the following result.

\begin{theorem}  \label{thm.tauproptnotlikeode}
Let $\tau$ be a continuous and non--negative function such that $-\overline{\tau}=\inf_{t\geq 0} t-\tau(t)$ and which obeys
\eqref{eq.tminustautoinfty}. Suppose also that $\tau$ obeys \eqref{eq.tauasytoqt}.
Let $a>b>0$ and $g$ satisfy \eqref{eq:g} and suppose $\psi\in
C([-\overline{\tau},0];(0,\infty))$. If the solution $x$ of
\eqref{T1.c} obeys $x(t)\to 0$ as $t\to\infty$, and $G$ is defined by \eqref{def.G}, then $x$ does not obey
\eqref{eq.odedecay}.
\end{theorem}

However if $q$ is sufficiently small, it can be shown that the main asymptotic behaviour of the differential equation \eqref{eq.ode} is preserved, in the sense that $x(t)$ is bounded above and below by $G^{-1}(t)$ times a constant as $t\to\infty$.

\begin{theorem} \label{thm.odedecay2}
Let $\tau$ be a continuous and non--negative function such that $-\overline{\tau}=\inf_{t\geq 0} t-\tau(t)$ and which obeys
\eqref{eq.tminustautoinfty}. Suppose also that $\tau$ obeys \eqref{eq.tauasytoqt} for some $q\in(0,1)$, that $a>b>0$ and moreover that
$a$ and $b$ obey
\begin{equation} \label{eq.abqsmall}
a>b\left(\frac{1}{1-q}\right)^{\beta/(\beta-1)}>0.
\end{equation}
Define $\Lambda>0$ by
\begin{equation}  \label{def.Lambdalim}
\Lambda=a-b\left(\frac{1}{1-q}\right)^{-1/(\beta-1)}.
\end{equation}
Suppose $g$ satisfies \eqref{eq:galt2} and suppose $\psi\in
C([-\overline{\tau},0];(0,\infty))$. If the solution $x$ of
\eqref{T1.c} obeys $x(t)\to 0$ as $t\to\infty$, and $G$ is defined by \eqref{def.G} then there is $\Lambda_0>0$
such that
\begin{equation}  \label{eq.odedecay2}
0<\Lambda_0\leq \liminf_{t\to\infty} \frac{G(x(t))}{t} \leq \limsup_{t\to\infty} \frac{G(x(t))}{t} \leq \Lambda^{-(\beta-1)}.
\end{equation}
Moreover
\[
\Lambda\leq \liminf_{t\to\infty} \frac{x(t)}{G^{-1}(t)}\leq \limsup_{t\to\infty} \frac{x(t)}{G^{-1}(t)}<+\infty.
\]
\end{theorem}
Recalling that the solution $y$ of the non--delay differential equation \eqref{eq.ode} obeys \eqref{eq.odeasymptotics}, Theorem~\ref{thm.odedecay2} shows that
\[
1<\liminf_{t\to\infty} \frac{x(t)}{y(t)}\leq \limsup_{t\to\infty}\frac{x(t)}{y(t)}<+\infty,
\]
as claimed.

We conjecture when $a$, $b$, $q$ and $\beta$ obey \eqref{eq.abqsmall}, and $\tau$ obeys \eqref{eq.tauasytoqt} that we can strengthen the conclusion of Theorem~\ref{thm.odedecay2} to obtain the limit
\begin{equation}  \label{eq.xpropdelayasy1}
\lim_{t\to\infty} \frac{x(t)}{G^{-1}(t)}= \left(a-b\left(\frac{1}{1-q}\right)^{\beta/(\beta-1)}\right)^{-1/(\beta-1)}=\Lambda,
\end{equation}
where $\Lambda$ is defined by \eqref{def.Lambdalim}.

In fact, by the methods of Theorem~\ref{thm.tauproptnotlikeode} it can be shown that if there is a $\lambda$ such that
\begin{equation} \label{eq.xpropdelayasy2}
\lim_{t\to\infty} \frac{x(t)}{G^{-1}(t)}=:\lambda\in(0,\infty),
\end{equation}
then we must have $\lambda=\Lambda$.

On the other hand, if $\tau$ obeys \eqref{eq.tauasytoqt} and $a$, $b$, $q$ and $\beta$ obey
\begin{equation} \label{eq.abqbigpre}
a<b\left(\frac{1}{1-q}\right)^{\beta/(\beta-1)},
\end{equation}
(with $a>b>0$) the method of proof of Theorem~\ref{thm.tauproptnotlikeode} shows that there is no $\lambda\in(0,\infty)$ such that
$x$ obeys \eqref{eq.xpropdelayasy2}. In the next section we investigate the case covered by \eqref{eq.abqbigpre} as well as the case
when the delay grows so quickly that $\tau(t)/t\to 1$ as $t\to\infty$.
\subsection{Proportional delay and asymptotic behaviour not equivalent to non--delay case}
Our next results demonstrates that once $\tau$ grows faster that $qt$ (where $q\in(0,1)$ is so large that it obeys \eqref{eq.abqbigpre}), the asymptotic behaviour of \eqref{T1.c} is no longer asymptotic to or bounded by the solution $y$ of the ordinary differential equation \eqref{eq.ode}. The exact rate of convergence can be determined in the case when $\tau$ obeys \eqref{eq.tauasytoqt} when $q\in(0,1)$ obeys \eqref{eq.abqbigpre}. Of course, the nonlinearity $g$ and the constants $a$ and $b$ still play an important role in determining
the asymptotic behaviour.

For these results, we place slightly different hypotheses on $g$ local to zero than the conditions \eqref{eq:galt2} imposed in Theorem~\ref{thm.odedecay} or \ref{thm.odedecay2}; now we require $g$ to not only be increasing, but to have a positive derivative
close to zero, and we ask that $g'$, rather than $g$, be regularly varying at $0$. The hypotheses are the following.
\begin{subequations}\label{eq:galt}
\begin{gather}
\label{eq:g1alt}g(0)=0; \\
\label{eq:g2alt}
g(x)>0 \quad x>0; \\
\label{eq:g3alt}\text{There is $\delta_1>0$ such that $g\in C^1(0,\delta_1)$, with $g'(x)>0$ for $x\in(0,\delta_1)$};\\
\label{eq:g4alt}
\text{There is $\beta>1$ such that }
g'\in \text{RV}_0(\beta-1). 
\end{gather}
\end{subequations}
We first deal with the case when $\tau$ obeys \eqref{eq.tauasytoqt} and $a$, $b$, $\beta$ and $q$ obey \eqref{eq.abqbig}.
In this case we can show that $x$ cannot be in $\text{RV}_\infty(-1/(\beta-1))$.
\begin{theorem}  \label{thm3.1q}
Let $\tau$ be a continuous and non--negative function such that $-\overline{\tau}=\inf_{t\geq 0} t-\tau(t)$ and which obeys
\eqref{eq.tminustautoinfty}. Suppose also that $\tau$ obeys \eqref{eq.tauasytoqt} for some $q\in(0,1)$, that $a>b>0$ and moreover that
\begin{equation} \label{eq.abqbig}
a<b\left(\frac{1}{1-q}\right)^{\beta/(\beta-1)}.
\end{equation}
Suppose $g$ satisfies \eqref{eq:galt} and suppose $\psi\in
C([-\overline{\tau},0];(0,\infty))$. If the solution $x$ of
\eqref{T1.c} viz.,
\begin{align*}
    x'(t) &=  - ag(x(t)) + bg(x(t-\tau(t)),\quad t\geq 0 \\
    x(t) &=  \psi(t),\quad t \in [-\bar{\tau},0]
\end{align*}
obeys $x(t)\to 0$ as $t\to\infty$ then
\begin{equation}  \label{eq.pantodecay}
\lim_{t\to\infty} \frac{\log x(t)}{\log t} = -\frac{1}{\beta}\frac{1}{\log(1/(1-q))}\log\left(\frac{a}{b}\right).
\end{equation}
\end{theorem}
It is a direct consequence of \eqref{eq.pantodecay}, \eqref{eq.abqbig}, \eqref{eq.odeasymptotics} and the fact that $G^{-1}\in \text{RV}_\infty(-1/(\beta-1))$ that
\[
\lim_{t\to\infty} \frac{x(t)}{y(t)}=\infty,
\]
where $y$ is the solution of \eqref{eq.ode}. Therefore, by Theorem~\ref{thm.odedecay}, \ref{thm.odedecay2} and \ref{thm3.1q}, once the delay grows sufficiently quickly, with the critical rate being
\[
\lim_{t\to\infty}\frac{\tau(t)}{t}=1-\left(\frac{a}{b}\right)^{-\beta/(\beta-1)},
\]
the solution of \eqref{T1.c} decays to zero more slowly than the solution $y$ of the non--delay equation \eqref{eq.ode}.

\section{Slowly Growing Delay for Equations with Regularly Varying Coefficient}
We now attempt to determine the asymptotic behaviour of solutions when the delay grows according to $\tau(t)/t\to 1$ as $t\to\infty$. It transpires
that the following theorem enables us to achieve this, provided a related limiting functional equation involving $\tau$ can be solved 
which involves an auxiliary function $\sigma$. The result follows by an application of Theorem~\ref{thm3.1gen}.

\begin{theorem}  \label{thm3.1}
Let $\tau$ be a continuous and non--negative function such that $-\overline{\tau}=\inf_{t\geq 0} t-\tau(t)$ and which obeys
\eqref{eq.tminustautoinfty}. Suppose also that $\sigma$ and $\tau$ obey \eqref{eq.t1}--\eqref{eq.t4}.
Let $a>b>0$ and $g$ satisfy \eqref{eq:galt} and suppose $\psi\in
C([-\overline{\tau},0];(0,\infty))$. If the solution of
\eqref{T1.c} viz.,
\begin{align*}
    x'(t) &=  - ag(x(t)) + bg(x(t-\tau(t)),\quad t\geq 0 \\
    x(t) &=  \psi(t),\quad t \in [-\bar{\tau},0]
\end{align*}
obeys $x(t)\to 0$ as $t\to\infty$ , then
\begin{equation}  \label{C}
\lim_{t\to\infty} \frac{\log g(x(t))}{\int_0^t
\frac{1}{\sigma(s)}\,ds} = -\log\left(\frac{a}{b}\right).
\end{equation}
Moreover, \eqref{C} is equivalent to
\begin{equation} \label{eq.C1}
\lim_{t\to\infty} \frac{\log x(t)}{\int_0^t
\frac{1}{\sigma(s)}\,ds} = -\frac{1}{\beta}\log\left(\frac{a}{b}\right).
\end{equation}
\end{theorem}
\eqref{eq.C1} is a direct consequence of \eqref{eq:g4} and \eqref{C}. \eqref{eq:g4} implies
\[
\lim_{x\to 0^+} \frac{\log g(x)}{\log x}=\beta.
\]
Then
\[
\lim_{t\to\infty} \frac{\log x(t)}{\int_0^t \frac{1}{\sigma(s)}\,ds}
=
\lim_{t\to\infty} \frac{\log g(x(t))}{\int_0^t
\frac{1}{\sigma(s)}\,ds}
\cdot \frac{\log x(t)}{\log g(x(t))}
= -\frac{1}{\beta}\log\left(\frac{a}{b}\right).
\]
The hypotheses on the auxiliary function $\sigma$ under which Theorem~\ref{thm3.1} holds will be explored and motivated in greater depth in the next section. Note however, that the conditions on the size of $\sigma$ and $\tau$ are asymptotic: the short run behaviour of $\tau$ and $\sigma$ is seen not to be important in being able to determine the rate of convergence. Neither are differentiability or monotonicity conditions required on $\tau$.
This feature of Theorem~\ref{thm3.1}, allow analysis to be extended to delay--differential equations with relatively badly behaved 
$\tau$. All that turns out to be important is the asymptotic rate of growth of $\tau$.

The presence of unbounded delay has just been mentioned, but it is not explicitly present in the statement of Theorem~\ref{thm3.1}.
However, the conditions \eqref{eq.t2} and \eqref{eq.t3} on $\sigma$, together with \eqref{eq.tminustautoinfty}, force
\begin{equation} \label{eq.tautoinfty}
\lim_{t\to\infty} \tau(t)=+\infty.
\end{equation}
Therefore, by also assuming \eqref{eq.tminustautoinfty} in Theorem~\ref{thm3.1}, the delay will be unbounded even though this is not explicitly stated. We have already noted that \eqref{eq.tminustautoinfty} is a reasonable assumption if we want $x(t)\to 0$ as $t\to\infty$.

To show that \eqref{eq.tautoinfty} must hold, first note that as \eqref{eq.t3} holds, there exists $T_1>0$ such that for all $t>T_1$ we have
\[
\int_{t-\tau(t)}^t \frac{1}{\sigma(s)}\,ds > \frac{1}{2}.
\]
Since $\sigma(t)\to\infty$ as $t\to\infty$, for every $M>0$ there exists $T_2(M)>0$ such that $\sigma(t)>M$ for all $t>T_2(M)$. Also
as $t-\tau(t)\to\infty$ as $t\to\infty$, there exists $T_3(M)>T_2(M)$ such that $t-\tau(t)>T_2(M)$ for all $t>T_3(M)$. Finally, let $T(M)=\max(T_1,T_2(M),T_3(M))$. Then for $t>T(M)$ we have
\[
\frac{1}{2}<\int_{t-\tau(t)}^t \frac{1}{\sigma(s)}\,ds <\int_{t-\tau(t)}^t \frac{1}{M}\,ds = \frac{\tau(t)}{M}.
\]
Hence $\tau(t)>M/2$ for $t>T(M)$. Since $M>0$ is arbitrary we have \eqref{eq.tautoinfty}.

Therefore, by the hypotheses in Theorem~\ref{thm3.1}, the delay is unbounded even though this is not explicitly stated. We have already noted that \eqref{eq.tminustautoinfty} is a reasonable assumption if we want $x(t)\to 0$ as $t\to\infty$.

We next show that Theorem~\ref{thm3.1} covers precisely the rapidly growing delay which is not covered by Theorems~\ref{thm.odedecay}, \ref{thm.odedecay2} and \ref{thm3.1q} which cover the case when $\tau(t)/t\to q\in[0,1)$ as $t\to\infty$. The question now is: how does the
condition \eqref{eq.t4} relate to the case not already covered by the results to date, namely the case when $\tau(t)/t\to1$ as $t\to\infty$.
Roughly speaking, we will now show that if the delay grows like $t$, then solutions grow at the rate determined by \eqref{eq.C1}. To do this,
we state an auxiliary result which shows how the linear or sublinear growth of $\sigma$ implies linear of sublinear growth in $\tau$.
\begin{lemma}  \label{lemma.sigtauasy}
Suppose $\tau$ is a non--negative continuous function which obeys \eqref{eq.tminustautoinfty}. Suppose $\sigma$ obeys \eqref{eq.t3} and
\begin{equation}  \label{eq.sigmatlambda}
\lim_{t\to\infty} \frac{\sigma(t)}{t}=\lambda\in[0,\infty].
\end{equation}
Then
\begin{equation} \label{eq.tauasytot}
\lim_{t\to\infty} \frac{\tau(t)}{t}=1-e^{-\lambda}.
\end{equation}
\end{lemma}
Therefore condition \eqref{eq.t4} implies that $\tau(t)/t\to 1$ as $t\to\infty$. Also the condition that $\sigma(t)/t\to 0$
as $t\to\infty$ implies that $\tau(t)/t\to 0$ as $t\to\infty$ and if $\sigma(t)/t$ tends to a finite limit as $t\to\infty$, then
$\tau(t)/t\to q$ as $t\to\infty$ for some $q\in(0,1)$.

\subsection{Concrete examples of $\sigma$ obeying \eqref{eq.t1}--\eqref{eq.t4}}
We now state general results which enable to explicitly construct $\sigma$ obeying \eqref{eq.t1}--\eqref{eq.t4} while at the same 
time only making assumptions concerning the asymptotic behaviour of $\tau$. 
\begin{proposition} \label{prop.tminustauinRVgamma}
Let $\tau$ be a continuous and non--negative function such that $-\overline{\tau}=\inf_{t\geq 0} t-\tau(t)$ and for which
\begin{equation}  \label{eq.tauminustasytgamma}
\text{There exists $\beta\in(0,1)$ such that }
\lim_{t\to\infty} \frac{\log(t-\tau(t))}{\log t}=\beta.
\end{equation}
Then there is a function $\sigma$ which obeys \eqref{eq.t1}, \eqref{eq.t2} and \eqref{eq.t3} such that
\begin{equation} 
\label{eq.int1sigasylogtgamma}
\lim_{t\to\infty} \frac{\int_0^t \frac{1}{\sigma(s)}\,ds}{\log_2 t} = \frac{1}{\log(1/\beta)}.
\end{equation}
\end{proposition}
\begin{proposition} \label{prop.tminustauinRV1}
Let $\tau$ be a continuous and non--negative function such that $-\overline{\tau}=\inf_{t\geq 0} t-\tau(t)$.
Suppose that $\varphi\in C[0,\infty);\mathbb{R})$ is such that
\begin{subequations}  \label{eq.phiconds}
\begin{gather}  \label{eq.phiinctoinfty}
\text{$\varphi$ is increasing on $[0,\infty)$ and $\lim_{t\to\infty} \varphi(t)=\infty$;}\\
\label{eq.phirv}
\varphi \in \text{RV}_\infty(0); \\
\label{eq.logexpphirv}
x\mapsto \log \varphi(e^x)\in \text{RV}_\infty(0),
\end{gather}
\end{subequations}
and
\begin{equation}  \label{eq.tauminustasytphi}
\lim_{t\to\infty} \frac{t-\tau(t)}{t/\varphi(t)}=1.
\end{equation}
Then there is a function $\sigma$ which obeys \eqref{eq.t1}, \eqref{eq.t2} and \eqref{eq.t3} such that
\begin{gather} \label{eq.sigasytlogphi}
\lim_{t\to\infty} \frac{\sigma(t)}{t\log \varphi(t)}=1,\\
\label{eq.int1sigasyphi}
\lim_{t\to\infty} \frac{\int_0^t \frac{1}{\sigma(s)}\,ds}{\log t/\log \varphi(t)} = 1.
\end{gather}
\end{proposition}
\subsection{Factors influencing the rate of decay of $x$}
We note that the relationship between the rate of growth of the unbounded delay $\tau$ and the rate of decay of the solution $x$ of \eqref{T1.c} to $0$ as $t\to\infty$ (which depends on $\sigma$) is embodied in the condition \eqref{eq.t3}. The limit  \eqref{eq.t3} relates the asymptotic behaviour of $\sigma$ to that of $\tau$. We see that the faster that $\sigma(t)\to\infty$ as $t\to\infty$, the faster that $1/\sigma(t)\to 0$, so in order for \eqref{eq.t3} to hold, $\tau(t)$ must tend to infinity faster as $t\to\infty$ to compensate for the rapid decay of $1/\sigma(t)$. Also, the faster that $\sigma$ tends to infinity, the slower that $\int_0^t 1/\sigma(s)\,ds$ tends so infinity as $t\to\infty$, and so by \eqref{C}, the slower that $x(t)\to 0$ as $t\to\infty$. Therefore, we see that the faster that $\tau(t)\to\infty$ as $t\to\infty$, the slower that $x(t)\to\infty$ as $t\to\infty$. This makes intuitive sense, as the longer the ``memory'' of the equation, the slower the convergence of asymptotically stable solutions to the equilibrium should be.

The limit \eqref{C} reveals that the rate of decay of $x(t)\to 0$ increases as $a$ increases and decreases as $b$ increases,
as should be expected; the greater the negative instantaneous feedback and the less the positive delayed feedback of the delayed term, the more
rapidly solutions of \eqref{T1.c} should converge to zero.

The limit \eqref{C} also reveals that the stronger the nonlinearity $g$ local to zero, the faster the rate of convergence of $x(t)\to 0$ as $t\to\infty$. Consider the solutions
$x_1$ and $x_2$ of \eqref{T1.c} in the case when $g=g_1$ and $g=g_2$ respectively. By \eqref{C}, we have
\[
\lim_{t\to\infty}\frac{\log g_1(x_1(t))}{\int_0^t
\frac{1}{\sigma(s)}\,ds}= -\log\left(\frac{a}{b}\right), \quad
\lim_{t\to\infty}\frac{\log g_2(x_2(t))}{\int_0^t
\frac{1}{\sigma(s)}\,ds}= -\log\left(\frac{a}{b}\right).
\]
Therefore
\begin{equation} \label{eq.gxtnosigma}
\lim_{t\to\infty} \frac{\log g_1(x_1(t))}{\log g_2(x_2(t))}=1.
\end{equation}
This limit is interesting in itself as it shows the impact of different nonlinearities on the rate of convergence of solutions, even when the auxiliary function $\sigma$ is not known.

To give a concrete example where we get different convergence rates arising from different nonlinearities, suppose that $g_1\in\text{RV}_0(\beta_1)$ and $g_2\in\text{RV}_\infty(\beta_2)$ where $\beta_2>\beta_1>1$. Then
\begin{equation} \label{eq.logggamma}
\lim_{x\to 0^+}
\frac{\log g_1(x)}{\log x}=\beta_1, \quad
\lim_{x\to 0^+}
\frac{\log g_2(x)}{\log x}=\beta_2.
\end{equation}
Then
\[
\lim_{x\to 0} \log\frac{g_2(x)}{g_1(x)}=\lim_{x\to 0^+} \log x\left(\frac{\log g_2(x)}{\log x}-\frac{\log g_1(x)}{\log x}\right)=
-\infty,
\]
so $g_2(x)/g_1(x)\to 0$ as $x\to 0^+$ and therefore $g_1$ dominates $g_2$ local to zero. We should therefore expect that $x_1$ tends to
zero more rapidly than $x_2$. By \eqref{eq.logggamma} and \eqref{eq.gxtnosigma} we have
\[
\lim_{t\to\infty} \frac{\log x_1(t)}{\log x_2(t)}
=
\lim_{t\to\infty} \frac{\log x_1(t)}{\log g_1(x_1(t))}\cdot \frac{\log g_1(x_1(t))}{\log g_2(x_2(t))} \cdot\frac{\log g_2(x_2(t))}{\log x_2(t)}
=\frac{\beta_2}{\beta_1}>1.
\]
Hence as $t\to\infty$, we have
\[
\log(x_1(t)/x_2(t))=\log x_1(t)-\log x_2(t)=\log x_2(t)\left(\frac{\log x_1(t)}{\log x_2(t)} -1 \right)\to -\infty,
\]
so $x_1(t)/x_2(t)\to 0$ as $t\to\infty$. Thus $x_1$ converges to zero more quickly to zero than $x_2$ as $t\to\infty$, as we anticipated.

\subsection{Motivation For Results}
In rough terms, Theorems~\ref{thm3.1} and~\ref{thm5.1} are proven by constructing one--parameter families of
functions $x_{L,\epsilon}$ and $x_{U,\epsilon}$ such that
\begin{subequations} \label{eq.motivatexuxl}
\begin{align}
g(x_{L,\epsilon}(t))&=x_1(\epsilon)\exp\left(-C_1(\epsilon)\int_0^t \frac{1}{\sigma(s)}\,ds\right),\\
g(x_{U,\epsilon}(t))&=x_2(\epsilon)\exp\left(-C_2(\epsilon)\int_0^t \frac{1}{\sigma(s)}\,ds\right).
\end{align}
\end{subequations}
where the monotonicity of $g$ ensures that the functions $x_{L,\epsilon}$ and $x_{U,\epsilon}$ are well--defined. These functions are constructed so that they are
upper and lower solutions of the solution $x$ of e.g., \eqref{T1.c}. This is achieved because $x_1(\epsilon)$, $x_2(\epsilon)$, $C_1(\epsilon)$
and $C_2(\epsilon)$ can be chosen so that there are $T_1(\epsilon), T_2(\epsilon)>0$ such that
\begin{align*}
x_{L,\epsilon}'(t)&<-ag(x_{L,\epsilon}(t))+bg(x_{L,\epsilon}(t-\tau(t))), \quad t>T_1(\epsilon), \\
x_{U,\epsilon}'(t)&>-ag(x_{U,\epsilon}(t))+bg(x_{U,\epsilon}(t-\tau(t))), \quad t>T_2(\epsilon).
\end{align*}
We choose $x_1(\epsilon)$ so small and $x_2(\epsilon)$ so large so that $x_{L,\epsilon}(t)<x(t)$ for $t\leq T_1(\epsilon)$ and $x_{U,\epsilon}(t)>x(t)$ for $t\leq T_2(\epsilon)$. The values of $x_1(\epsilon)$ and $x_2(\epsilon)$ play no role in the differential inequality. The parameters $C_1(\epsilon)$ and $C_2(\epsilon)$ are chosen so that the differential inequalities are satisfied on $[T_1(\epsilon),\infty)$ and $[T_2(\epsilon),\infty)$ respectively. The values of $T_1(\epsilon)$ and $T_2(\epsilon)$ are chosen so as to use asymptotic information about $\sigma$ and $\tau$ that is present in \eqref{eq.t2} and in \eqref{eq.t3} especially: this information is mainly used to satisfy the differential inequalities. The comparison principle now implies that $x_{L,\epsilon}(t)<x(t)<x_{U,\epsilon}(t)$ for all $t\geq 0$, and hence that $g(x_{L,\epsilon}(t))< g(x(t))<g(x_{U,\epsilon}(t))$ for all $t\geq 0$. The upper and lower estimates on $g(x(t))$ are known explicitly by the construction \eqref{eq.motivatexuxl}. Finally, we send the parameter $\epsilon\to 0^+$: the exact asymptotic limit \eqref{C} is obtained because $C_1(\epsilon)$ and $C_2(\epsilon)$ have been designed so that both tend to the same limit as $\epsilon\to 0^+$. Roughly speaking for each estimate, we need \textit{two} adjustable constants $x_i$ and $C_i$ to satisfy \textit{two} inequalities: one for the differential inequality on $(T_i(\epsilon),\infty)$ and one for the ``initial condition'' on $[-\bar{\tau},T_i(\epsilon)]$. The free parameter $\epsilon$ is used at the end of the proof to match exactly the upper and lower estimates. In fact, to give sufficient flexibility in the construction of the upper and lower estimates, we sometimes have additional free parameters $C_1$ and $C_2$ which can be sent to $C_1(\epsilon)$ and $C_2(\epsilon)$ in advance of taking the limit as $\epsilon\to 0^+$.

These are the broad guidelines followed in constructing the upper and lower estimates, and do not cover all the subtleties
encountered: sometimes the objectives are in conflict and the construction can become quite delicate and require some iteration. As a general rule, it
is more difficult to construct a very good upper estimate as there is some interaction between all three terms in the differential inequality for
$x_U$. For the lower estimate, the presence of the derivative term can generally be ignored by using the fact that the estimates constructed
are decreasing functions; therefore the relative size of the two terms on the righthand side of the differential inequality for $x_L$ is all that matters. The monotonicity of the estimates also allows the analysis to be extended easily to the equation \eqref{T1.d} with a
maximum functional, and simplifies the choice of estimates that must be taken in order to satisfy constraints on the ``initial conditions''.

We motivate now the functional forms of $x_L$ and $x_U$ and the hypotheses required in Theorems~\ref{thm3.1} and \ref{thm5.1}.
If a function $y$ is written in the form
\begin{equation} \label{eq.gyrep}
g(y(t))=g(y(0))\exp\left(-C\int_0^t \frac{1}{\sigma(s)}\,ds\right),
\end{equation}
as $g$ is in $C^1$, it is easily seen that
\begin{multline}\label{eq:heuristic}
y'(t) + ag(y(t)) -
bg(y(t-\tau(t))) \\
=y(0)e^{-C\int_0^t
\frac{1}{\sigma(s)}\,ds}\biggl\{
-\frac{1}{g'(y(t))}\frac{C}{\sigma(t)}+a
-be^{C\int_{t-\tau(t)}^t
\frac{1}{\sigma(s)}\,ds}\biggr\}
\end{multline}
with a similar equality holding when $y'(t) + ag(y(t)) -
b\sup_{t-\tau(t)\leq s\leq t}g(y(t-\tau(t)))$ is on the lefthand side. Therefore analysis of the righthand side of
\eqref{eq:heuristic} is the same whether we consider equation \eqref{T1.c} or \eqref{T1.d}.
A reasonable objective now is to
ensure that the term in curly braces in \eqref{eq:heuristic} is negligible (at least as $t\to\infty$) so that $y$ can be close to a solution of
\eqref{T1.c}.
In order that $y(t)$ tends to zero, we need $\int_0^t 1/\sigma(s)\,ds\to\infty$ as $t\to\infty$, while a
condition of the form $\sigma(t)\to\infty$ as $t\to\infty$ will preclude exponential decay. Moreover, as mentioned in the discussion after
Theorem~\ref{thm3.1}, the assumption that $\sigma(t)\to\infty$ as $t\to\infty$ is consistent with $\tau(t)\to\infty$ as $t\to\infty$.
This explains the rationale behind the conditions \eqref{eq.t2}. Moreover, if $y(t)\to0$ as $t\to\infty$,
then $g'(y(t))$ tends to a nontrivial limit by \eqref{eq:g4}, so as we suppose that $\sigma(t)\to\infty$ as $t\to\infty$,
the first term in the curly brackets in \eqref{eq:heuristic} tends to zero as $t\to\infty$. Therefore, in order for $y$ to be
in some sense ``close to'' a solution of \eqref{T1.c}, we need
a condition of the form
\[
\lim_{t\to\infty} a
-be^{C\int_{t-\tau(t)}^t
\frac{1}{\sigma(s)}\,ds}=0,
\]
which makes sense of the hypothesis \eqref{eq.t3} on $\sigma$, if we choose $C=\log(a/b)$. Now, if we take logarithms across \eqref{eq.gyrep}
and use $C=\log(a/b)$ and \eqref{eq.t3} we have
\[
\frac{\log g(y(t))}{\int_0^t \frac{1}{\sigma(s)}\,ds} =\frac{\log g(y(0))}{\int_0^t \frac{1}{\sigma(s)}\,ds}-\log\left(\frac{a}{b}\right).
\]
Taking limits as $t\to\infty$ gives
\[
\lim_{t\to\infty} \frac{\log g(y(t))}{\int_0^t \frac{1}{\sigma(s)}\,ds} = -\log\left(\frac{a}{b}\right).
\]
Since $y$ should be close to the solution $x$ of \eqref{T1.c}, this motivates the claimed result \eqref{C}, and therefore the construction of
$x_L$ and $x_U$ in \eqref{eq.motivatexuxl}.

Of course, this argument is a long way from being a rigorous proof; it however motivates the choice of conjecture, and an identity of
the form \eqref{eq:heuristic} in fact plays an important role in the proof of Theorems~\ref{thm3.1} and \ref{thm5.1}.

In some sense, our calculation leaves the functional form of $\sigma$ undetermined: it is left as an open question whether a function $\sigma$ exists which obeys the conditions \eqref{eq.t2} and \eqref{eq.t3} required in order to approximately fit $y$ as a solution. This leaves the question of
how to find such a function open. 
However, examples of equations whose asymptotic behaviour is determined by finding an appropriate $\sigma$ are given at the end of the section. More about the role of the function $\sigma$, and its connections with the solution of a class of functional equations (called Schr\"oder equations~\cite{Kuc}) is written in~\cite{AppBuck:07a}.

\section{Summary of Main Results and Examples in Regularly Varying Case}
\subsection{Unifying the main results}
Since the condition that $\sigma(t)/t\to 0$ as $t\to\infty$ implies that $\tau(t)/t\to 0$ as $t\to\infty$,
$\sigma(t)/t$ tends to a finite limit as $t\to\infty$ implies $\tau(t)/t\to q$ as $t\to\infty$ for some $q\in(0,1)$, and
\eqref{eq.t4} implies that $\tau(t)/t\to 1$ as $t\to\infty$, we can unify Theorems~\ref{thm.odedecay}, ~\ref{thm.odedecay2}, \ref{thm3.1q} and and~\ref{thm3.1} by means of the parameter $\lambda$ in \eqref{eq.sigmatlambda}.

In order to state a result which does this, we first unify the hypotheses \eqref{eq:galt} and \eqref{eq:galt2} on $g$ to give
\begin{subequations}\label{eq:galt3}
\begin{gather}
\label{eq:g1alt3}g(0)=0; \\
\label{eq:g2alt3}
g(x)>0 \quad x>0; \\
\label{eq:g3alt3}\text{There is $\delta_1>0$ such that $g\in C^1(0,\delta_1)$ with $g'(x)>0$ for $x\in (0,\delta_1)$};\\
\label{eq:g4alt3}
\text{There is $\beta>1$ such that }
g'\in \text{RV}_0(\beta-1). 
\end{gather}
\end{subequations}
\begin{theorem} \label{thm.unify}
Let $\tau$ be a continuous and non--negative function which obeys \eqref{eq.tminustautoinfty} such that $-\overline{\tau}=\inf_{t\geq 0} t-\tau(t)$.
Suppose that $\sigma$ and $\tau$ obey \eqref{eq.t1}--\eqref{eq.t3}, and that there exists
$\lambda\in[0,\infty]$ such that $\sigma$ obeys \eqref{eq.sigmatlambda}.
Let $a>b>0$ and $g$ satisfy \eqref{eq:galt3}, and suppose $\psi\in
C([-\overline{\tau},0];(0,\infty))$. Suppose also that the solution $x$ of
\eqref{T1.c} obeys $x(t)\to 0$ as $t\to\infty$.
\begin{itemize}
\item[(i)] If $\lambda=0$, and $G$ is defined by \eqref{def.G}, then $x$ obeys
\[
\lim_{t\to\infty} \frac{x(t)}{G^{-1}(t)}=(a-b)^{-1/(\beta-1)}.
\]
\item[(ii)] If $0<\lambda<\frac{\beta-1}{\beta}\log(a/b)$, $\Lambda$ is defined by \eqref{def.Lambdalim}, and $G$ is defined by \eqref{def.G}, then $x$ obeys
\[
\Lambda\leq \liminf_{t\to\infty} \frac{x(t)}{G^{-1}(t)}\leq \limsup_{t\to\infty} \frac{x(t)}{G^{-1}(t)}<+\infty.
\]
\item[(iii)] If $\frac{\beta-1}{\beta}\log(a/b)<\lambda<+\infty$, then $x$ obeys
\[
\lim_{t\to\infty} \frac{\log x(t)}{\log t} = -\frac{1}{\beta}\frac{1}{\lambda}\log\left(\frac{a}{b}\right).
\]
\item[(iv)] If $\lambda=\infty$, then $x$ obeys
\begin{equation*}
\lim_{t\to\infty} \frac{\log x(t)}{\int_0^t
\frac{1}{\sigma(s)}\,ds} = -\frac{1}{\beta}\log\left(\frac{a}{b}\right).
\end{equation*}
\end{itemize}
\end{theorem}
\subsection{Examples}  \label{sec.examp}
In the following section, we show the versatility of the results
in the last section, by considering a selection of examples with
with different rates of growth in the delay $\tau$. We state the
results for each example in turn.

In order to do this, we must determine how the auxiliary function $\sigma$ can be chosen for a given problem.
The role of the function $\sigma$ is explained further in \cite{AppBuck:07a}.
In the problems here to which Theorem~\ref{thm3.1} or \ref{thm5.1} could be applied, we do not have that
$\tau$ is asymptotic to $\sigma$, as scrutiny of the calculations involved in
Examples~\ref{example:examp1} and~\ref{example:examp2} reveal.
The relationship between $\sigma$ and $\tau$ is nontrivial and must be determined for each problem by analysis of \eqref{eq.t3},  thereby justifying general theorems such as Theorem~\ref{thm3.1} and~\ref{thm5.1}. The introduction of the function $\sigma$ also enables us to remove monotonicity and
differentiability conditions on $\tau$ often required in the study of differential equations with delay. Moreover, the asymptotic form of the condition \eqref{eq.t3} shows that it is the behaviour of $\tau(t)$ as $t\to\infty$ that determines the asymptotic behaviour of solutions of \eqref{T1.c} and \eqref{T1.d}; the behaviour of $\tau(t)$ on any compact interval $[0,T]$ is not material.

In each case the common hypotheses are that the solution $x$ of
equation \eqref{T1.c} is studied, with $\tau$ a continuous and
non--negative function, where $-\overline{\tau}=\inf_{t\geq 0}
t-\tau(t)$ is finite and $t-\tau(t)\to\infty$ as $t\to\infty$. We have $a>b>0$, the initial function
$\psi\in C([-\overline{\tau},0];(0,\infty))$. We suppose that $x(t)\to 0$ as $t\to\infty$. All results stated here for
solutions of \eqref{T1.c} apply equally to the max--type equation \eqref{T1.d}.

We consider for concreteness two functions for $g$: let
$g_1:[0,\infty)\to[0,\infty):x\mapsto g_1(x)=x^\beta$
for $\beta>1$. Such a function obeys all the conditions in \eqref{eq:galt3}.
We can also consider the non--polynomial function $g_2$ such that $g_2:[0,\delta)\to[0,\infty)$ obeys $g_2(x)=x^\beta\log(1/x)$ for $0<x\leq \delta<1$
and $g_2(0)=0$. Such a function $g_2$ also obeys all hypotheses in \eqref{eq:galt3}.
Therefore we can determine the asymptotic behaviour of $x$ using Theorem~\ref{thm3.1} or Theorem~\ref{thm.odedecay} if $g=g_1$ or $g=g_2$.

In order to apply Theorem~\ref{thm.odedecay} it is first necessary to determine the asymptotic behaviour of $G_i^{-1}(x)$ as $x\to\infty$ for $i=1,2$.
To do this we compute $G_1$ and $G_2$. First we have that
\[
G_1(x)=\int_x^\delta \frac{1}{g_1(u)}\,du \sim \frac{1}{\beta-1}\frac{x}{g(x)}=\frac{1}{\beta-1} x^{1-\beta}, \quad x\to 0^+.
\]
Therefore
\[
\lim_{x\to \infty} \frac{G_1^{-1}(x)}{x^{1/(1-\beta)}}=\left(\frac{1}{\beta-1}\right)^{1/(\beta-1)}.
\]

For $G_2$ we see that
\[
G_2(x)=\int_{x}^{\delta} \frac{1}{u^\beta\log(1/u)}\,du
\sim \frac{1}{\beta-1} \frac{1}{x^{\beta-1}\log(1/x)} \quad \text{as $x\to 0^+$}.
\]
Hence
\[
\lim_{y\to\infty} y G_2^{-1}(y)^{\beta-1}\log(1/G^{-1}(y))
=
\lim_{x\to 0^+} \frac{G_2(x)}{\frac{1}{x^{\beta-1}\log(1/x)}}=\frac{1}{\beta-1}.
\]
From this it can be shown that $G_2^{-1}(x)\sim x^{-1/(\beta-1)} (\log x)^{-1/(\beta-1)}$ as $x\to\infty$.

We should note that it is not necessary that $g$ assume exactly the form of $g_1$ or $g_2$ above in order for us to determine asymptotic results.
Suppose merely that
\[
\lim_{x\to 0^+} \frac{g(x)}{g_i(x)}=1
\]
where $i=1,2$. Then $g\in \text{RV}_0(\beta)$, and moreover we can show that
\[
\lim_{x\to\infty} \frac{G^{-1}(x)}{G_i^{-1}(x)}=1.
\]
Therefore, if $g$ is a function which is positive on $(0,\infty)$, is increasing and continuously differentiable on an interval $(0,\delta_1)$, and obeys $g(0)=0$, we can still apply Theorem~\ref{thm.odedecay}.
\begin{example} \label{example:exampminus2}
Suppose that $\tau$ is bounded, or that $\tau(t)\to\infty$ but $\tau(t)/t\to 0$ as $t\to\infty$.
\begin{itemize}
\item[(i)] If
\[
\lim_{x\to 0^+} \frac{g(x)}{x^{\beta}}=1
\]
and $x(t)\to 0$ as $t\to\infty$ then
\[
\lim_{t\to\infty} t^{\frac{1}{\beta-1}} x(t) = \left(\frac{1}{(a-b)(\beta-1)}\right)^{1/\beta-1}.
\]
\item[(ii)] If
\[
\lim_{x\to 0^+} \frac{g(x)}{x^{\beta}\log(1/x)}=1
\]
and $x(t)\to 0$ as $t\to\infty$, then
\[
\lim_{t\to\infty} \frac{x(t)}{t^{-1/(\beta-1)} (\log t)^{-1/(\beta-1)}}=\left(\frac{1}{a-b}\right)^{1/\beta-1}.
\]
\end{itemize}
\end{example}
\begin{example} \label{example:exampminus1}
Suppose  $\tau(t)/t\to q\in(0,1)$ as $t\to\infty$ with $a>b(1-q)^{-\beta/(\beta-1)}$.
\begin{itemize}
\item[(i)] If
\[
\lim_{x\to 0^+} \frac{g(x)}{x^{\beta}}=1
\]
and $x(t)\to 0$ as $t\to\infty$ then
\[
0<\liminf_{t\to\infty} t^{\frac{1}{\beta-1}} x(t) \leq \limsup_{t\to\infty} t^{\frac{1}{\beta-1}} x(t) <+\infty.
\]
\item[(ii)] If
\[
\lim_{x\to 0^+} \frac{g(x)}{x^{\beta}\log(1/x)}=1
\]
and $x(t)\to 0$ as $t\to\infty$, then
\[
0<
\liminf_{t\to\infty} \frac{x(t)}{t^{-1/(\beta-1)} (\log t)^{-1/(\beta-1)}}
\leq
\limsup_{t\to\infty} \frac{x(t)}{t^{-1/(\beta-1)} (\log t)^{-1/(\beta-1)}} <+\infty.
\]
\end{itemize}
\end{example}
\begin{example} \label{example:examp3}
Suppose $\tau(t)/t\to q\in(0,1)$ as $t\to\infty$ with $a<b(1-q)^{-\beta/(\beta-1)}$. If $x(t)\to 0$ as $t\to\infty$, then
\[
\lim_{t\to\infty} \frac{\log x(t)}{\log t} =
-\frac{1}{\beta}\frac{1}{\log(1/(1-q))}\log\left(\frac{a}{b}\right).
\]
\end{example}
\begin{example}  \label{example:examp2}
Let $C>0$, $\gamma>0$,
\[
\lim_{t\to\infty} \frac{t-\tau(t)}{t/\log^\gamma t} = C.
\]
If $g(x)\sim g_1(x)$ or $g(x)\sim g_2(x)$ as $x\to 0^+$ and $x(t)\to 0$ as $t\to\infty$ then
\[
\lim_{t\to\infty} \frac{\log x(t)}{\log t/\log\log t} =
-\frac{1}{\beta}\frac{1}{\gamma}\log\left(\frac{a}{b}\right).
\]
\end{example}
\begin{example}  \label{example:examp1}
Let $C>0$, $\gamma\in(0,1)$,
\[
\lim_{t\to\infty} \frac{t-\tau(t)}{t^\gamma}=C.
\]
If $g(x)\sim g_1(x)$ or $g(x)\sim g_2(x)$ as $x\to 0^+$ and $x(t)\to 0$ as $t\to\infty$ then
\[
\lim_{t\to\infty} \frac{\log x(t)}{\log \log t} =
-\frac{1}{\beta}\frac{1}{\log(1/\gamma)}\log\left(\frac{a}{b}\right).
\]
\end{example}
%
%
%
As we move down through this list of examples, the rate of growth
of the delay becomes faster; the rate of growth of the auxiliary
function $\sigma$ will also be shown to be faster; and the rate of
decay of the solutions becomes slower. 
Example~\ref{example:examp3} deals with an equation with
(approximately) proportional delay. In this case our results are
consistent with more precise results which have been determined
for the so-called pantograph equation when $\beta=1$.
\section{Equations With Non--Regularly Varying Coefficients} \label{sec.nonrv}
We can apply theorem~\ref{thm.odedecaygen} and \ref{thm3.1superhypq} to equations where the function $g$ is not in $\text{RV}_0(\beta)$ for some
$\beta>1$. The results are as follows. We first consider hypotheses on $g$ which enable us to consider equations with slowly growing delay $\tau$:
\begin{subequations}\label{eq:galt5}
\begin{gather}
\label{eq:g1alt5}g(0)=0; \\
\label{eq:g2alt5}
g(x)>0 \quad x>0; \\
\label{eq:g3alt5}\text{There is $\delta_1>0$ such that $g$ is increasing on $(0,\delta_1)$};\\
\label{eq:g4alt5}
g\circ G^{-1}\in \text{RV}_\infty(-1), \quad G^{-1}\in \text{RV}_\infty(0).
\end{gather}
\end{subequations}
The hypothesis \eqref{eq:g4alt5} represents a strengthening of hypothesis \eqref{eq:g4alt4} in the case $\gamma=1$.
We may now apply Theorem~\ref{thm.odedecaygen} directly to the equation \eqref{T1.c}.
\begin{theorem} \label{thm.odedecaysuperhypsmallg}
Let $\tau$ be a continuous and non--negative function such that $-\overline{\tau}=\inf_{t\geq 0} t-\tau(t)$ and which obeys
\eqref{eq.tminustautoinfty}. Suppose also that there is $q\in[0,1)$ such that $\tau$ obeys \eqref{eq.tauasytoqt}.
Suppose that $a>b>0$ in such a manner that
\begin{equation} \label{eq.abqsmallsuperhyp}
a>b\frac{1}{1-q}>0.
\end{equation}
Let $g$ satisfy \eqref{eq:galt5} and suppose $\psi\in
C([-\overline{\tau},0];(0,\infty))$. If the solution $x$ of
\eqref{T1.c} obeys $x(t)\to 0$ as $t\to\infty$, and $G$ is defined by \eqref{def.G} then
\begin{equation}  \label{eq.odedecaysuperhyp}
\lim_{t\to\infty} \frac{x(t)}{G^{-1}(t)} =1.
\end{equation}
\end{theorem}
Therefore if $y$ is the solution of the corresponding non--delay equation \eqref{eq.ode}, we have that
\[
\lim_{t\to\infty} \frac{x(t)}{y(t)}=1.
\]
Our next result deals with the case when $q\in(0,1)$ is so large that it does not satisfy \eqref{eq.abqbigsuperhyp}.
The hypotheses on $g$ in this case are a special case of the hypotheses \eqref{eq:galt6}
\begin{subequations}\label{eq:galt7}
\begin{gather}
\label{eq:g1alt7}g(0)=0; \\
\label{eq:g2alt7}
g(x)>0 \quad x>0; \\
\label{eq:g3alt7}\text{There is $\delta_1>0$ such that $g\in C^1(0,\delta_1)$, with $g'(x)>0$ for $x\in(0,\delta_1)$};\\
\label{eq:g4alt7}
g'\circ g^{-1}\in \text{RV}_0(1).
\end{gather}
\end{subequations}
It turns out that the hypothesis \eqref{eq:g4alt7} often implies \eqref{eq:g4alt4}. The following is therefore a
simple application of Theorem~\ref{thm3.1superhypq}.
\begin{theorem}  \label{thm3.1superhypqgamma1}
Let $\tau$ be a continuous and non--negative function such that $-\overline{\tau}=\inf_{t\geq 0} t-\tau(t)$ and which obeys
\eqref{eq.tminustautoinfty}. Suppose also that $\tau$ obeys \eqref{eq.tauasytoqt} for some $q\in(0,1)$, that $a>b>0$ and moreover that
\begin{equation} \label{eq.abqbigsuperhypgamma1}
a<b\frac{1}{1-q}.
\end{equation}
Suppose $g$ satisfies \eqref{eq:galt7} and suppose $\psi\in
C([-\overline{\tau},0];(0,\infty))$. If the solution $x$ of
\eqref{T1.c} obeys $x(t)\to 0$ as $t\to\infty$ then
\begin{equation}  \label{eq.pantodecaysuperhypgamma1}
\lim_{t\to\infty} \frac{\log g(x(t))}{\log t} = -\frac{1}{\log(1/(1-q))}\log\left(\frac{a}{b}\right).
\end{equation}
\end{theorem}
Finally, we can apply Theorem~\ref{thm3.1gen} to get the following result in the case when the delay grows rapidly.
\begin{theorem}  \label{thm3.1nonrv}
Let $\tau$ be a continuous and non--negative function such that $-\overline{\tau}=\inf_{t\geq 0} t-\tau(t)$ and which obeys
\eqref{eq.tminustautoinfty}. Suppose also that $\sigma$ and $\tau$ obey \eqref{eq.t1}--\eqref{eq.t4}.
Let $a>b>0$ and $g$ satisfy \eqref{eq:galt7} and suppose $\psi\in
C([-\overline{\tau},0];(0,\infty))$. If the solution $x$ of
\eqref{T1.c} obeys $x(t)\to 0$ as $t\to\infty$ , then
\begin{equation}  \label{Cnonrv}
\lim_{t\to\infty} \frac{\log g(x(t))}{\int_0^t
\frac{1}{\sigma(s)}\,ds} = -\log\left(\frac{a}{b}\right).
\end{equation}
\end{theorem}
\subsection{Examples}
The difficult conditions to verify are \eqref{eq:g4alt5} (for Theorem~\ref{thm.odedecaysuperhypsmallg}) and
\eqref{eq:g4alt7} (for Theorem~\ref{thm3.1superhypqgamma1}). We now consider two examples of $g$ that are so flat at $0$ that 
the above results apply. However, this necessitates finding the asymptotic behaviour of rather complicated functions such as 
$g'\circ g^{-1}$ and $g\circ G^{-1}$. We record our findings in the following lemmata. 
\begin{lemma} \label{lemma.ex1GinvGammaasy}
Suppose that $\alpha>0$ and that $g$ obeys
\begin{equation} \label{eq.ex1nonrvgeeasy}
\lim_{x\to 0^+} \frac{g(x)}{e^{-1/x^\alpha}}=1.
\end{equation}
Then
\begin{gather} \label{eq.ex1nonrvGinvasy}
\lim_{y\to\infty} \frac{G^{-1}(y)}{(\log y)^{-1/\alpha}}=1, \quad G^{-1}\in \text{RV}_\infty(0),\\
\label{eq.ex1nonrvGammaasy}
\lim_{y\to\infty} \frac{\Gamma(y)}{\log^{-(\alpha+1)/\alpha}(y)/y}=\frac{1}{\alpha}, \quad \Gamma=g\circ G^{-1} \in \text{RV}_\infty(-1).
\end{gather}
\end{lemma}
\begin{lemma} \label{lemma.ex1GinvGamma1asy}
Suppose that $\alpha>0$ and that $g$ obeys
\begin{equation} \label{eq.ex1nonrvgeeprasy}
\lim_{x\to 0^+} \frac{g'(x)}{\alpha x^{-\alpha-1} e^{-1/x^\alpha}}=1.
\end{equation}
Then $\Gamma_1=g'\circ g^{-1}$ obeys
\begin{equation} \label{eq.ex1nonrvGamma1asy}
\lim_{y\to 0^+} \frac{\Gamma_1(y)}{\alpha \log(1/y)^{(\alpha+1)/\alpha}  y}=1, \quad \Gamma_1\in \text{RV}_0(1).
\end{equation}
\end{lemma}
\begin{lemma}  \label{lemma.ex2GinvGammaasy}
Suppose that $g$ obeys
\begin{equation} \label{eq.ex2nonrvgeeasy}
\lim_{x\to 0^+} \frac{g(x)}{\exp(-e^{1/x})}=1.
\end{equation}
Then
\begin{gather} \label{eq.ex2nonrvGinvasy}
\lim_{y\to\infty} \frac{G^{-1}(y)}{1/\log_2 y}=1, \quad G^{-1}\in \text{RV}_\infty(0),\\
\label{eq.ex2nonrvGammaasy}
\lim_{y\to\infty} \frac{\Gamma(y)}{1/(y\log y(\log_2 y)^2)}=1, \quad \Gamma=g\circ G^{-1} \in \text{RV}_\infty(-1).
\end{gather}
\end{lemma}
\begin{lemma} \label{lemma.ex2GinvGamma1asy}
Suppose that $g$ obeys
\begin{equation} \label{eq.ex2nonrvgeeprasy}
\lim_{x\to 0^+} \frac{g'(x)}{x^{-2}e^{1/x}\exp(-e^{1/x})}=1.
\end{equation}
Then $\Gamma_1=g'\circ g^{-1}$ obeys
\begin{equation} \label{eq.ex2nonrvGamma1asy}
\lim_{y\to 0^+} \frac{\Gamma_1(y)}{\log(1/y)\log_2(1/y)^2 y}=1, \quad \Gamma_1\in \text{RV}_0(1).
\end{equation}
\end{lemma}
With these asymptotic results to hand, we can obtain precise rates of convergence to $0$ of the solution of \eqref{T1.c} for the 
functions $g$ above. 
\begin{example} \label{example.asygeminuspoly}
Suppose that $a>b>0$, $\tau(t)/t\to q\in[0,1)$, that $g(0)=0$, $g$ is positive, $g$ is increasing on $(0,\delta_1)$, and $x(t)\to 0$ as $t\to\infty$
\begin{itemize}
\item[(a)] Suppose that $a>b(1-q)^{-1}$. If there is $\alpha>0$ such that
\[
\lim_{x\to 0^+} \frac{g(x)}{e^{-1/x^\alpha}}=1
\]
Then
\[
\lim_{t\to\infty} x(t)(\log t)^{1/\alpha}=1.
\]
\item[(b)] Suppose that $a<b(1-q)^{-1}$. If there is $\alpha>0$ such that
\begin{equation} \label{eq.exampgprnonrvasypoly}
\lim_{x\to 0^+} \frac{g'(x)}{\alpha x^{-(\alpha+1)}e^{-1/x^\alpha}}=1
\end{equation}
Then
\[
\lim_{t\to\infty}  x(t)(\log t)^{1/\alpha} = \left(\frac{\log(1/(1-q))}{\log(a/b)}\right)^{1/\alpha}.
\]
\item[(c)] Suppose that $q=1$ and that there is $C>0$ and $\gamma>0$ such that
\[
\lim_{t\to\infty} \frac{t-\tau(t)}{t/(\log t)^\gamma}=C.
\]
If $g$ obeys \eqref{eq.exampgprnonrvasypoly}, then
\[
\lim_{t\to\infty} x(t) (\log t/\log_2 t)^{1/\alpha} = \left(\frac{\gamma}{\log(a/b)}\right)^{1/\alpha}.
\]
\item[(d)] Suppose that $q=1$ and that there is $C>0$ and $\gamma\in(0,1)$ such that
\[
\lim_{t\to\infty} \frac{t-\tau(t)}{t^{\gamma}}=C.
\]
If $g$ obeys \eqref{eq.exampgprnonrvasypoly}, then
\[
\lim_{t\to\infty} x(t) (\log_2 t)^{1/\alpha} = \left(\frac{\log(1/\gamma)}{\log(a/b)}\right)^{1/\alpha}.
\]
\end{itemize}
\end{example}
\begin{example} \label{example.asygeminusexp}
Suppose that $a>b>0$, $\tau(t)/t\to q\in[0,1)$, that $g(0)=0$, $g$ is positive, $g$ is increasing on $(0,\delta_1)$, and $x(t)\to 0$ as $t\to\infty$
\begin{itemize}
\item[(a)] Suppose that $a>b(1-q)^{-1}$. If
\[
\lim_{x\to 0^+} \frac{g(x)}{\exp(-e^{1/x})}=1,
\]
then
\[
\lim_{t\to\infty} x(t)\log_2 t=1.
\]
\item[(b)] Suppose that $a<b(1-q)^{-1}$. If
\begin{equation} \label{eq.exampgprnonrvasyexp}
\lim_{x\to 0^+} \frac{g'(x)}{x^{-2}e^{1/x}\exp(-e^{1/x})}=1,
\end{equation}
then
\[
\lim_{t\to\infty}  x(t)\log_2 t = 1.
\]
\item[(c)] Suppose that $q=1$ and that there is $C>0$ and $\gamma>0$ such that
\[
\lim_{t\to\infty} \frac{t-\tau(t)}{t/(\log t)^\gamma}=C.
\]
If $g$ obeys \eqref{eq.exampgprnonrvasyexp}, then
\[
\lim_{t\to\infty} x(t) \log_2 t  = 1.
\]
\item[(d)] Suppose that $q=1$ and that there is $C>0$ and $\gamma\in(0,1)$ such that
\[
\lim_{t\to\infty} \frac{t-\tau(t)}{t^{\gamma}}=C.
\]
If $g$ obeys \eqref{eq.exampgprnonrvasyexp}, then
\[
\lim_{t\to\infty} x(t) \log_3 t  = 1.
\]
\end{itemize}
\end{example}

\section{Equations with Maximum Functionals} \label{sec.max}
We may also consider the asymptotic behaviour of solutions of the equation \eqref{T1.d}.
The proofs of the asymptotic results for this equation are the same as those for \eqref{T1.c} except at one stage of the proof. ,
This is because the functions used as upper and lower solutions in Theorem~\ref{thm.odedecay} and \ref{thm3.1} are also employed for \eqref{T1.d},
in the sense that the functional dependence of the comparison functions on the solution of the underlying functional differential equation and data are the same in both proofs (they are not the same functions because the solutions of \eqref{T1.c} and \eqref{T1.d} are not the same, in general).
These comparison functions are monotone decreasing and small on their domain of definition, and $g$ is assumed to be increasing
in some neighbourhood to the right of zero, so we may write
\[
\max_{t-\tau(t)\leq s\leq t} g(x_U(s))=g(x_U(t-\tau(t))), \quad t>T
\]
where $x_U$ is the upper comparison function, and $T>0$ is sufficiently large. The same identity holds for lower comparison functions.
Therefore the upper comparison function which satisfies the differential inequality
\[
x_U'(t)>-ag(x_U(t))+bg(x_{U}(t-\tau(t))), \quad t>T
\]
also satisfies
\[
x_U'(t)>-ag(x_U(t))+b\max_{t-\tau(t)\leq s\leq t}g(x_{U}(s)), \quad t>T,
\]
with an analogous pair of inequalities holding for the lower comparison functions.
The comparison principle now shows that these upper and lower comparison functions bound the solution above and below. Therefore we have a
direct analogue of Theorems~\ref{thm3.1}. If an analogue of Theorem~\ref{thm.odedecay} can be shown, then the analogue of 
Theorem~\ref{thm.unify}
follows directly. The relevant results are now stated.
\begin{theorem} \label{thm.odedecaymax}
Let $\tau$ be a continuous and non--negative function such that $-\overline{\tau}=\inf_{t\geq 0} t-\tau(t)$ and which obeys
\eqref{eq.tminustautoinfty}. Suppose also $\tau$ obeys \eqref{eq.tauttto0}.
Let $a>b>0$ and $g$ satisfy \eqref{eq:galt2} and suppose $\psi\in
C([-\overline{\tau},0];(0,\infty))$. If the solution of
\eqref{T1.d} viz.,
\begin{align*}
    x'(t) &=  - ag(x(t)) + b\max_{t-\tau(t)\leq s\leq t}g(x(s)),\quad t\geq 0 \\
    x(t) &=  \psi(t),\quad t \in [-\bar{\tau},0]
\end{align*}
obeys $x(t)\to 0$ as $t\to\infty$, and $G$ is defined by \eqref{def.G} then
\begin{equation}  \label{eq.odedecaymax}
\lim_{t\to\infty} \frac{G(x(t))}{t} = a-b.
\end{equation}
Moreover
\begin{equation*} 
\lim_{t\to\infty} \frac{x(t)}{G^{-1}(t)}=(a-b)^{-1/(\beta-1)}.
\end{equation*}
\end{theorem}
The comparison approach gives that $x$ obeys
\[
G_0(x(t))\leq at, \quad t\geq 0; \quad G_0(x(t))\geq \Lambda_1 t, \quad t\geq 1
\]
for some $\Lambda_1>0$, where $G_0(x)=\int_x^{x(0)} 1/g(u)\,du$. The proof that this implies \eqref{eq.odedecaymax} is given
in the final section.
\begin{theorem} \label{thm.odedecay2max}
Let $\tau$ be a continuous and non--negative function such that $-\overline{\tau}=\inf_{t\geq 0} t-\tau(t)$ and which obeys
\eqref{eq.tminustautoinfty}. Suppose also that $\tau$ obeys \eqref{eq.tauasytoqt} for some $q\in(0,1)$, that $a>b>0$ and moreover that
$a$, $b$, $\beta$ and $q$ obey \eqref{eq.abqsmall}. Let $\Lambda$ obey \eqref{def.Lambdalim}. Suppose $g$ satisfies \eqref{eq:galt2} and suppose $\psi\in
C([-\overline{\tau},0];(0,\infty))$. If the solution $x$ of
\eqref{T1.d} obeys $x(t)\to 0$ as $t\to\infty$, and $G$ is defined by \eqref{def.G} then there is $\Lambda_0>0$
such that
\begin{equation*} 
0<\Lambda_0\leq \liminf_{t\to\infty} \frac{G(x(t))}{t} \leq \limsup_{t\to\infty} \frac{G(x(t))}{t} \leq \Lambda^{-(\beta-1)}.
\end{equation*}
Moreover
\[
\Lambda\leq \liminf_{t\to\infty} \frac{x(t)}{G^{-1}(t)}\leq \limsup_{t\to\infty} \frac{x(t)}{G^{-1}(t)}<+\infty.
\]
\end{theorem}

\begin{theorem} \label{thm5.1q}
Let $\tau$ be a continuous and non--negative function such that $-\overline{\tau}=\inf_{t\geq 0} t-\tau(t)$ and which obeys
\eqref{eq.tminustautoinfty}. Suppose also that $\tau$ obeys \eqref{eq.tauasytoqt} for some $q\in(0,1)$, that $a>b>0$ and moreover that
$a$, $b$, $\beta$ and $q$ obey \eqref{eq.abqbig}.
Suppose $g$ satisfies \eqref{eq:galt} and suppose $\psi\in
C([-\overline{\tau},0];(0,\infty))$. If the solution $x$ of
\eqref{T1.d}
obeys $x(t)\to 0$ as $t\to\infty$ then
\begin{equation*}  
\lim_{t\to\infty} \frac{\log x(t)}{\log t} = -\frac{1}{\beta}\frac{1}{\log(1/(1-q))}\log\left(\frac{a}{b}\right).
\end{equation*}
\end{theorem}

\begin{theorem}  \label{thm5.1}
Let $\tau$ be a continuous and non--negative function such that $-\overline{\tau}=\inf_{t\geq 0} t-\tau(t)$ and which obeys
\eqref{eq.tminustautoinfty}. Suppose that $\sigma$ and $\tau$ obey \eqref{eq.t1}--\eqref{eq.t4}.
Let $a>b>0$ and $g$ satisfy \eqref{eq:galt} and suppose $\psi\in
C([-\overline{\tau},0];(0,\infty))$. If the solution of
\eqref{T1.d}
obeys $x(t)\to 0$ as $t\to\infty$ , then
\begin{equation}  \label{Cm}
\lim_{t\to\infty} \frac{\log g(x(t))}{\int_0^t
\frac{1}{\sigma(s)}\,ds} = -\log\left(\frac{a}{b}\right).
\end{equation}
Moreover, \eqref{Cm} is equivalent to
\begin{equation*} 
\lim_{t\to\infty} \frac{\log x(t)}{\int_0^t
\frac{1}{\sigma(s)}\,ds} = -\frac{1}{\beta}\log\left(\frac{a}{b}\right).
\end{equation*}
\end{theorem}
These results can be unified, just as we had for \eqref{T1.c} in Theorem~\ref{thm.unify}.
\begin{theorem} \label{thm.unifymax}
Let $\tau$ be a continuous and non--negative function which obeys \eqref{eq.tminustautoinfty} such that $-\overline{\tau}=\inf_{t\geq 0} t-\tau(t)$.
Suppose that $\sigma$ and $\tau$ obey \eqref{eq.t1}--\eqref{eq.t3}, and that there exists
$\lambda\in[0,\infty]$ such that $\sigma$ obeys \eqref{eq.sigmatlambda}.
Let $a>b>0$ and $g$ satisfy \eqref{eq:galt3}, and suppose $\psi\in
C([-\overline{\tau},0];(0,\infty))$. Suppose also that the solution $x$ of
\eqref{T1.d} obeys $x(t)\to 0$ as $t\to\infty$.
\begin{itemize}
\item[(i)] If $\lambda=0$, and $G$ is defined by \eqref{def.G}, then $x$ obeys
\[
\lim_{t\to\infty} \frac{x(t)}{G^{-1}(t)}=(a-b)^{-1/(\beta-1)}.
\]
\item[(ii)] If $0<\lambda<\frac{\beta-1}{\beta}\log(a/b)$, $\Lambda$ is defined by \eqref{def.Lambdalim}, and $G$ is defined by \eqref{def.G}, then $x$ obeys
\[
\Lambda\leq \liminf_{t\to\infty} \frac{x(t)}{G^{-1}(t)}\leq \limsup_{t\to\infty} \frac{x(t)}{G^{-1}(t)}<+\infty.
\]
\item[(iii)] If $\frac{\beta-1}{\beta}\log(a/b)<\lambda<+\infty$, then $x$ obeys
\[
\lim_{t\to\infty} \frac{\log x(t)}{\log t} = -\frac{1}{\beta}\frac{1}{\lambda}\log\left(\frac{a}{b}\right).
\]
\item[(iv)] If $\lambda=\infty$, then $x$ obeys
\begin{equation*}
\lim_{t\to\infty} \frac{\log x(t)}{\int_0^t
\frac{1}{\sigma(s)}\,ds} = -\frac{1}{\beta}\log\left(\frac{a}{b}\right).
\end{equation*}
\end{itemize}
\end{theorem}

\section{Proof of General Results}

\subsection{Proof of Theorem~\ref{thm3.1}}
The proof in part (b) is almost identical, so we prove part (a) only. 
To prove part (a), let $\epsilon>0$ and define for all $t\geq -\bar{\tau}$ the function $x_{U,\epsilon}$ by
$x_{U,\epsilon}(t)=\epsilon+\max_{-\bar{\tau}\leq s\leq 0} \psi(s)=:M_\epsilon>0$. Then $x_{U,\epsilon}(t)>x(t)$
for all $t\in[-\bar{\tau},0]$. For $t>0$ we have
\[
x_{U,\epsilon}'(t)+ag(x_{U,\epsilon}(t))-bg(x_{U,\epsilon}(t-\tau(t)))
=(a-b)g(M_\epsilon)>0.
\]
Hence
\begin{align*}
x_{U,\epsilon}'(t)&>-ag(x_{U,\epsilon}(t))+bg(x_{U,\epsilon}(t-\tau(t))), \quad t>0\\
x_{U,\epsilon}(t)&>x(t), \quad t\in[-\bar{\tau},0].
\end{align*}
Therefore $x(t)<x_{U,\epsilon}(t)=M_\epsilon=\epsilon+\max_{-\bar{\tau}\leq s\leq 0} \psi(s)$ for all $t\geq -\bar{\tau}$. Letting $\epsilon\to0^+$
gives \eqref{eq.xbounded}. 

It is a consequence of \eqref{eq.xbounded} that there is $x^\ast\in[0,\infty)$ such that 
\[
x^\ast=\limsup_{t\to\infty} x(t).
\]
We suppose that $x^\ast>0$ and show that this leads to a contradiction, proving that $x(t)\to 0$ as $t\to\infty$. 
Since $x^\ast>0$, \eqref{eq:g3} implies that 
$g(x^\ast)>0$. Let $\epsilon_0>0$ be so small that 
\begin{equation} \label{eq.epsg}
-ag(x^\ast(1+2\epsilon))+b(1+2\epsilon)g(x^\ast)<0, \quad -a+b(1+2\epsilon)<0, \quad 0<\epsilon<\epsilon_0.
\end{equation}
The existence of such an $\epsilon_0>0$ follows from the continuity of $g$ and the fact that $a>b>0$ and $g(x^\ast)>0$. Since $\tau$ obeys 
\eqref{eq.tminustautoinfty}, and $g$ is non--decreasing and continuous, we have 
\[
\limsup_{t\to\infty} g(x(t-\tau(t)))=g(x^\ast).
\] 
Therefore for every $\epsilon>0$ which obeys \eqref{eq.epsg}, there exists $T(\epsilon)>0$ such that 
\begin{equation*}
g(x(t-\tau(t)))\leq (1+\epsilon)g(x^\ast), \quad x(t)\leq (1+\epsilon)x^\ast, \quad t\geq T(\epsilon). 
\end{equation*} 
Therefore we have 
\[
x'(t)\leq -ag(x(t))+b(1+\epsilon)g(x^\ast), \quad t\geq T(\epsilon), \quad x(T(\epsilon))\in[0,(1+\epsilon)x^\ast].
\]
Define $x_{+,\epsilon}$ by 
\begin{equation*}
x_{+,\epsilon}'(t)=-ag(x_{+,\epsilon}(t))+b(1+2\epsilon)g(x^\ast), \quad t\geq T(\epsilon); \quad x_{+,\epsilon}(T(\epsilon))=x^\ast(1+2\epsilon). 
\end{equation*}
Therefore $x(t)<x_{+,\epsilon}(t)$ for $t\geq T(\epsilon)$. By \eqref{eq.epsg} we have $x_{+,\epsilon}'(T(\epsilon))<0$. 
Define $G_\epsilon(x)=-ag(x)+b(1+2\epsilon)g(x^\ast)$ for $x\in[0,x^\ast(1+2\epsilon)]$. By the second statement in \eqref{eq.epsg} 
we have $G_\epsilon(x^\ast)<0$ and as $g$ is non--decreasing, $G_\epsilon$ is non--increasing. Since $G_\epsilon(0)=b(1+2\epsilon)g(x^\ast)>0$
there is a maximal $x_\ast(\epsilon)\in(0,x^\ast)$ such that $G_\epsilon(x)<0$ for all $x>x_\ast(\epsilon)$ and $G_\epsilon(x_\ast(\epsilon))=0$.
We have $x_{+,\epsilon}'(t)=G_\epsilon(x_{+,\epsilon}(t))$ for $t\geq T(\epsilon)$ and $x_{+,\epsilon}(T(\epsilon))=x^\ast(1+2\epsilon)>x^\ast>x_\ast(\epsilon)$ and so
$x_{+,\epsilon}$ is decreasing on $[T(\epsilon),\infty)$ and attains the limit $\lim_{t\to\infty} x_{+,\epsilon}(t)=x_\ast(\epsilon)$. Therefore
\[
x^\ast=\limsup_{t\to\infty} x(t)\leq \lim_{t\to\infty} x_{+,\epsilon}(t)=x_\ast(\epsilon)<x^\ast,
\] 
a contradiction. Hence we must have $x^\ast=0$, proving that $x(t)\to 0$ as $t\to\infty$ as required.

\subsection{Proof of Theorem~\ref{thm.odedecaygen}}
We first need to prove that
\begin{equation}  \label{eq.GxgtLambdat}
G(x(t))\geq \Lambda t, \quad t\geq1.
\end{equation}

Define $G_0$ by
\begin{equation} \label{def.G0}
G_0(x)=\int_x^{x(0)} \frac{1}{g(u)}\,du, \quad x>0.
\end{equation}
Note that $G_0$ is decreasing on $(0,\infty)$ is therefore invertible.
Since $x(t)>0$ for all $t\geq 0$, $b>0$ and $g(x(t-\tau(t)))>0$ for all $t\geq 0$ we have
\[
x'(t)\geq -ag(x(t)), \quad t>0.
\]
Hence
\[
-G_0(x(t))=\int_{x(0)}^{x(t)} \frac{1}{g(u)}\,du=\int_0^t \frac{x'(s)}{g(x(s))}\,ds \geq -at,\quad t\geq 0.
\]
Therefore
\begin{equation} \label{eq.Gxleqat}
G_0(x(t))\leq at, \quad t\geq 0.
\end{equation}
Note that there exists $c\in[-\infty,0)$ such that $G_0:(0,\infty)\to (-c,\infty)$, because $G_0(0)=\infty$, $G_0(x(0))=0$ and $G_0$ is decreasing.
Therefore $G_0^{-1}:(-c,\infty)\to (0,\infty)$. Hence we may define $\Gamma:[0,\infty)\to (0,\infty)$ by
\begin{equation}\label{def.Gamma}
\Gamma(x)=g(G_0^{-1}(x)).
\end{equation}
Therefore $\Gamma\in \text{RV}_\infty(-\gamma)$.

%

Since $a$, $b$, $q$ and $\gamma$ obey \eqref{eq.abqsmallgamma}, we may fix $\epsilon\in(0,1)$ so small that
\begin{equation} \label{eq.condabepsa}
a>b\frac{1}{1-\epsilon}(1-q-\epsilon)^{-\gamma}, \quad 1-q-\epsilon>0.
\end{equation}
Define $\eta(\epsilon)\in(0,1)$ so that
\begin{equation} \label{def.etaeps}
\eta(\epsilon)<\frac{1}{a(1-q-\epsilon)}\left(a-b\frac{1}{1-\epsilon}(1-q-\epsilon)^{-\gamma}\right).
\end{equation}
Therefore for the same $\epsilon\in(0,1)$, since $\Gamma\in \text{RV}_\infty(-\gamma)$, we have
\[
\lim_{x\to\infty} \frac{\Gamma((1-q-\epsilon)x)}{\Gamma(x)}=(1-q-\epsilon)^{-\gamma}.
\]
Therefore there exists $x_1(\epsilon)>0$ such that
\begin{equation} \label{eq.gammarvepsa}
\frac{\Gamma(x(1-q-\epsilon))}{\Gamma(x)}\leq \frac{1}{1-\epsilon}\left(\frac{1}{1-q-\epsilon}\right)^{\gamma}, \quad x\geq x_1(\epsilon).
\end{equation}
Next, as $G_0(x)\to\infty$ as $x\to 0^+$, and $g$ is increasing on $(0,\delta_1)$, there is $\delta(\epsilon)>0$ such that
\begin{equation} \label{eq.epsGdelxa}
\eta(\epsilon) G_0(\delta(\epsilon))>x_1(\epsilon), \quad \eta(\epsilon)(1-q-\epsilon)G_0(\delta(\epsilon))>G_0(\delta_1).
\end{equation}
Since $x(t)\to 0$ as $t\to\infty$ and $\tau(t)/t\to q$ as $\to\infty$, there is a $T_1(\epsilon)>0$ such that
\begin{equation} \label{eq.T1defa}
x(t)\leq \delta(\epsilon), \quad \tau(t)< (q+\epsilon) t, \quad t>T_1(\epsilon).
\end{equation}
By \eqref{eq.Gxleqat}, \eqref{eq.T1defa}, \eqref{eq.epsGdelxa} and using the fact that $G_0$ is decreasing, we have
\begin{equation} \label{eq.aT1x1epsa}
aT_1(\epsilon)\geq G_0(x(T_1(\epsilon)))\geq G_0(\delta(\epsilon))>\frac{x_1(\epsilon)}{\eta(\epsilon)}.
\end{equation}
Next define
\begin{equation} \label{def.lambdaepsa}
0<\lambda(\epsilon)=\frac{\eta(\epsilon)(1-q-\epsilon)G_0(\delta(\epsilon))}{T_1(\epsilon)}.
\end{equation}
Then by \eqref{eq.aT1x1epsa}, as $aT_1(\epsilon)\geq G_0(\delta(\epsilon))$, we have
\[
\lambda(\epsilon)=\eta(\epsilon)(1-q-\epsilon)\frac{G_0(\delta(\epsilon))}{T_1(\epsilon)}\leq a\eta(\epsilon)(1-q-\epsilon).
\]
By \eqref{def.etaeps}, we therefore have
\begin{equation} \label{eq.lambdaltabepsa}
\lambda(\epsilon)<a-b\frac{1}{1-\epsilon}(1-q-\epsilon)^{-\beta/(\beta-1)}.
\end{equation}
Define $x_{U,\epsilon}$ by
\begin{equation} \label{def.xUodea}
x_{U,\epsilon}(t)=G_0^{-1}(\lambda(\epsilon)t),\quad t\geq T_1(\epsilon).
\end{equation}
Since $t-\tau(t)\to\infty$, there exists $T_+(\epsilon)>T_1(\epsilon)$ such that
$T_+(\epsilon)=\sup\{t>T_1(\epsilon)\,:\, t-\tau(t)=T_1(\epsilon)\}$. Then we have
$t-\tau(t)\geq T_1(\epsilon)$ for all $t\geq T_+(\epsilon)$ and $T_+(\epsilon)-\tau(T_+(\epsilon))=T_1(\epsilon)$. Since
$T_+(\epsilon)>T_1(\epsilon)$, we have $\tau(T_+(\epsilon))<(q+\epsilon) T_+(\epsilon)$. Hence with $T_2(\epsilon):=T_1(\epsilon)/(1-q-\epsilon)$,
we have $T_2(\epsilon)>T_1(\epsilon)$ and
\[
T_1(\epsilon)=T_+(\epsilon)-\tau(T_+(\epsilon))>T_+(\epsilon)-(q+\epsilon) T_+(\epsilon).
\]
Therefore $T_2(\epsilon)=T_1(\epsilon)/(1-q-\epsilon)>T_+(\epsilon)$. Thus for $t>T_2(\epsilon)>T_+(\epsilon)$ we have $t-\tau(t)\geq T_1(\epsilon)$.

For $t\in[T_1(\epsilon), T_2(\epsilon)]$ we have that $x_{U,\epsilon}$ is decreasing, so by \eqref{def.xUodea} and \eqref{def.lambdaepsa} we have
\begin{align*}
x_{U,\epsilon}(t)&\geq x_{U,\epsilon}(T_2(\epsilon))=x_{U,\epsilon}(T_1(\epsilon)/(1-q-\epsilon))\\
&=G_0^{-1}(\lambda(\epsilon)T_1(\epsilon)/(1-q-\epsilon))\\
&=G_0^{-1}(\eta(\epsilon) G_0(\delta(\epsilon))).
\end{align*}
By \eqref{def.etaeps}, $\eta(\epsilon)\in(0,1)$, so we have $G_0(\delta(\epsilon))>\eta(\epsilon) G_0(\delta(\epsilon))$. Since $G_0^{-1}$ is decreasing, $\delta(\epsilon)=G_0^{-1}(G_0(\delta(\epsilon)))<G_0^{-1}(\eta(\epsilon) G_0(\delta(\epsilon)))$. Therefore for $t\in[T_1(\epsilon),T_2(\epsilon)]$ we have $x_{U,\epsilon}(t)>\delta(\epsilon)$. By \eqref{eq.T1defa} we have
\begin{equation} \label{eq.xUicodea}
x_{U,\epsilon}(t)>\delta(\epsilon)\geq x(t), \quad t\in[T_1(\epsilon),T_2(\epsilon)].
\end{equation}

Next we show that for $t\geq T_1(\epsilon)$ we have $x_{U,\epsilon}(t)<\delta_1$.
To see this, note that $x_{U,\epsilon}$ is decreasing on $[T_1(\epsilon),\infty)$, we have
\begin{align*}
x_{U,\epsilon}(t)&\leq x_{U,\epsilon}(T_1(\epsilon))\\
&=G_0^{-1}(\lambda(\epsilon)T_1(\epsilon))\\
&=G_0^{-1}(\eta(\epsilon)(1-q-\epsilon) G_0(\delta(\epsilon))).
\end{align*}
Since $G_0^{-1}$ is decreasing, the second member of \eqref{eq.epsGdelxa} yields $G_0^{-1}(\eta(\epsilon)(1-q-\epsilon)G_0(\delta(\epsilon)))<\delta_1$. Hence
\begin{equation} \label{eq.xUinmonotoneregiona}
x_{U,\epsilon}(t)<\delta_1, \quad t\geq T_1(\epsilon).
\end{equation}

For $t>T_2(\epsilon)>T_1(\epsilon)$, we have $\tau(t)<(q+\epsilon) t$ by \eqref{eq.T1defa}. Thus $t-\tau(t)>(1-q-\epsilon)t$. Since $x_{U,\epsilon}$
is decreasing on $[T_1(\epsilon),\infty)$, and for $t\geq T_2(\epsilon)$ we have  $t-\tau(t)>(1-q-\epsilon)t \geq (1-q-\epsilon)T_2(\epsilon)=T_1(\epsilon)$, we have $x_{U,\epsilon}(t-\tau(t))<x_{U,\epsilon}(t(1-q-\epsilon))\leq x_{U,\epsilon}(T_1(\epsilon))$.
By \eqref{eq.xUinmonotoneregiona} we have $x_{U,\epsilon}(t-\tau(t))<x_{U,\epsilon}(t(1-q-\epsilon))<\delta_1$ for $t\geq T_2(\epsilon)$. Since
$g$ is increasing on $(0,\delta_1)$ we have
\begin{equation} \label{eq.xUlosetaua}
g(x_{U,\epsilon}(t-\tau(t)))<g(x_{U,\epsilon}(t(1-q-\epsilon))), \quad t>T_2(\epsilon).
\end{equation}
For $t\geq T_2(\epsilon)$, we have $G_0(x_{U,\epsilon}(t))=\lambda(\epsilon)t$, so $G_0'(x_{U,\epsilon}(t))x_{U,\epsilon}'(t)=\lambda(\epsilon)$. Thus
\begin{equation} \label{eq.xUpra}
x_{U,\epsilon}'(t)=-\lambda(\epsilon) g(x_{U,\epsilon}(t)), \quad t>T_2(\epsilon).
\end{equation}
Therefore for $t\geq T_2(\epsilon)$ using \eqref{eq.xUpra}, \eqref{eq.xUlosetaua}, \eqref{def.xUodea} and \eqref{def.Gamma} in turn we have
\begin{align*}
\lefteqn{x_{U,\epsilon}'(t)+ag(x_{U,\epsilon}(t))-bg(x_{U,\epsilon}(t-\tau(t)))}\\
&=-\lambda(\epsilon)g(x_{U,\epsilon}(t))+ag(x_{U,\epsilon}(t))-bg(x_{U,\epsilon}(t-\tau(t)))\\
&=g(x_{U,\epsilon}(t))\left\{a-\lambda(\epsilon)-b\frac{g(x_{U,\epsilon}(t-\tau(t)))}{g(x_{U,\epsilon}(t))}\right\}\\
&>g(x_{U,\epsilon}(t))\left\{a-\lambda(\epsilon)-b\frac{g(x_{U,\epsilon}(t(1-q-\epsilon)))}{g(x_{U,\epsilon}(t))}\right\}\\
&=g(x_{U,\epsilon}(t))\left\{a-\lambda(\epsilon)-b\frac{g(G_0^{-1}(\lambda(\epsilon)t(1-q-\epsilon)))}{g(G_0^{-1}(\lambda(\epsilon)t))}\right\}\\
&=g(x_{U,\epsilon}(t))\left\{a-\lambda(\epsilon)-b\frac{\Gamma(\lambda(\epsilon)t(1-q-\epsilon))}{\Gamma(\lambda(\epsilon)t)}\right\}.
\end{align*}
Therefore
\begin{multline} \label{eq.xUdiffineq1a}
x_{U,\epsilon}'(t)+ag(x_{U,\epsilon}(t))-bg(x_{U,\epsilon}(t-\tau(t)))
\\>g(x_{U,\epsilon}(t))\left\{a-\lambda(\epsilon)-b\frac{\Gamma(\lambda(\epsilon)t(1-q-\epsilon))}{\Gamma(\lambda(\epsilon)t)}\right\}, \quad t\geq T_2(\epsilon).
\end{multline}
For $t\geq T_2(\epsilon)=T_1(\epsilon)/(1-q-\epsilon)$, we have by \eqref{def.lambdaepsa}, \eqref{eq.epsGdelxa}
\[
\lambda(\epsilon)t \geq \lambda(\epsilon)\frac{T_1(\epsilon)}{1-q-\epsilon}=\eta(\epsilon) G_0(\delta(\epsilon))>x_1(\epsilon).
\]
Therefore by \eqref{eq.gammarvepsa} we have
\begin{equation} \label{eq.gammarvepsta}
\frac{\Gamma(\lambda(\epsilon)t(1-q-\epsilon))}{\Gamma(\lambda(\epsilon)t)} \leq \frac{1}{1-\epsilon}\left(\frac{1}{1-q-\epsilon}\right)^{\gamma},
\quad t\geq T_2(\epsilon).
\end{equation}
By \eqref{eq.xUdiffineq1a} and \eqref{eq.gammarvepsta} we have for $t\geq T_2(\epsilon)$
\[
x_{U,\epsilon}'(t)+ag(x_{U,\epsilon}(t))-bg(x_{U,\epsilon}(t-\tau(t)))
>g(x_{U,\epsilon}(t))\left\{a-\lambda(\epsilon)-b\frac{1}{1-\epsilon}\left(\frac{1}{1-q-\epsilon}\right)^{\gamma}
\right\}.
\]
By \eqref{eq.lambdaltabepsa} we have
\begin{equation} \label{eq.xUdiffineq2a}
x_{U,\epsilon}'(t)>-ag(x_{U,\epsilon}(t))+bg(x_{U,\epsilon}(t-\tau(t))), \quad t>T_2(\epsilon).
\end{equation}
By \eqref{eq.xUicodea} and \eqref{eq.xUdiffineq2a}, by the comparison principle we have $x(t)<x_{U,\epsilon}(t)$ for all $t\geq T_1(\epsilon)$.
Therefore $x(t)<G_0^{-1}(\lambda(\epsilon)t)$ for $t\geq T_1(\epsilon)$. Hence $G_0(x(t))>\lambda(\epsilon)t$ for $t\geq T_1(\epsilon)$.
Since $\epsilon\in(0,1)$ obeying \eqref{eq.condabepsa} is fixed we have \eqref{eq.GxgtLambdata}.
Therefore \eqref{eq.odedecay2} holds, once we remember that $G$ also obeys \eqref{eq.Gxleqat}, together with the fact that
$G_0(x)/G(x)\to 1$ as $x\to 0^+$.

\subsection{Proof of Theorem~\ref{thm3.1superhypq}}
Let $\lambda=\log(1/(1-q))>0$. Define $\sigma(t)=\lambda(t+\bar{\tau}+1)$ for $t\geq -\bar{\tau}$.
Since $\tau(t)/t\to q\in(0,1)$ as $t\to\infty$, we have that $\sigma$ obeys \eqref{eq.t1}, \eqref{eq.t2} and \eqref{eq.t3}.

We first get an upper bound. By \eqref{eq:g3}, we note that there is a $\delta_0>0$ such that $g'(x)>0$ for all $x\in(0,\delta_0)$. Hence $g^{-1}:[0,g(\delta_0)]\to [0,\delta_0]$. Clearly there is a $\delta_1>0$ such that $\delta_1<g(\delta_0)$, or $g^{-1}(\delta_1)<\delta_0$. Then for $x\in(0,\delta_1)$, we may define $\Gamma_1:(0,\delta_1)\to(0,\infty)$ by
\[
\Gamma_1(x)=g'(g^{-1}(x)), \quad x\in[0,\delta_1).
\]
By \eqref{eq:g4alt6}, we have $\Gamma_1\in \text{RV}_0(1/\gamma)$.

Since
\[
a<b\left(\frac{1}{1-q}\right)^{\gamma}, \quad q=1-e^{-\lambda},  
\]
we have
\[
\log(a/b)<\lambda \gamma.
\]
Therefore we can choose $\epsilon\in(0,1)$ so small that
\begin{equation}    \label{eq.epssmallq}
1-\frac{1}{\gamma}\log(a/b) \frac{1}{\lambda}(1+\epsilon)^{3/2}>0.
\end{equation}
Since $\Gamma_1\in \text{RV}_0(1/\gamma)$, there exists $0<\delta_2(\epsilon)<\delta_0$ such that $g(\delta_2(\epsilon)/2)<1$ and
\[
x^{1/\gamma\cdot (1+\epsilon)^2}\leq  \Gamma_1(x) \leq x^{1/\gamma \cdot (1-\epsilon)^2}, \quad x\in(0,g(\delta_2(\epsilon)/2)].
\]
Let $\delta_3>0$ be so small that $g(\delta_3/2)<1$. Now, let $\delta(\epsilon)=\min(\delta_0,\delta_1,\delta_2(\epsilon),\delta_3)$. Since
$\delta(\epsilon)\leq \delta_2$, we have $\delta(\epsilon)/2\leq \delta_2(\epsilon)/2<\delta_0/2<\delta_0$. Therefore $g(\delta(\epsilon)/2)\leq g(\delta_2(\epsilon)/2)$ and so we have
\begin{align}
\label{eq.estdelta2q1}
0<B_1(\epsilon)&:=\frac{1}{\gamma}(1-\epsilon)^2 \leq
\frac{\log \Gamma_1(x)}{\log x}
\\
&\qquad\qquad\leq \frac{1}{\gamma}(1+\epsilon)^2=:B_2(\epsilon),\quad x\in(0,g(\delta(\epsilon)/2)], \nonumber\\
\label{eq.estdelta3q}
g(\delta(\epsilon)/2)&<1, \quad g'(x)>0 \quad x\in(0,\delta(\epsilon)).
\end{align}

Since $x(t)\to 0$ as $t\to\infty$, there exists $T_0(\epsilon)>0$ such that $x(t)\leq \delta(\epsilon)/2$ for all $t\geq T_0(\epsilon)$.

By \eqref{eq.t1} and \eqref{eq.t3} for each $\epsilon\in(0,1)$, there is a
$T_1(\epsilon)>0$ such that
\begin{equation} \label{eq.defT1q}
\int_{t-\tau(t)}^t \frac{1}{\sigma(s)}\,ds \leq 1+\epsilon, \quad
t\geq T_1(\epsilon).
\end{equation}

Now let $0<c_2<c_2(\epsilon)$, where
$c_2(\epsilon)\in(0,\log(a/b))$ is the solution of
$g_\epsilon(c)=0$, where the function $g_\epsilon$ is defined by
\[
g_\epsilon(x) = -x\epsilon + a - be^{x(1+\epsilon)}, \quad x\geq
0.
\]
The existence and uniqueness of the solution are guaranteed by the
fact that $g_\epsilon(0)>0$, $g_{\epsilon}((1+\epsilon)^{-1}\log(a/b))<0$ and
$g_\epsilon$ is continuous and decreasing on $(0,\infty)$. Note
moreover that as $\epsilon\downarrow 0$, we have that
$c_2(\epsilon)\to \log(a/b)$. A further consequence of the construction is that
\[
c_2<\frac{1}{1+\epsilon}\log(a/b).
\]

Since $\sigma(t)/t\to\lambda$ as $t\to\infty$, we have
\[
\lim_{t\to\infty} \frac{\int_0^t \frac{1}{\sigma(s)}\,ds}{\log \sigma(t)}=\frac{1}{\lambda}.
\]
Therefore there exists $\theta_1(\epsilon)>0$ such that
\[
\sigma(t)\geq e, \quad
\frac{\int_0^t \frac{1}{\sigma(s)}\,ds}{\log \sigma(t)}\leq \frac{1}{\lambda}(1+\epsilon)^{1/2}, \quad t\geq \theta_1(\epsilon).
\]
Therefore for $t\geq \theta_1(\epsilon)$ we have
\begin{align*}
\lefteqn{\log \sigma(t) - \frac{1}{\gamma}\log(a/b)(1+\epsilon)\int_0^t \frac{1}{\sigma(s)}\,ds}\\
&\geq \frac{1}{1/\lambda (1+\epsilon)^{1/2}} \int_0^t \frac{1}{\sigma(s)}\,ds - \frac{1}{\gamma}\log(a/b)(1+\epsilon)\int_0^t \frac{1}{\sigma(s)}\,ds\\
&=\frac{\lambda}{(1+\epsilon)^{1/2}}\int_0^t \frac{1}{\sigma(s)}\,ds \left\{1-\frac{1}{\lambda}\frac{1}{\gamma}\log(a/b)(1+\epsilon)^{3/2}\right\}.
\end{align*}
The quantity in curly brackets is positive by \eqref{eq.epssmallq}. Therefore as $\int_0^t 1/\sigma(s)\,ds\to\infty$ as $t\to\infty$, for every
$\epsilon\in(0,1)$ obeying \eqref{eq.epssmallq}, there exists $T_2(\epsilon)>0$ such that
\begin{multline} \label{eq.defT2q}
\log \sigma(t)  -\frac{1}{\gamma}\log(a/b)(1+\epsilon)\int_0^t \frac{1}{\sigma(s)}\,ds + B_2(\epsilon)\log g(\delta(\epsilon)/2)\\
> \log(1/\epsilon), \quad t\geq T_2(\epsilon).
\end{multline}

Since $x(t)\to 0$ as $t\to\infty$, there exists a $T_3(\epsilon)>0$ such that
\begin{equation} \label{eq.defT3q}
g(x(t))\leq g(\delta(\epsilon)/2)e^{-2c_2}, \quad t\geq T_3(\epsilon). 
\end{equation}
Let $T_4(\epsilon)=\max(T_0(\epsilon),T_1(\epsilon),T_2(\epsilon),T_3(\epsilon))$.
Since $t-\tau(t)\to\infty$ as $t\to\infty$, there exists $T_5(\epsilon)>T_4(\epsilon)$ such that $T_5(\epsilon)=\sup\{t>T_4(\epsilon):t-\tau(t)=T_4(\epsilon)\}$. Also $t-\tau(t)\geq T_4(\epsilon)$ for $t\geq T_5(\epsilon)$.

Define $x_2(\epsilon)$ by
\begin{equation} \label{def.x2q}
x_2(\epsilon)= g(\delta(\epsilon)/2)e^{c_2\int_0^{T_4(\epsilon)} \frac{1}{\sigma(s)}\,ds}.
\end{equation}
Then for $t\geq T_4(\epsilon)$ we have
\[
0<x_2(\epsilon)e^{-c_2\int_0^t \frac{1}{\sigma(s)}\,ds}= g(\delta(\epsilon)/2)e^{c_2\int_0^{T_4(\epsilon)} \frac{1}{\sigma(s)}\,ds}e^{-c_2\int_0^t \frac{1}{\sigma(s)}\,ds} \leq g(\delta(\epsilon)/2)<1.
\]
Since $\delta_0\leq \delta(\epsilon)$, we have that $g^{-1}:[0,g(\delta(\epsilon))]\to[0,\delta(\epsilon)]$. Therefore, we can define
\begin{equation} \label{def.xUlongmemq}
x_{U,\epsilon}(t)=g^{-1}\left(x_2(\epsilon)e^{-c_2\int_0^t \frac{1}{\sigma(s)}\,ds}\right), \quad t\geq T_4(\epsilon).
\end{equation}
Therefore we have
\begin{equation*}
x_{U,\epsilon}(t)\leq \delta(\epsilon)/2, \quad t\geq T_4(\epsilon); \quad g(x_{U,\epsilon}(t))<1, \quad t\geq T_4(\epsilon).
\end{equation*}
Next for $t\in[T_4(\epsilon),T_5(\epsilon)]$, noting that $T_5(\epsilon)-\tau(T_5(\epsilon))=T_4(\epsilon)$ and $T_4(\epsilon)\geq T_3(\epsilon)$, we have by \eqref{eq.defT3q}, \eqref{def.x2q} that
\begin{align*}
g(x(t))e^{c_2\int_0^t \frac{1}{\sigma(u)}\,du}
&\leq  g(x(t))e^{c_2\int_0^{T_5(\epsilon)} \frac{1}{\sigma(u)}\,du}\\
&\leq g(\delta(\epsilon)/2)e^{-2c_2} e^{c_2\int_0^{T_5(\epsilon)} \frac{1}{\sigma(u)}\,du}\\
&=g(\delta(\epsilon)/2)e^{-2c_2} e^{c_2\int_{0}^{T_5(\epsilon)-\tau(T_5(\epsilon))}\frac{1}{\sigma(u)}\,du}\cdot e^{c_2\int_{T_5(\epsilon)-\tau(T_5(\epsilon))}^{T_5(\epsilon)} \frac{1}{\sigma(u)}\,du}\\
&=g(\delta(\epsilon)/2)e^{-2c_2} e^{c_2\int_{0}^{T_4(\epsilon)}\frac{1}{\sigma(u)}\,du}\cdot e^{c_2\int_{T_5(\epsilon)-\tau(T_5(\epsilon))}^{T_5(\epsilon)} \frac{1}{\sigma(u)}\,du}\\
&=x_2(\epsilon) e^{-2c_2} \cdot e^{c_2\int_{T_5(\epsilon)-\tau(T_5(\epsilon))}^{T_5(\epsilon)} \frac{1}{\sigma(u)}\,du}.
\end{align*}
Since $T_5(\epsilon)> T_4(\epsilon)\geq T_1(\epsilon)$, by \eqref{eq.defT1q} for $t\in[T_4(\epsilon),T_5(\epsilon)]$ we have
\[
g(x(t))e^{c_2\int_0^t \frac{1}{\sigma(u)}\,du}\leq
x_2(\epsilon) e^{-2c_2} \cdot e^{c_2(1+\epsilon)}= x_2(\epsilon) e^{-c_2(1-\epsilon)}<x_2(\epsilon).
\]
Hence for $t\in[T_4(\epsilon),T_5(\epsilon)]$ we have
\[
g(x(t)) < x_2(\epsilon)e^{-c_2\int_0^t \frac{1}{\sigma(u)}\,du}=g(x_{U,\epsilon}(t))\leq g(\delta(\epsilon)/2)<g(\delta(\epsilon))\leq g(\delta_0).
\]
Therefore we have
\begin{equation} \label{eq.xUlongmemicq}
x(t)<x_{U,\epsilon}(t), \quad t\in[T_4(\epsilon),T_5(\epsilon)].
\end{equation}

Since $g\in C^1(0,\delta_0)$, and $\delta_1<g(\delta_0)$,
we have that $g^{-1}\in C^1(0,\delta_1)$. Now for $t\geq T_4(\epsilon)$ we have
$g(x_{U,\epsilon}(t))\leq \delta(\epsilon)/2<\delta(\epsilon)\leq \delta_1$, so therefore $x_{U,\epsilon}\in C^1[T_5(\epsilon),\infty)$
and
\[
x_{U,\epsilon}'(t)=-\frac{c_2}{\sigma(t)g'(x_{U,\epsilon}(t))}x_2(\epsilon)e^{-c_2\int_0^t \frac{1}{\sigma(s)}\,ds}, \quad t\geq T_5(\epsilon).
\]
For $t>T_5(\varepsilon)$, we have that $t-\tau(t)\geq T_4(\epsilon)$. Therefore we get
\begin{multline} \label{eq.xUdiffeqlongmem1q}
x_{U,\epsilon}'(t) + ag(x_{U,\epsilon}(t)) -
bg(x_{U,\epsilon}(t-\tau(t)))
\\=x_2(\epsilon)e^{-c_2\int_{0}^t \frac{1}{\sigma(s)}\,ds}\left(
-\frac{c_2}{g'(x_{U,\epsilon}(t))\sigma(t)} + a - b e^{c_2\int^{t}_{t-\tau(t)}
\frac{1}{\sigma(s)}\,ds} \right), \quad t>T_5(\epsilon).
\end{multline}
Next we estimate $g'(x_{U,\epsilon}(t))\sigma(t)$ for $t\geq T_5(\epsilon)$. Since this quantity is positive, we may consider $\log(g'(x_{U,\epsilon}(t))\sigma(t))$.
Since \eqref{eq.estdelta2q1} holds 
and $g$ is increasing on $(0,\delta(\epsilon)/2)$ we have
\[
B_1(\epsilon)\leq
\frac{\log \Gamma_1(g(x))}{\log g(x)} \leq B_2(\epsilon), \quad x\in(0,\delta(\epsilon)/2].
\]
Since $t\geq T_5(\epsilon)>T_4(\epsilon)$, we have $x_{U,\epsilon}(t)\leq \delta(\epsilon)/2$. Therefore
\[
0<B_1(\epsilon)\leq
\frac{\log \Gamma_1(g(x_{U,\epsilon}(t)))}{\log g(x_{U,\epsilon}(t))} \leq B_2(\epsilon),\quad t\geq T_5(\epsilon).
\]
For $t\geq T_5(\epsilon)>T_4(\epsilon)$, we have that $g(x_{U,\epsilon}(t))\leq g(\delta/2)<1$, so $\log g(x_{U,\epsilon}(t))<0$.
Hence
\begin{equation} \label{eq.B2loggammaq}
B_1(\epsilon)\log g(x_{U,\epsilon}(t))\geq \frac{\log \Gamma_1(g(x_{U,\epsilon}(t)))}{\log g(x_{U,\epsilon}(t))}\log g(x_{U,\epsilon}(t))\geq B_2(\epsilon)\log g(x_{U,\epsilon}(t)).
\end{equation}
Hence for $t\geq T_5(\epsilon)>T_4(\epsilon)$, by using \eqref{eq.B2loggammaq}, \eqref{def.xUlongmemq}, and \eqref{def.x2q}  in turn
we get
\begin{align*}
\log(g'(x_{U,\epsilon}(t))\sigma(t))
&=\log \sigma(t) + \frac{\log \Gamma_1(g(x_{U,\epsilon}(t)))}{\log g(x_{U,\epsilon}(t))}\log g(x_{U,\epsilon}(t))\\
&\geq \log \sigma(t) + B_2(\epsilon)\log g(x_{U,\epsilon}(t))\\
&= \log \sigma(t) + B_2(\epsilon)\left( \log x_2(\epsilon) -c_2\int_0^t \frac{1}{\sigma(s)}\,ds\right)\\
&=\log \sigma(t) + B_2(\epsilon)\log g(\delta(\epsilon)/2) -B_2(\epsilon) c_2\int_{T_4(\epsilon)}^t \frac{1}{\sigma(s)}\,ds\\
&>\log \sigma(t) + B_2(\epsilon)\log g(\delta(\epsilon)/2) -B_2(\epsilon) c_2\int_0^t \frac{1}{\sigma(s)}\,ds\\
&=\log \sigma(t) + B_2(\epsilon)\log g(\delta(\epsilon)/2) \\
&\qquad\qquad-\frac{1}{\gamma}(1+\epsilon)^2 c_2\int_0^t \frac{1}{\sigma(s)}\,ds.
\end{align*}
Since $c_2<(1+\epsilon)^{-1}\log(a/b)$, we have
\[
\frac{1}{\gamma}(1+\epsilon)^2 c_2\int_0^t \frac{1}{\sigma(s)}\,ds
< \frac{1}{\gamma} (1+\epsilon) \log(a/b) \int_0^t \frac{1}{\sigma(s)}\,ds.
\]
Hence for $t\geq T_5(\epsilon)>T_4(\epsilon)$ we have
\[
\log(g'(x_{U,\epsilon}(t))\sigma(t))
>\log \sigma(t) + B_2(\epsilon)\log g(\delta(\epsilon)/2) -
\frac{1}{\gamma}(1+\epsilon) \log(a/b) \int_0^t \frac{1}{\sigma(s)}\,ds.
\]
Therefore, by \eqref{eq.defT2q}, we get
\begin{equation} \label{eq.gsigq}
g'(x_{U,\epsilon}(t))\sigma(t)>1/\epsilon, \quad t\geq T_5(\epsilon).
\end{equation}
Since $t>T_5(\epsilon)$ by inserting \eqref{eq.gsigq} and \eqref{eq.defT1q} into \eqref{eq.xUdiffeqlongmem1q}
\[
x_{U,\epsilon}'(t) + ag(x_{U,\epsilon}(t)) -
bg(x_{U,\epsilon}(t-\tau(t)))
>x_2(\epsilon)e^{-c_2\int_{0}^t \frac{1}{\sigma(s)}\,ds}\left(
-\frac{c_2}{1/\epsilon} + a - b e^{c_2(1+\epsilon)} \right).
\]
The quantity in brackets is nothing other than $g_\epsilon(c_2)>0$. Thus we have that
\begin{align*}
x_{U,\epsilon}'(t)&>-ag(x_{U,\epsilon}(t)) + bg(x_{U,\epsilon}(t-\tau(t))), \quad t>T_5(\epsilon), \\
x_{U,\epsilon}(t)&> x(t)>0, \quad
t\in[T_4(\epsilon),T_5(\epsilon)].
\end{align*}
Therefore $x(t)<x_{U,\epsilon}(t)$ for all $t\geq T_4(\epsilon)$. Hence
\[
g(x(t)) < x_2(\epsilon) e^{-c_2\int_0^t \frac{1}{\sigma(s)}\,ds},
\quad t\geq T_4(\epsilon).
\]
Thus
\[
\limsup_{t\to\infty} \frac{\log g(x(t))}{\int_0^t
\frac{1}{\sigma(s)}\,ds} \leq -c_2.
\]
Letting $c_2\uparrow c_2(\epsilon)$, we get
\[
\limsup_{t\to\infty} \frac{\log g(x(t))}{\int_0^t
\frac{1}{\sigma(s)}\,ds} \leq -c_2(\epsilon).
\]
Finally, because letting $\epsilon\downarrow0$ yields
$c_2(\epsilon)\to\log (a/b)$, by taking the limit as
$\epsilon\to0$, we get
\begin{equation} \label{eq.b3}
\limsup_{t\to\infty} \frac{\log g(x(t))}{\int_0^t
\frac{1}{\sigma(s)}\,ds} \leq -\log\left( \frac{a}{b}\right).
\end{equation}

We now determine a lower bound for $x$.
By \eqref{eq.t1} and \eqref{eq.t3}, for each $\epsilon\in(0,1)$ there is $T_6(\epsilon)>0$ such that
\[
\int_{t-\tau(t)}^t \frac{1}{\sigma(s)}\,ds  \geq 1-\epsilon, \quad
t\geq T_6(\epsilon).
\]
Also, let $c_1>c_1(\epsilon)=(1-\epsilon)^{-1} \log(a/b)>0$. Note that the definition of $T_0$ gives $x(t)\leq \delta(\epsilon)/2$ for all $t\geq T_0$. Define $T_7(\epsilon)=\max(T_0,T_6(\epsilon))$. Since $t-\tau(t)\to\infty$ as $t\to\infty$ we have that there is $T_8(\epsilon)>T_7(\epsilon)$
such that $T_8(\epsilon)=\sup\{t>T_7(\epsilon):t-\tau(t)=T_7(\epsilon)\}$. Then $t-\tau(t)\geq T_7(\epsilon)$ for all $t\geq T_8(\epsilon)$.
Next define
\[
x_1(\epsilon) = \frac{1}{2}\min\left\{\min_{T_7(\epsilon)\leq s\leq T_8(\epsilon)} g(x(s))e^{c_1\int_0^s\frac{1}{\sigma(u)}\,du},g(\delta(\epsilon)/2)e^{c_1\int_{0}^{T_7(\epsilon)} \frac{1}{\sigma(u)}\,du}\right\}
\]
so that
\[
x_1(\epsilon)<\min_{T_7(\epsilon)\leq s\leq T_8(\epsilon)} g(x(s))e^{c_1\int_0^s\frac{1}{\sigma(u)}\,du},
\quad x_1(\epsilon)<g(\delta(\epsilon)/2)e^{c_1\int_0^{T_7(\epsilon)} \frac{1}{\sigma(u)}\,du}.
\]
For $t\geq T_7(\epsilon)$ we have
\[
0<x_1(\epsilon)e^{-c_1\int_0^t \frac{1}{\sigma(s)}\,ds}\leq x_1(\epsilon)e^{-c_1\int_0^{T_7(\epsilon)} \frac{1}{\sigma(s)}\,ds}
<
g(\delta(\epsilon)/2)<g(\delta(\epsilon)).
\]
Since $g$ is increasing on $(0,\delta(\epsilon))$, we may define $g^{-1}:[0,g(\delta(\epsilon))]\to[0,\delta(\epsilon)]$ and therefore the function $x_{L,\epsilon}$ given by
\[
x_{L,\epsilon}(t) = g^{-1}\left(x_1(\epsilon)e^{-c_1\int_0^t
\frac{1}{\sigma(s)}\,ds}\right), \quad t\geq T_7(\epsilon)
\]
is well--defined.
Then $x_{L,\epsilon}(t)<\delta$ for all $t\geq T_7(\epsilon)$ and
\[
g(x_{L,\epsilon}(t)) = x_1(\epsilon)e^{-c_1\int_0^t
\frac{1}{\sigma(s)}\,ds}, \quad t\geq T_7(\epsilon).
\]
Then for $t\in[T_7(\epsilon),T_8(\epsilon)]$ we have
\begin{align*}
g(x_{L,\epsilon}(t))e^{c_1\int_0^t \frac{1}{\sigma(s)}\,ds}
&= x_1(\epsilon)\\
&<\min_{T_7(\epsilon)\leq s\leq T_8(\epsilon)} g(x(s))e^{c_1\int_0^s\frac{1}{\sigma(u)}\,du}\\
&\leq g(x(t))e^{c_1\int_0^t\frac{1}{\sigma(u)}\,du}.
\end{align*}
Thus $g(x_{L,\epsilon}(t))<g(x(t))$ for $t\in[T_7(\epsilon),T_8(\epsilon)]$, so therefore $x_{L,\epsilon}(t)<x(t)$ for
$t\in[T_7(\epsilon),T_8(\epsilon)]$. To see this, suppose to the contrary that there exists $t_1\in[T_7(\epsilon),T_8(\epsilon)]$
such that $x_{L,\epsilon}(t_1)\geq x(t_1)$. Since $x_{L,\epsilon}(t_1)<\delta(\epsilon)$ we have  $g(x(t_1))\leq g(x_{L,\epsilon}(t_1))<g(\delta(\epsilon))$.
Since we must also have $g(x_{L,\epsilon}(t_1))<g(x(t_1))$ there is a contradiction; hence $x_{L,\epsilon}(t)<x(t)$ for
$t\in[T_7(\epsilon),T_8(\epsilon)]$.

Since $g$ is in $C^1(0,\delta(\epsilon))$ and $\sigma$ is continuous $x_{L,\epsilon}$ is in $C^1(T_7(\epsilon),\infty)$ and moreover
\[
g'(x_{L,\epsilon}(t))x_{L,\epsilon}'(t)=-\frac{c_1}{\sigma(t)}x_1(\epsilon)e^{-c_1\int_0^t
\frac{1}{\sigma(s)}\,ds}.
\]
Since $x_{L,\epsilon}(t)\in(0,\delta(\epsilon))$ we have $x_{L,\epsilon}'(t)<0$ for all $t>T_7(\epsilon)$. Moreover, because $t-\tau(t)\geq T_7(\epsilon)$ for all $t>T_8(\epsilon)$, we have
\begin{multline*}
x_{L,\epsilon}'(t) + ag(x_{L,\epsilon}(t)) -
bg(x_{L,\epsilon}(t-\tau(t))) \\
=x_1(\epsilon)e^{-c_1\int_0^t
\frac{1}{\sigma(s)}\,ds}\biggl(
-\frac{1}{g'(x_{L,\epsilon}(t))}\frac{c_1}{\sigma(t)}+a
-be^{c_1\int_{t-\tau(t)}^t
\frac{1}{\sigma(s)}\,ds}\biggr).
\end{multline*}
For $t\geq T_8(\epsilon)\geq T_6(\epsilon)$, we have
\[
e^{c_1\int_{t-\tau(t)}^t \frac{1}{\sigma(s)}\,ds}\geq
e^{c_1(1-\epsilon)}.
\]
Thus for $t\geq T_8(\epsilon)$, as $g'(x_{L,\epsilon}(t))>0$ and $\sigma(t)>0$
\[
-\frac{1}{g'(x_{L,\epsilon}(t))}\frac{c_1}{\sigma(t)} + a - b e^{c_1\int^{t}_{t-\tau(t)}
\frac{1}{\sigma(s)}\,ds} < a-be^{c_1(1-\epsilon)}<0.
\]
Thus
\begin{gather*}
x_{L,\epsilon}'(t) < -ag(x_{L,\epsilon}(t)) +
bg(x_{L,\epsilon}(t-\tau(t))), \quad t>T_8(\epsilon), \\
x_{L,\epsilon}(t)<x(t), \quad t\in [T_7(\epsilon),T_8(\epsilon)].
\end{gather*}
As $g$ is increasing on $(0,\delta(\epsilon))$ and $x(t-\tau(t))$ and $x_{U,\epsilon}(t-\tau(t))$ are both in $(0,\delta(\epsilon))$ for $t\geq T_8(\epsilon)$, a standard comparison argument now shows that $x(t)>x_{L,\epsilon}(t)$ for all $t\geq T_7(\epsilon)$, which gives
\[
g(x(t)) > g(x_{L,\epsilon}(t))=x_1(\epsilon)e^{-c_1\int_0^t \frac{1}{\sigma(s)}\,ds},
\quad t\geq T_7(\epsilon).
\]
Since $\int_0^t \sigma(s)^{-1}\,ds \to\infty$ as $t\to\infty$,
\[
\liminf_{t\to\infty} \frac{\log g(x(t))}{\int_0^t
\frac{1}{\sigma(s)}\,ds} \geq -c_1.
\]
Hence, letting $c_1\uparrow c_1(\epsilon)$, we get
\[
\liminf_{t\to\infty} \frac{\log g(x(t))}{\int_0^t
\frac{1}{\sigma(s)}\,ds} \geq -\frac{1}{1-\epsilon}
\log\left(\frac{a}{b}\right).
\]
Letting $\epsilon\downarrow0$ yields
\begin{equation}  \label{eq.a3}
\liminf_{t\to\infty}\frac{\log g(x(t))}{\int_0^t
\frac{1}{\sigma(s)}\,ds} \geq -\log\left(\frac{a}{b}\right).
\end{equation}
Combining \eqref{eq.a3} and \eqref{eq.b3} gives
\[
\lim_{t\to\infty}\frac{\log g(x(t))}{\int_0^t
\frac{1}{\sigma(s)}\,ds} = -\log\left(\frac{a}{b}\right),
\]

Finally, because
\[
\lim_{t\to\infty} \frac{\int_0^t \frac{1}{\sigma(s)}\,ds}{\log t}=\frac{1}{\lambda}
\]
we have
\[
\lim_{t\to\infty} \frac{\log g(x(t))}{\log t}
=
\lim_{t\to\infty} \frac {\log g(x(t))}{\int_0^t \frac{1}{\sigma(s)}\,ds}\cdot\frac{\int_0^t \frac{1}{\sigma(s)}\,ds}{\log t}
=
-\log(a/b) \cdot \lambda.
\]
Since $\lambda=\log(1/(1-q))$ we have
\[
\lim_{t\to\infty} \frac{\log g(x(t))}{\log t}=-\frac{1}{\log(1/(1-q))}\log\left(\frac{a}{b}\right),
\]
as required.

\subsection{Proof of Theorem~\ref{thm3.1gen}}
We first need to prove a preliminary lemma.
\begin{lemma} \label{lemma.intlogsigma0}
Suppose that $\sigma$ obeys \eqref{eq.t1}, \eqref{eq.t4}. Then
\begin{equation} \label{eq.intlogsigma0}
\lim_{t\to\infty} \frac{\int_0^t \frac{1}{\sigma(s)}\,ds}{\log \sigma(t)}=0.
\end{equation}
\end{lemma}
\begin{proof}
In the case when $1/\sigma\in L^1(0,\infty)$, \eqref{eq.intlogsigma0} is automatically true. On the other hand, when \eqref{eq.t2}
holds, the limit in \eqref{eq.intlogsigma0} (if it exists) is of indeterminate form. By \eqref{eq.t4}, for every $M>1$, there exists
$T(M)>0$ such that $\sigma(t)/t\geq M$ for all $t\geq T(M)$. Therefore for $t\geq T(M)$ we have
\[
\int_{T(M)}^t \frac{1}{\sigma(s)}\,ds \leq \int_{T(M)}^t \frac{1}{Ms}\,ds = \frac{1}{M}\left(\log t - \log T(M)\right).
\]
Also for $t\geq T(M)$ we have $\log \sigma(t)\geq \log M+\log t\geq \log t$. Thus for $t\geq T(M)$ we have
\begin{align*}
\frac{\int_0^t \frac{1}{\sigma(s)}\,ds}{\log \sigma(t)}
&=
\frac{\int_0^{T(M)} \frac{1}{\sigma(s)}\,ds}{\log \sigma(t)} + \frac{\int_{T(M)}^t \frac{1}{\sigma(s)}\,ds}{\log \sigma(t)}\\
&\leq
\frac{\int_0^{T(M)} \frac{1}{\sigma(s)}\,ds}{\log t}+
\frac{1}{M}\cdot \frac{\log t - \log T(M)}{\log t}.
\end{align*}
Hence
\[
\limsup_{t\to\infty}
\frac{\int_0^t \frac{1}{\sigma(s)}\,ds}{\log \sigma(t)}
\leq \frac{1}{M}.
\]
Letting $M>1$, which is arbitrary, tend to $\infty$, we get \eqref{eq.intlogsigma0}.
\end{proof}

\begin{proof}[Proof of Theorem~\ref{thm3.1gen}]
We first get an upper bound. By \eqref{eq:g3}, we note that there is a $\delta_0>0$ such that $g'(x)>0$ for all $x\in(0,\delta_0)$. Hence $g^{-1}:[0,g(\delta_0)]\to [0,\delta_0]$. Clearly there is a $\delta_1>0$ such that $\delta_1<g(\delta_0)$, or $g^{-1}(\delta_1)<\delta_0$. Then for $x\in(0,\delta_1)$, we may define $\Gamma_1:(0,\delta_1)\to(0,\infty)$ by
\[
\Gamma_1(x)=g'(g^{-1}(x)), \quad x\in[0,\delta_1).
\]
Therefore $\Gamma\in \text{RV}_0(1/\gamma)$.
Therefore
\[
\lim_{x\to 0^+} \frac{\log \Gamma_1(x)}{\log x}=\frac{1}{\gamma},
\]
and so there is $0<\delta_2<\delta_0$ such that $g(\delta_2/2)<1$ and
\[
\frac{1}{\gamma}-\frac{1}{2}\frac{1}{\gamma}\leq  \frac{\log \Gamma_1(x)}{\log x} \leq \frac{1}{\gamma}+\frac{1}{2}\frac{1}{\gamma}, \quad x\in(0,g(\delta_2/2)].
\]
Let $\delta_3>0$ be so small that $g(\delta_3/2)<1$. Now, let $\delta=\min(\delta_0,\delta_1,\delta_2,\delta_3)$. Since
$\delta\leq \delta_2$, we have $\delta/2\leq \delta_2/2<\delta_0/2<\delta_0$. Therefore $g(\delta/2)\leq g(\delta_2/2)$ and so we have
\begin{gather}
\label{eq.estdelta2}
0<B_1:=\frac{1}{\gamma}-\frac{1}{2}\frac{1}{\gamma}\leq
\frac{\log \Gamma_1(x)}{\log x} \leq \frac{1}{\gamma}+\frac{1}{2}\frac{1}{\gamma}=:B_2,\,x\in(0,g(\delta/2)], \\
\label{eq.estdelta3}
g(\delta/2)<1, \quad g'(x)>0 \quad x\in(0,\delta).
\end{gather}

Since $x(t)\to 0$ as $t\to\infty$, there exists $T_0>0$ such that $x(t)\leq \delta/2$ for all $t\geq T_0$.

By \eqref{eq.t1} and \eqref{eq.t3} for each $\epsilon\in(0,1)$, there is a
$T_1(\epsilon)>0$ such that
\begin{equation} \label{eq.defT1}
\int_{t-\tau(t)}^t \frac{1}{\sigma(s)}\,ds \leq 1+\epsilon, \quad
t\geq T_1(\epsilon).
\end{equation}

Now let $0<c_2<c_2(\epsilon)$, where
$c_2(\epsilon)\in(0,\log(a/b))$ is the solution of
$g_\epsilon(c)=0$, where the function $g_\epsilon$ is defined by
\[
g_\epsilon(x) = -x\epsilon + a - be^{x(1+\epsilon)}, \quad x\geq
0.
\]
The existence and uniqueness of the solution are guaranteed by the
fact that $g_\epsilon(0)>0$, $g_{\epsilon}(\log(a/b))<0$ and
$g_\epsilon$ is continuous and decreasing on $(0,\infty)$. Note
moreover that as $\epsilon\downarrow 0$, we have that
$c_2(\epsilon)\to \log(a/b)$.

Since $\sigma(t)/t\to\infty$ as $t\to\infty$, by Lemma~\ref{lemma.intlogsigma0} we have that
\[
\lim_{t\to\infty} \frac{\int_0^t \frac{1}{\sigma(s)}\,ds}{\log \sigma(t)}=0.
\]
Hence for every $\epsilon\in(0,1)$, there exists $T_2(\epsilon)>0$ such that
\begin{equation} \label{eq.defT2}
\log \sigma(t) + B_2\log g(\delta/2) -B_2 c_2\int_0^t \frac{1}{\sigma(s)}\,ds > \log(1/\epsilon), \quad t\geq T_2(\epsilon).
\end{equation}

Since $x(t)\to 0$ as $t\to\infty$, there exists a $T_3(\epsilon)>0$ such that
\begin{equation} \label{eq.defT3}
g(x(t))\leq g(\delta/2)e^{-2c_2}, \quad t\geq T_3(\epsilon). 
\end{equation}
Let $T_4(\epsilon)=\max(T_0,T_1(\epsilon),T_2(\epsilon),T_3(\epsilon))$.
Since $t-\tau(t)\to\infty$ as $t\to\infty$, there exists $T_5(\epsilon)>T_4(\epsilon)$ such that $T_5(\epsilon)=\sup\{t>T_4(\epsilon):t-\tau(t)=T_4(\epsilon)\}$. Also $t-\tau(t)\geq T_4(\epsilon)$ for $t\geq T_5(\epsilon)$.

Define $x_2(\epsilon)$ by
\begin{equation} \label{def.x2}
x_2(\epsilon)= g(\delta/2)e^{c_2\int_0^{T_4(\epsilon)} \frac{1}{\sigma(s)}\,ds}.
\end{equation}
Then for $t\geq T_4(\epsilon)$ we have
\[
0<x_2(\epsilon)e^{-c_2\int_0^t \frac{1}{\sigma(s)}\,ds}= g(\delta/2)e^{c_2\int_0^{T_4(\epsilon)} \frac{1}{\sigma(s)}\,ds}e^{-c_2\int_0^t \frac{1}{\sigma(s)}\,ds} \leq g(\delta/2)<1.
\]
Since $\delta_0\leq \delta$, we have that $g^{-1}:[0,g(\delta)]\to[0,\delta]$. Therefore, we can define
\begin{equation} \label{def.xUlongmem}
x_{U,\epsilon}(t)=g^{-1}\left(x_2(\epsilon)e^{-c_2\int_0^t \frac{1}{\sigma(s)}\,ds}\right), \quad t\geq T_4(\epsilon).
\end{equation}
Therefore we have
\begin{equation*}
x_{U,\epsilon}(t)\leq \delta/2, \quad t\geq T_4(\epsilon); \quad g(x_{U,\epsilon}(t))<1, \quad t\geq T_4(\epsilon).
\end{equation*}
Next for $t\in[T_4(\epsilon),T_5(\epsilon)]$, noting that $T_5(\epsilon)-\tau(T_5(\epsilon))=T_4(\epsilon)$ and $T_4(\epsilon)\geq T_3(\epsilon)$, we have by \eqref{eq.defT3}, \eqref{def.x2} that
\begin{align*}
g(x(t))e^{c_2\int_0^t \frac{1}{\sigma(u)}\,du}
&\leq  g(x(t))e^{c_2\int_0^{T_5(\epsilon)} \frac{1}{\sigma(u)}\,du}\\
&\leq g(\delta/2)e^{-2c_2} e^{c_2\int_0^{T_5(\epsilon)} \frac{1}{\sigma(u)}\,du}\\
&=g(\delta/2)e^{-2c_2} e^{c_2\int_{0}^{T_5(\epsilon)-\tau(T_5(\epsilon))}\frac{1}{\sigma(u)}\,du}\cdot e^{c_2\int_{T_5(\epsilon)-\tau(T_5(\epsilon))}^{T_5(\epsilon)} \frac{1}{\sigma(u)}\,du}\\
&=g(\delta/2)e^{-2c_2} e^{c_2\int_{0}^{T_4(\epsilon)}\frac{1}{\sigma(u)}\,du}\cdot e^{c_2\int_{T_5(\epsilon)-\tau(T_5(\epsilon))}^{T_5(\epsilon)} \frac{1}{\sigma(u)}\,du}\\
&=x_2(\epsilon) e^{-2c_2} \cdot e^{c_2\int_{T_5(\epsilon)-\tau(T_5(\epsilon))}^{T_5(\epsilon)} \frac{1}{\sigma(u)}\,du}.
\end{align*}
Since $T_5(\epsilon)> T_4(\epsilon)\geq T_1(\epsilon)$, by\eqref{eq.defT1} for $t\in[T_4(\epsilon),T_5(\epsilon)]$ we have
\[
g(x(t))e^{c_2\int_0^t \frac{1}{\sigma(u)}\,du}\leq
x_2(\epsilon) e^{-2c_2} \cdot e^{c_2(1+\epsilon)}= x_2(\epsilon) e^{-c_2(1-\epsilon)}<x_2(\epsilon).
\]
Hence for $t\in[T_4(\epsilon),T_5(\epsilon)]$ we have
\[
g(x(t)) < x_2(\epsilon)e^{-c_2\int_0^t \frac{1}{\sigma(u)}\,du}=g(x_{U,\epsilon}(t))\leq g(\delta/2)<g(\delta)\leq g(\delta_0).
\]
Therefore we have
\begin{equation} \label{eq.xUlongmemic}
x(t)<x_{U,\epsilon}(t), \quad t\in[T_4(\epsilon),T_5(\epsilon)].
\end{equation}

Since $g\in C^1(0,\delta_0)$, and $\delta_1<g(\delta_0)$,
we have that $g^{-1}\in C^1(0,\delta_1)$. Now for $t\geq T_4(\epsilon)$ we have
$g(x_{U,\epsilon}(t))\leq \delta/2<\delta\leq \delta_1$, so therefore $x_{U,\epsilon}\in C^1[T_5(\epsilon),\infty)$
and
\[
x_{U,\epsilon}'(t)=-\frac{c_2}{\sigma(t)g'(x_{U,\epsilon}(t))}x_2(\epsilon)e^{-c_2\int_0^t \frac{1}{\sigma(s)}\,ds}, \quad t\geq T_5(\epsilon).
\]
For $t>T_5(\varepsilon)$, we have that $t-\tau(t)\geq T_4(\epsilon)$. Therefore we get
\begin{multline} \label{eq.xUdiffeqlongmem1}
x_{U,\epsilon}'(t) + ag(x_{U,\epsilon}(t)) -
bg(x_{U,\epsilon}(t-\tau(t)))
\\=x_2(\epsilon)e^{-c_2\int_{0}^t \frac{1}{\sigma(s)}\,ds}\left(
-\frac{c_2}{g'(x_{U,\epsilon}(t))\sigma(t)} + a - b e^{c_2\int^{t}_{t-\tau(t)}
\frac{1}{\sigma(s)}\,ds} \right), \quad t>T_5(\epsilon).
\end{multline}
Next we estimate $g'(x_{U,\epsilon}(t))\sigma(t)$ for $t\geq T_5(\epsilon)$. Since this quantity is positive, we may consider $\log(g'(x_{U,\epsilon}(t))\sigma(t))$.
Since \eqref{eq.estdelta2} holds
and $g$ is increasing on $(0,\delta/2)$ we have
\[
B_1\leq
\frac{\log \Gamma_1(g(x))}{\log g(x)} \leq B_2, \quad x\in(0,\delta/2],
\]
Since $t\geq T_5(\epsilon)>T_4(\epsilon)$, we have $x_{U,\epsilon}(t)\leq \delta/2$. Therefore
\[
0<B_1\leq
\frac{\log \Gamma_1(g(x_{U,\epsilon}(t)))}{\log g(x_{U,\epsilon}(t))} \leq B_2,\quad t\geq T_5(\epsilon).
\]
For $t\geq T_5(\epsilon)>T_4(\epsilon)$, we have that $g(x_{U,\epsilon}(t))\leq g(\delta/2)<1$, so $\log g(x_{U,\epsilon}(t))<0$.
Hence
\begin{equation} \label{eq.B2loggamma}
B_1\log g(x_{U,\epsilon}(t))\geq \frac{\log \Gamma_1(g(x_{U,\epsilon}(t)))}{\log g(x_{U,\epsilon}(t))}\log g(x_{U,\epsilon}(t))\geq B_2\log g(x_{U,\epsilon}(t)).
\end{equation}
Hence for $t\geq T_5(\epsilon)>T_4(\epsilon)$, by using \eqref{eq.B2loggamma}, \eqref{def.xUlongmem}, \eqref{def.x2} and \eqref{eq.defT2} in turn
we get
\begin{align*}
\log(g'(x_{U,\epsilon}(t))\sigma(t))
&=\log \sigma(t) + \frac{\log \Gamma_1(g(x_{U,\epsilon}(t)))}{\log g(x_{U,\epsilon}(t))}\log g(x_{U,\epsilon}(t))\\
&\geq \log \sigma(t) + B_2\log g(x_{U,\epsilon}(t))\\
&= \log \sigma(t) + B_2\left( \log x_2(\epsilon) -c_2\int_0^t \frac{1}{\sigma(s)}\,ds\right)\\
&=\log \sigma(t) + B_2\log g(\delta/2) -B_2 c_2\int_{T_4(\epsilon)}^t \frac{1}{\sigma(s)}\,ds\\
&>\log \sigma(t) + B_2\log g(\delta/2) -B_2 c_2\int_0^t \frac{1}{\sigma(s)}\,ds\\
&>\log(1/\epsilon).
\end{align*}
Therefore
\begin{equation} \label{eq.gsig}
g'(x_{U,\epsilon}(t))\sigma(t)>1/\epsilon, \quad t\geq T_5(\epsilon).
\end{equation}
Since $t>T_5(\epsilon)$ by inserting \eqref{eq.gsig} and \eqref{eq.defT1} into \eqref{eq.xUdiffeqlongmem1}
\[
x_{U,\epsilon}'(t) + ag(x_{U,\epsilon}(t)) -
bg(x_{U,\epsilon}(t-\tau(t)))
>x_2(\epsilon)e^{-c_2\int_{0}^t \frac{1}{\sigma(s)}\,ds}\left(
-\frac{c_2}{1/\epsilon} + a - b e^{c_2(1+\epsilon)} \right).
\]
The quantity in brackets is nothing other than $g_\epsilon(c_2)>0$. Thus we have that
\begin{align*}
x_{U,\epsilon}'(t)&>-ag(x_{U,\epsilon}(t)) + bg(x_{U,\epsilon}(t-\tau(t))), \quad t>T_5(\epsilon), \\
x_{U,\epsilon}(t)&> x(t)>0, \quad
t\in[T_4(\epsilon),T_5(\epsilon)].
\end{align*}
Therefore $x(t)<x_{U,\epsilon}(t)$ for all $t\geq T_4(\epsilon)$. Hence
\[
g(x(t)) < x_2(\epsilon) e^{-c_2\int_0^t \frac{1}{\sigma(s)}\,ds},
\quad t\geq T_4(\epsilon).
\]
Thus
\[
\limsup_{t\to\infty} \frac{\log g(x(t))}{\int_0^t
\frac{1}{\sigma(s)}\,ds} \leq -c_2.
\]
Letting $c_2\uparrow c_2(\epsilon)$, we get
\[
\limsup_{t\to\infty} \frac{\log g(x(t))}{\int_0^t
\frac{1}{\sigma(s)}\,ds} \leq -c_2(\epsilon).
\]
Finally, because letting $\epsilon\downarrow0$ yields
$c_2(\epsilon)\to\log (a/b)$, by taking the limit as
$\epsilon\to0$, we get
\begin{equation} \label{eq.b2}
\limsup_{t\to\infty} \frac{\log g(x(t))}{\int_0^t
\frac{1}{\sigma(s)}\,ds} \leq -\log\left( \frac{a}{b}\right).
\end{equation}

We now determine a lower bound for $x$.
By \eqref{eq.t1} and \eqref{eq.t3}, for each $\epsilon\in(0,1)$ there is $T_6(\epsilon)>0$ such that
\[
\int_{t-\tau(t)}^t \frac{1}{\sigma(s)}\,ds  \geq 1-\epsilon, \quad
t\geq T_6(\epsilon).
\]
Also, let $c_1>c_1(\epsilon)=(1-\epsilon)^{-1} \log(a/b)>0$. Note that the definition of $T_0$ gives $x(t)\leq \delta/2$ for all $t\geq T_0$.
Define $T_7(\epsilon)=\max(T_0,T_6(\epsilon))$. Since $t-\tau(t)\to\infty$ as $t\to\infty$ we have that there is $T_8(\epsilon)>T_7(\epsilon)$
such that $T_8(\epsilon)=\sup\{t>T_7(\epsilon):t-\tau(t)=T_7(\epsilon)\}$. Then $t-\tau(t)\geq T_7(\epsilon)$ for all $t\geq T_8(\epsilon)$.
Next define
\[
x_1(\epsilon) = \frac{1}{2}\min\left\{\min_{T_7(\epsilon)\leq s\leq T_8(\epsilon)} g(x(s))e^{c_1\int_0^s\frac{1}{\sigma(u)}\,du},g(\delta/2)e^{c_1\int_{0}^{T_7(\epsilon)} \frac{1}{\sigma(u)}\,du}\right\}
\]
so that
\[
x_1(\epsilon)<\min_{T_7(\epsilon)\leq s\leq T_8(\epsilon)} g(x(s))e^{c_1\int_0^s\frac{1}{\sigma(u)}\,du},
\quad x_1(\epsilon)<g(\delta/2)e^{c_1\int_0^{T_7(\epsilon)} \frac{1}{\sigma(u)}\,du}.
\]
For $t\geq T_7(\epsilon)$ we have
\[
0<x_1(\epsilon)e^{-c_1\int_0^t \frac{1}{\sigma(s)}\,ds}\leq x_1(\epsilon)e^{-c_1\int_0^{T_7(\epsilon)} \frac{1}{\sigma(s)}\,ds}
<
g(\delta/2)<g(\delta).
\]
Since $g$ is increasing on $(0,\delta)$, we may define $g^{-1}:[0,g(\delta)]\to[0,\delta]$ and therefore the function $x_{L,\epsilon}$ given by
\[
x_{L,\epsilon}(t) = g^{-1}\left(x_1(\epsilon)e^{-c_1\int_0^t
\frac{1}{\sigma(s)}\,ds}\right), \quad t\geq T_7(\epsilon)
\]
is well--defined.
Then $x_{L,\epsilon}(t)<\delta$ for all $t\geq T_7(\epsilon)$ and
\[
g(x_{L,\epsilon}(t)) = x_1(\epsilon)e^{-c_1\int_0^t
\frac{1}{\sigma(s)}\,ds}, \quad t\geq T_7(\epsilon).
\]
Then for $t\in[T_7(\epsilon),T_8(\epsilon)]$ we have
\begin{align*}
g(x_{L,\epsilon}(t))e^{c_1\int_0^t \frac{1}{\sigma(s)}\,ds}
&= x_1(\epsilon)\\
&<\min_{T_7(\epsilon)\leq s\leq T_8(\epsilon)} g(x(s))e^{c_1\int_0^s\frac{1}{\sigma(u)}\,du}\\
&\leq g(x(t))e^{c_1\int_0^t\frac{1}{\sigma(u)}\,du}.
\end{align*}
Thus $g(x_{L,\epsilon}(t))<g(x(t))$ for $t\in[T_7(\epsilon),T_8(\epsilon)]$, so therefore $x_{L,\epsilon}(t)<x(t)$ for
$t\in[T_7(\epsilon),T_8(\epsilon)]$. To see this, suppose to the contrary that there exists $t_1\in[T_7(\epsilon),T_8(\epsilon)]$
such that $x_{L,\epsilon}(t_1)\geq x(t_1)$. Since $x_{L,\epsilon}(t_1)<\delta$ we have  $g(x(t_1))\leq g(x_{L,\epsilon}(t_1))<g(\delta)$.
Since we must also have $g(x_{L,\epsilon}(t_1))<g(x(t_1))$ there is a contradiction; hence $x_{L,\epsilon}(t)<x(t)$ for
$t\in[T_7(\epsilon),T_8(\epsilon)]$.

Since $g$ is in $C^1(0,\delta)$ and $\sigma$ is continuous $x_{L,\epsilon}$ is in $C^1(T_7(\epsilon),\infty)$ and moreover
\[
g'(x_{L,\epsilon}(t))x_{L,\epsilon}'(t)=-\frac{c_1}{\sigma(t)}x_1(\epsilon)e^{-c_1\int_0^t
\frac{1}{\sigma(s)}\,ds}.
\]
Since $x_{L,\epsilon}(t)\in(0,\delta)$ we have $x_{L,\epsilon}'(t)<0$ for all $t>T_7(\epsilon)$. Moreover, because $t-\tau(t)\geq T_7(\epsilon)$ for all $t>T_8(\epsilon)$, we have
\begin{multline*}
x_{L,\epsilon}'(t) + ag(x_{L,\epsilon}(t)) -
bg(x_{L,\epsilon}(t-\tau(t))) \\
=x_1(\epsilon)e^{-c_1\int_0^t
\frac{1}{\sigma(s)}\,ds}\biggl(
-\frac{1}{g'(x_{L,\epsilon}(t))}\frac{c_1}{\sigma(t)}+a
-be^{c_1\int_{t-\tau(t)}^t
\frac{1}{\sigma(s)}\,ds}\biggr).
\end{multline*}
For $t\geq T_8(\epsilon)\geq T_6(\epsilon)$, we have
\[
e^{c_1\int_{t-\tau(t)}^t \frac{1}{\sigma(s)}\,ds}\geq
e^{c_1(1-\epsilon)}.
\]
Thus for $t\geq T_8(\epsilon)$, as $g'(x_{L,\epsilon}(t))>0$ and $\sigma(t)>0$
\[
-\frac{1}{g'(x_{L,\epsilon}(t))}\frac{c_1}{\sigma(t)} + a - b e^{c_1\int^{t}_{t-\tau(t)}
\frac{1}{\sigma(s)}\,ds} < a-be^{c_1(1-\epsilon)}<0.
\]
Thus
\begin{gather*}
x_{L,\epsilon}'(t) < -ag(x_{L,\epsilon}(t)) +
bg(x_{L,\epsilon}(t-\tau(t))), \quad t>T_8(\epsilon), \\
x_{L,\epsilon}(t)<x(t), \quad t\in [T_7(\epsilon),T_8(\epsilon)].
\end{gather*}
As $g$ is increasing on $(0,\delta)$ and $x(t-\tau(t))$ and $x_{U,\epsilon}(t-\tau(t))$ are both in $(0,\delta)$ for $t\geq T_8(\epsilon)$, a standard comparison argument now shows that $x(t)>x_{L,\epsilon}(t)$ for all $t\geq T_7(\epsilon)$, which gives
\[
g(x(t)) > g(x_{L,\epsilon}(t))=x_1(\epsilon)e^{-c_1\int_0^t \frac{1}{\sigma(s)}\,ds},
\quad t\geq T_7(\epsilon).
\]
Since $\int_0^t \sigma(s)^{-1}\,ds \to\infty$ as $t\to\infty$,
\[
\liminf_{t\to\infty} \frac{\log g(x(t))}{\int_0^t
\frac{1}{\sigma(s)}\,ds} \geq -c_1.
\]
Hence, letting $c_1\uparrow c_1(\epsilon)$, we get
\[
\liminf_{t\to\infty} \frac{\log g(x(t))}{\int_0^t
\frac{1}{\sigma(s)}\,ds} \geq -\frac{1}{1-\epsilon}
\log\left(\frac{a}{b}\right).
\]
Letting $\epsilon\downarrow0$ yields
\begin{equation}  \label{eq.a2}
\liminf_{t\to\infty}\frac{\log g(x(t))}{\int_0^t
\frac{1}{\sigma(s)}\,ds} \geq -\log\left(\frac{a}{b}\right).
\end{equation}
Combining \eqref{eq.a2} and \eqref{eq.b2} gives \eqref{Cgen}.
\end{proof}

\section{Proof of Results for Sublinear Delay} \label{sec.proofs}

\subsection{Proof of Theorem~\ref{thm.odedecay}}
Since $g\in \text{RV}_0(\beta)$ for $\beta>1$, $1/g\in \text{RV}_0(-\beta)$, so $G_0\in\text{RV}_0(1-\beta)$. Since $\beta>1$ we have
$G_0^{-1}\in\text{RV}_\infty(-1/(\beta-1))$. Since $-1/(\beta-1)<0$ and $g\in\text{RV}_0(\beta)$ we have
$\Gamma\in \text{RV}_\infty(-\beta/(\beta-1))$. Therefore $g$ obeys \eqref{eq:g4alt4} with $\gamma=\beta/(\beta-1)>1$, and so  Theorem~\ref{thm.odedecaygen} applies. Therefore \eqref{eq.GxgtLambdat} and \eqref{eq.Gxleqat} hold, so we are left to prove \eqref{eq.odedecay}.

Since $t-\tau(t)\to\infty$ as $t\to\infty$, there exists $T_0>1$
such that $t-\tau(t)\geq 1$ for all $t\geq T_0$. Recalling that $x'(t)\geq -ag(x(t))$ for all $t\geq 0$, for all $t\geq T_0$ we have
\[
G_0(x(t-\tau(t)))-G_0(x(t))=\int_{x(t-\tau(t))}^{x(t)} \frac{1}{g(u)}\,du = \int_{t-\tau(t)}^t \frac{x'(s)}{g(x(s))}\,ds \geq -a\tau(t),
\]
so
\begin{equation} \label{eq.G0tG0ttau}
G_0(x(t))\leq a\tau(t)+G_0(x(t-\tau(t))), \quad t\geq T_0.
\end{equation}
Define $G_1(x)=1/G_0(x)$ for $x>0$. Then $G_1\in \text{RV}_0(\beta-1)$ and $G_1$ is increasing. Since $\beta>1$ we have that  $G_1^{-1}\in\text{RV}_0(1/(\beta-1))$. Rearranging \eqref{eq.G0tG0ttau} we have
\[
\frac{G_1(x(t-\tau(t)))}{G_1(x(t))}\leq 1+a\tau(t)G_1(x(t-\tau(t))), \quad t\geq T_0.
\]
Therefore for $t\geq T_0$ we have
\[
\frac{G_1(x(t-\tau(t)))}{G_1(x(t))}\leq 1+a\frac{\tau(t)}{t}\cdot\frac{t}{t-\tau(t)}\cdot\frac{t-\tau(t)}{G_1(x(t-\tau(t)))}.
\]
Since $\tau(t)/t\to 0$ as $t\to\infty$ and \eqref{eq.GxgtLambdat} holds we have
\begin{equation}\label{eq.G1limsup1}
\limsup_{t\to\infty} \frac{G_1(x(t-\tau(t)))}{G_1(x(t))}\leq 1.
\end{equation}
Therefore for every $\epsilon\in(0,1)$ there is $T_1(\epsilon)>0$ such that
\[
G_1(x(t-\tau(t)))\leq (1+\epsilon)G_1(x(t)), \quad t\geq T_1(\epsilon).
\]
Therefore as $G_1^{-1}$ is increasing, we have
\[
\frac{x(t-\tau(t))}{x(t)}\leq \frac{G_1^{-1}((1+\epsilon)G_1(x(t)))}{G_1^{-1}(G_1(x(t)))}, \quad t\geq T_1(\epsilon).
\]
Now $G_1(x(t))\to 0$ as $t\to\infty$ and $G_1^{-1}\in \text{RV}_0(1/(\beta-1))$, so
\[
\limsup_{t\to\infty}
\frac{x(t-\tau(t))}{x(t)}
\leq
\lim_{t\to\infty} \frac{G_1^{-1}((1+\epsilon)G_1(x(t)))}{G_1^{-1}(G_1(x(t)))}
=
(1+\epsilon)^{1/(\beta-1)}.
\]
Letting $\epsilon\to 0^+$ gives
\[
\limsup_{t\to\infty} \frac{x(t-\tau(t))}{x(t)}\leq 1.
\]
Therefore there exists $T_2(\epsilon)>0$ such that $x(t-\tau(t))\leq (1+\epsilon)x(t)$ for all $t\geq T_2(\epsilon)$.
Since $g$ is increasing $g(x(t-\tau(t)))\leq g((1+\epsilon)x(t))$ for all $t\geq T_2(\epsilon)$. Therefore as $g\in\text{RV}_0(\beta)$ and
$x(t)\to 0$ as $t\to\infty$ we have
\[
\limsup_{t\to\infty} \frac{g(x(t-\tau(t)))}{g(x(t))}\leq \limsup_{t\to\infty}\frac{g((1+\epsilon)x(t))}{g(x(t))}=(1+\epsilon)^\beta.
\]
Letting $\epsilon\to 0^+$ we have
\[
\limsup_{t\to\infty} \frac{g(x(t-\tau(t)))}{g(x(t))}\leq 1,
\]
which implies
\[
\limsup_{t\to\infty} \frac{x'(t)}{g(x(t))}=\limsup_{t\to\infty} -a + b \frac{g(x(t-\tau(t)))}{g(x(t))}\leq -(a-b).
\]
Therefore, for every $\epsilon\in(0,a-b)$ there exists $T_3(\epsilon)>0$ such that
\[
\frac{x'(t)}{g(x(t))}\leq -(a-b)+\epsilon, \quad t\geq T_3(\epsilon).
\]
Since $t-\tau(t)\to\infty$ as $t\to\infty$, there is $T_4(\epsilon)>T_3(\epsilon)$ such that $t-\tau(t)\geq T_3(\epsilon)$
for all $t\geq T_4(\epsilon)$. Therefore for $t\geq T_4(\epsilon)$ we have
\[
G_0(x(t-\tau(t)))-G_0(x(t))=
\int_{t-\tau(t)}^t
\frac{x'(s)}{g(x(s))}\,ds \leq \left\{-(a-b)+\epsilon\right\}\tau(t).
\]
Therefore for $t\geq T_4(\epsilon)$ we have
$G_0(x(t-\tau(t))) +\left\{(a-b)-\epsilon\right\}\tau(t) \leq G_0(x(t))$.
Since $G_1=1/G_0$ we have
\[
\frac{G_1(x(t-\tau(t)))}{G_1(x(t))}
\geq
1 +\left\{(a-b)-\epsilon\right\}\tau(t)G_1(x(t-\tau(t))), \quad t\geq T_4(\epsilon).
\]
Since $G_1(x)>0$ for all $x<x(0)$ and $\tau(t)\geq 0$ for all $t\geq 0$ we have
\begin{equation} \label{eq.G1liminf1}
\liminf_{t\to\infty} \frac{G_1(x(t-\tau(t)))}{G_1(x(t))}\geq 1.
\end{equation}
Combining \eqref{eq.G1limsup1} and \eqref{eq.G1liminf1} we have
\[
\lim_{t\to\infty} \frac{G_1(x(t-\tau(t)))}{G_1(x(t))}=1.
\]
Proceeding as above, we can show that this implies
\[
\lim_{t\to\infty} \frac{g(x(t-\tau(t)))}{g(x(t))}=1.
\]
Therefore we have
\[
\lim_{t\to\infty} \frac{x'(t)}{g(x(t))}=-(a-b).
\]
Hence for every $\epsilon\in(0,a-b)$ there is $T_5(\epsilon)>0$ such that
\[
-(a-b)-\epsilon\leq \frac{x'(t)}{g(x(t))}\leq -(a-b)+\epsilon, \quad t\geq T_5(\epsilon).
\]
Therefore for $t\geq T_5(\epsilon)$ we have
\[
\left\{-(a-b)-\epsilon\right\}(t-T_5(\epsilon))\leq \int_{T_5(\epsilon)}^t \frac{x'(s)}{g(x(s))}\,ds   \leq  \left\{-(a-b)+\epsilon\right\}(t-T_5(\epsilon)),
\]
so for $t\geq T_5(\epsilon)$
\[
\left\{-(a-b)-\epsilon\right\}(t-T_5(\epsilon))\leq G_0(x(T_5(\epsilon)))-G_0(x(t))  \leq  \left\{-(a-b)+\epsilon\right\}(t-T_5(\epsilon)),
\]
from which \eqref{eq.odedecay} follows, by first taking limits as $t\to\infty$, and then letting $\epsilon\to 0^+$.

\section{Proof of Results for Proportional Delay}
\subsection{Proof of Theorem~\ref{thm.tauproptnotlikeode}}
We employ the following result in the proof of Theorem~\ref{thm.tauproptnotlikeode}.
\begin{lemma}  \label{lemma.hfgrv}
Suppose $f(t)/g(t)\to L>0$ as $t\to\infty$
\begin{itemize}
\item[(i)] If $g(t)\to\infty$ as $t\to\infty$ and $h\in \text{RV}_\infty(\gamma)$ is decreasing, then
\begin{equation}  \label{eq.hfgrv}
\lim_{t\to\infty} \frac{h(f(t))}{h(g(t))} = L^\gamma.
\end{equation}
\item[(ii)] If $g(t)\to 0$ as $t\to\infty$ and $h\in \text{RV}_0(\gamma)$ is increasing, then \eqref{eq.hfgrv} still holds.
\end{itemize}
\end{lemma}
\begin{proof}
For part (i) since $h$ is decreasing, for every $\epsilon\in(0,1)$ there is a $T(\epsilon)>0$ such that
\[
h(Lg(t)(1-\epsilon))>h(f(t))>h(Lg(t)(1+\epsilon)), \quad t>T(\epsilon).
\]
Now as $g(t)\to \infty$ as $t\to\infty$ and $h\in \text{RV}_\infty(\gamma)$ we have
\[
\{L(1-\epsilon)\}^\gamma
=
\lim_{t\to\infty}
\frac{h(Lg(t)(1-\epsilon))}{h(g(t))}
\geq
\limsup_{t\to\infty} \frac{h(f(t))}{h(g(t))},
\]
so
\[
\limsup_{t\to\infty} \frac{h(f(t))}{h(g(t))} \leq L^\gamma.
\]
Similarly, we have
\[
\liminf_{t\to\infty} \frac{h(f(t))}{h(g(t))} \geq L^\gamma,
\]
whence \eqref{eq.hfgrv}.

For part (ii) since $h$ is increasing, for every $\epsilon\in(0,1)$ there is a $T(\epsilon)>0$ such that
\[
h(Lg(t)(1-\epsilon))<h(f(t))<h(Lg(t)(1+\epsilon)), \quad t>T(\epsilon).
\]
Now as $g(t)\to 0$ as $t\to\infty$ and $h\in \text{RV}_0(\gamma)$ we have
\[
\{L(1-\epsilon)\}^\gamma
=
\lim_{t\to\infty}
\frac{h(Lg(t)(1-\epsilon))}{h(g(t))}
\geq
\limsup_{t\to\infty} \frac{h(f(t))}{h(g(t))},
\]
so
\[
\limsup_{t\to\infty} \frac{h(f(t))}{h(g(t))} \leq L^\gamma.
\]
Similarly, we have
\[
\liminf_{t\to\infty} \frac{h(f(t))}{h(g(t))} \geq L^\gamma,
\]
whence \eqref{eq.hfgrv}.
\end{proof}

\begin{proof}[Proof of Theorem~\ref{thm.tauproptnotlikeode}]
Suppose that $G(x(t))/t\to a-b$ as $t\to\infty$. Therefore
\[
\lim_{t\to\infty} \frac{G(x(t-\tau(t)))}{t}
=\lim_{t\to\infty}  \frac{G(x(t-\tau(t)))}{t-\tau(t)}\cdot \frac{t-\tau(t)}{t}=(a-b)(1-q).
\]
Since $g\in \text{RV}_0(\beta)$, we have $1/g\in \text{RV}_0(-\beta)$, so $G\in \text{RV}_0(1-\beta)$.
Thus $G^{-1}\in \text{RV}_\infty(-1/(\beta-1))$ and moreover $G^{-1}$ is decreasing.
Therefore by Lemma~\ref{lemma.hfgrv} part (i) we get
\begin{equation} \label{eq.xGinvcontra}
\lim_{t\to\infty} \frac{x(t)}{G^{-1}(t)}=(a-b)^{-1/(\beta-1)},
\end{equation}
and
\[
\lim_{t\to\infty} \frac{x(t-\tau(t))}{G^{-1}(t)}=\left((a-b)(1-q)\right)^{-1/(\beta-1)}.
\]
Since $g$ is increasing, and $g\in\text{RV}_0(\beta)$, by part (ii) of Lemma~\ref{lemma.hfgrv} we have
\[
\lim_{t\to\infty} \frac{g(x(t))}{g(G^{-1}(t))}=(a-b)^{-\beta/(\beta-1)}, \quad \lim_{t\to\infty} \frac{g(x(t-\tau(t)))}{g(G^{-1}(t))}=\left((a-b)(1-q)\right)^{-\beta/(\beta-1)}.
\]
Hence
\begin{equation*}
\lim_{t\to\infty} \frac{-ag(x(t))+bg(x(t-\tau(t)))}{g(G^{-1}(t))}
=-a(a-b)^{-\beta/(\beta-1)}+ b\left((a-b)(1-q)\right)^{-\beta/(\beta-1)}.
\end{equation*}
This implies
\begin{equation} \label{eq.xprgG}
\lim_{t\to\infty} \frac{x'(t)}{g(G^{-1}(t))}
=-(a-b)^{-\beta/(\beta-1)}\left(a - b(1-q)^{-\beta/(\beta-1)}\right).
\end{equation}
Now $g\circ G^{-1} \in \text{RV}_\infty(-\beta/(\beta-1))$. Therefore
\[
\lim_{t\to\infty}
\frac{\log |(g\circ G^{-1})(t)|}{\log t}=-\frac{\beta}{\beta-1}.
\]
Since $\beta>1$, $-\beta/(\beta-1)<-1$. Therefore $g\circ G^{-1}\in L^1(0,\infty)$; since $x(t)\to 0$ as $t\to\infty$, by  L'H\^{o}pital's rule
and \eqref{eq.xprgG} we get
\[
\lim_{t\to\infty} \frac{x(t)}{\int_t^\infty g(G^{-1}(s))\,ds}
=(a-b)^{-\beta/(\beta-1)}\left(a - b(1-q)^{-\beta/(\beta-1)}\right).
\]
Since $G(0)=+\infty$ we have
\[
\int_t^\infty g(G^{-1}(s))\,ds=\int_{G^{-1}(t)}^{G^{-1}(\infty)} g(u) G'(u)\,du  
=\int_0^{G^{-1}(t)} g(u)\frac{1}{g(u)}\,du=G^{-1}(t).
\]
Therefore
\[
\lim_{t\to\infty} \frac{x(t)}{G^{-1}(t)}
=(a-b)^{-\beta/(\beta-1)}\left(a - b(1-q)^{-\beta/(\beta-1)}\right).
\]
Now \eqref{eq.xGinvcontra} implies $x(t)/G^{-1}(t)\to (a-b)^{-1/(\beta-1)}$ as $t\to\infty$, so we must have
\[
(a-b)^{-1/(\beta-1)}
=(a-b)^{-\beta/(\beta-1)}\left(a - b(1-q)^{-\beta/(\beta-1)}\right).
\]
Since $a>b$ this implies
\[
a-b
=a - b(1-q)^{-\beta/(\beta-1)},
\]
and as $b>0$ we have $1= (1-q)^{-\beta/(\beta-1)}$, which implies $1-q=1$ or $q=0$, which contradicts
\eqref{eq.tauasytoqt}.
\end{proof}

\section{Proof of Theorem~\ref{thm.odedecay2}}
We are going to prove this result in two parts: first, we will show that there exists $\Lambda_1>0$ such that
\begin{equation}  \label{eq.GxgtLambdata}
at\geq G_0(x(t))\geq \Lambda_1 t, \quad t\geq1.
\end{equation}
where $G_0$ is defined by \eqref{def.G0}. From this the result
\[
\liminf_{t\to\infty} \frac{x(t)}{G^{-1}(t)}\leq \limsup_{t\to\infty} \frac{x(t)}{G^{-1}(t)}<+\infty
\]
holds. Then we will use this estimate to show that
\begin{equation}  \label{eq.xpropdelayasy1liminf}
\liminf_{t\to\infty} \frac{x(t)}{G^{-1}(t)}\geq \Lambda
\end{equation}
where $\Lambda$ is given by \eqref{def.Lambdalim}.

We first note that \eqref{eq.GxgtLambdata} holds.
Define $G_0$ by \eqref{def.G0}; once again we note that $G_0$ is decreasing on $(0,\infty)$ is therefore invertible. We may also deduce
\eqref{eq.Gxleqat} viz.,
\begin{equation*}
G_0(x(t))\leq at, \quad t\geq 0.
\end{equation*}
As before define $\Gamma:[0,\infty)\to (0,\infty)$ by \eqref{def.Gamma}. Then $\Gamma\in \text{RV}_\infty(-\beta/(\beta-1))$, and
so $g$ obeys \eqref{eq:g4} (and \eqref{eq:g4alt4} in particular) with $\gamma=\beta/(\beta-1)$.  Since $a$, $b$, $\beta$, $q$ obey
\eqref{eq.abqsmall} and $q\in[0,1)$, they obey \eqref{eq.abqsmallsuperhyp}. Since $\tau$ obeys \eqref{eq.tauasytoqt}, all
the hypotheses of Theorem~\ref{thm.odedecaygen} hold, so we have that \eqref{eq.GxgtLambdata}.

\subsection{Proof of \eqref{eq.xpropdelayasy1liminf}}
We show when $a$, $b$, $q$ and $\beta$ obey \eqref{eq.abqsmall}, and $\tau$ obeys \eqref{eq.tauasytoqt} that we can to obtain the limit
\begin{equation*}
\liminf_{t\to\infty} \frac{x(t)}{G^{-1}(t)} \geq \Lambda,
\end{equation*}
where $\Lambda$ is defined by \eqref{def.Lambdalim}.
To do this we first need two preliminary lemmata.

\begin{lemma} \label{lemma.existlambda}
Let $a>b>0$, $\beta>1$, $q\in(0,1)$ obey \eqref{eq.abqsmall}, and let $\Lambda$ be given by \eqref{def.Lambdalim}.
Define $\lambda_1=a^{-1/(\beta-1)}$. Then there is a sequence $(\lambda_n)_{n\geq 1}$
defined by
\begin{equation}  \label{def.lambdaseq}
a\lambda_{n+1}^\beta=\lambda_{n+1}+b\lambda_n^\beta(1-q)^{-\beta/(\beta-1)},  \quad n\geq 1
\end{equation}
such that $\lambda_n\in[a^{-1/(\beta-1)},\Lambda)$ for all $n\geq 1$, $(\lambda_n)_{n\geq 1}$ is increasing, and $\lambda_n\to\Lambda$ as $n\to\infty$.
\end{lemma}
\begin{proof}
Write the $n$-th level hypothesis as
\[
(\lambda_j)_{j\leq n}\text{ is well--defined by \eqref{def.lambdaseq} and increasing }\quad
\lambda_j\in [a^{-1/(\beta-1)},\Lambda), \quad j\leq n.
\]
We have that the first level hypothesis is true. Suppose the $n$-th level hypothesis holds. Define
\[
f_n(\lambda)=a\lambda^\beta-\lambda -b\lambda_n^\beta(1-q)^{-\beta/(\beta-1)}, \quad \lambda\in[0,\Lambda].
\]
Note that $a\Lambda^\beta-\Lambda -b\Lambda^\beta(1-q)^{-\beta/(\beta-1)}=0$. Therefore as $\lambda_n\in (0,\Lambda)$ by hypothesis
\[
f_n(\Lambda)=a\Lambda^\beta-\Lambda-b\lambda_n^\beta(1-q)^{-\beta/(\beta-1)}=b(1-q)^{-\beta/(\beta-1)}\left(\Lambda^\beta-\lambda_n^\beta\right)>0.
\]
Also
\begin{align*}
f_n(\lambda_n)&=a\lambda_n^\beta-\lambda_n-b\lambda_n(1-q)^{-\beta/(\beta-1)}=\lambda_n\left\{\left(a-b(1-q)^{-\beta/(\beta-1)}\right)\lambda_n^{\beta-1}-1\right\}\\
&=\lambda_n((\lambda_n/\Lambda)^{\beta-1}-1)<0.
\end{align*}
Therefore there exists $\lambda_{n+1}\in(\lambda_n,\Lambda)$ such that $f_n(\lambda_{n+1})=0$. Such a $\lambda_{n+1}$ obeys \eqref{def.lambdaseq}.
We now show that \eqref{def.lambdaseq} uniquely defines $\lambda_{n+1}\in(\lambda_n,\Lambda)$.
Since $f_n'(\lambda)=a\beta\lambda^{\beta-1}-1$ and $f_n''(\lambda)=a\beta(\beta-1)\lambda^{\beta-2}>0$, we have that $f_n'(\lambda)>f_n'(\lambda_n)=a\beta\lambda_n^{\beta-1}-1\geq \beta-1>0$, because
$\lambda_n\geq a^{-1/(\beta-1)}$. Since $f_n$ is increasing on $(\lambda_n,\Lambda)\subseteq(a^{-1/(\beta-1)},\Lambda)$, $\lambda_{n+1}$
is uniquely defined by \eqref{def.lambdaseq}. Therefore the $(n+1)$--th level hypothesis holds. Since $(\lambda_n)_{n\geq 1}$ is increasing
and bounded above by $\Lambda$, we have that $\lambda_n\to \Lambda_\ast\in(a^{-1/(\beta-1)},\Lambda]$. By \eqref{def.lambdaseq} we have
$a\Lambda_\ast^\beta-\Lambda_\ast -b\Lambda_\ast^\beta(1-q)^{-\beta/(\beta-1)}=0$, which implies $\Lambda_\ast=\Lambda$.
\end{proof}

\begin{lemma} \label{lemma.liminfphiGinv}
Let $g$ obey \eqref{eq:galt} and $G$ be given by \eqref{def.G}. Let $x$ be positive and continuous with $x(t)\to 0$ as $t\to\infty$ and suppose that $x$ obeys
\[
\liminf_{t\to\infty} \frac{x(t)}{G^{-1}(t)}\geq \lambda>0.
\]
If $\tau$ is continuous and obeys \eqref{eq.tauasytoqt} for some $q\in(0,1)$ and $\varphi$ is defined by
\begin{equation}  \label{def.varphi}
\varphi(t)=bg(x(t-\tau(t))), \quad t>0.
\end{equation}
If $b>0$, then
\[
\liminf_{t\to\infty} \frac{\varphi(t)}{G^{-1}(t)}\geq b\lambda^\beta (1-q)^{-\beta/(\beta-1)}.
\]
\end{lemma}
\begin{proof}
For every $\epsilon\in(0,1)$ there is a $T_1(\epsilon)>0$ such that $x(t)>\lambda(1-\epsilon)G^{-1}(t)$ for $t>T_1(\epsilon)$.
Since $t-\tau(t)\to\infty$ as $t\to\infty$ there exists $T_2(\epsilon)>T_1(\epsilon)$ such that $t-\tau(t)>T_1(\epsilon)$
for all $t>T_2(\epsilon)$. Hence $x(t-\tau(t))>\lambda(1-\epsilon)G^{-1}(t-\tau(t))$ for $t>T_2(\epsilon)$. Since $x(t)\to 0$
as $t\to\infty$, there exists $T_3>0$ such that $x(t-\tau(t))<\delta_1$ for all $t>T_3$. Also for every $\epsilon\in(0,1)$ there is
$T_4(\epsilon)>0$ such that $t-\tau(t)\leq (1+\epsilon)(1-q)t$ for all $t\geq T_4(\epsilon)$.
Since $G^{-1}\in \text{RV}_\infty(-1/(\beta-1))$, for every $\epsilon\in(0,1)$ there exists $T_5(\epsilon)>0$ such that
\[
G^{-1}((1+\epsilon)(1-q)t)>(1-\epsilon)\{(1+\epsilon)(1-q)\}^{-\beta/(\beta-1)}G^{-1}(t), \quad t>T_5(\epsilon).
\]
Let $T_6(\epsilon)=\max(T_2(\epsilon),T_3,T_4(\epsilon),T_5(\epsilon))$. Then for $t>T_6(\epsilon)$ we have
\[
\delta_1>x(t-\tau(t))>\lambda(1-\epsilon)G^{-1}(t-\tau(t))\geq \lambda(1-\epsilon)G^{-1}((1+\epsilon)(1-q)t).
\]
Therefore
\[
\delta_1>x(t-\tau(t))>\lambda(1-\epsilon)^2\{(1+\epsilon)(1-q)\}^{-1/(\beta-1)}G^{-1}(t), \quad t>T_6(\epsilon).
\]
Hence as $g$ is increasing on $[0,\delta_1]$ we have
\[
g(x(t-\tau(t)))>g(\lambda(1-\epsilon)^2\{(1+\epsilon)(1-q)\}^{-1/(\beta-1)}G^{-1}(t)), \quad t>T_6(\epsilon).
\]
Hence
\begin{align*}
\liminf_{t\to\infty} \frac{\varphi(t)}{g(G^{-1}(t))}
&= b
\liminf_{t\to\infty} \frac{g(x(t-\tau(t)))}{g(G^{-1}(t))}\\
&
\geq b
\liminf_{t\to\infty} \frac{g(\lambda(1-\epsilon)^2\{(1+\epsilon)(1-q)\}^{-1/(\beta-1)}G^{-1}(t))}{g(G^{-1}(t))}\\
&=b\lambda^\beta(1-\epsilon)^{2\beta}\{(1+\epsilon)(1-q)\}^{-\beta/(\beta-1)}.
\end{align*}
Therefore
\[
\liminf_{t\to\infty} \frac{\varphi(t)}{g(G^{-1}(t))}
\geq b\lambda^\beta (1-q)^{-\beta/(\beta-1)},
\]
as required.
\end{proof}

We are now in a position to prove \eqref{eq.xpropdelayasy1liminf}. To do this, it is important first to show that $\Lambda_1\in(0,\infty)$ defined by
\begin{equation} \label{eq.xpropdelaylimsup}
\Lambda_1=\limsup_{t\to\infty} \frac{x(t)}{G^{-1}(t)}
\end{equation}
obeys $\Lambda_1\geq \Lambda$. The argument used to prove this can then be adapted easily to prove \eqref{eq.xpropdelayasy1liminf}.
We formulate the desired result in the following proposition.

\begin{proposition} \label{prop.liminfsharpboundpropdelay}
Let $x$ be the solution of \eqref{T1.c} and suppose that all the hypotheses of Theorem~\ref{thm.odedecay2} hold. Then there exists a $\Lambda_1\in(0,\infty)$ such that \eqref{eq.xpropdelaylimsup}, and $\Lambda_1\geq\Lambda$  where $\Lambda$ is given by \eqref{def.Lambdalim}.
Moreover $x$ obeys \eqref{eq.xpropdelayasy1liminf}.
\end{proposition}
\begin{proof}
The existence of a $\Lambda_1\in(0,\infty)$ satisfying \eqref{eq.xpropdelaylimsup} is a consequence of \eqref{eq.GxgtLambdata}. We next show that $\Lambda_1\geq \Lambda$: the proof of
this will be by contradiction.

Define $\lambda_1=a^{-1/(\beta-1)}$ and $(\lambda_n)_{n\geq 1}$ by \eqref{def.lambdaseq}. Then $\lambda_n\to\Lambda$ as $n\to\infty$. Since
$G_0(x(t))\leq at$ for all $t\geq 0$, as $G_0^{-1} \in \text{RV}_\infty(-1/(\beta-1)$ and $G_0^{-1}$ is decreasing, we have
$x(t)\geq G_0^{-1}(at)$. Therefore
\[
\liminf_{t\to\infty} \frac{x(t)}{G_0^{-1}(t)}\geq \liminf_{t\to\infty} \frac{ G_0^{-1}(at)}{G_0^{-1}(t)}=a^{-1/(\beta-1)}=\lambda_1.
\]
Clearly we have $\Lambda_1\geq \lambda_1$. If $\Lambda_1=\lambda_1$, then
\[
\lim_{t\to\infty} \frac{x(t)}{G^{-1}(t)}=\lambda_1<\Lambda,
\]
which produces a contradiction. Therefore $\Lambda_1>\lambda_1$.

Suppose that $\Lambda_1<\Lambda$. By Lemma~\ref{lemma.existlambda} there is a minimal $n'>1$ such that $\Lambda>\lambda_{n'+1}>\Lambda_1$ but $\lambda_{n'}\leq \Lambda_1$.

We will show that the following statements are true:
\begin{itemize}
\item[(a)] If $\lambda_{n+1}\leq \Lambda_1$, and
\[
\liminf_{t\to\infty} \frac{x(t)}{G^{-1}(t)}\geq \lambda_n
\]
then
\[
\liminf_{t\to\infty} \frac{x(t)}{G^{-1}(t)}\geq \lambda_{n+1}
\]
\item[(b)] If $\lambda_{n+1}>\Lambda_1$, and
\[
\liminf_{t\to\infty} \frac{x(t)}{G^{-1}(t)}\geq \lambda_n
\]
then
\begin{equation} \label{eq.limlambda1}
\lim_{t\to\infty} \frac{x(t)}{G^{-1}(t)}=\Lambda_1.
\end{equation}
\end{itemize}
The consequence of these statements is that \eqref{eq.limlambda1} holds.

To see this, note that by using statement (a) successively, we have that
\[
\liminf_{t\to\infty} \frac{x(t)}{G^{-1}(t)}\geq \lambda_{n'}>\lambda_{n'-1}>\ldots>\lambda_1.
\]
Since $\lambda_{n'+1}>\Lambda_1$ by applying statement (b) we have \eqref{eq.limlambda1}.

The limit \eqref{eq.limlambda1} now leads to a contradiction, because if $\lim_{t\to\infty} x(t)/G^{-1}(t)$ is finite and positive, it must be $\Lambda$. Therefore $\Lambda_1=\Lambda$. But $\Lambda_1<\Lambda$ by hypothesis, so this contradiction forces $\Lambda_1\geq \Lambda$. It therefore remains to prove the inferences (a), (b).

\textbf{Proof of Statement (a).}
Let $\epsilon\in(0,1)$ be sufficiently small. Since $g\in\text{RV}_0(\beta)$ we have
\[
\lim_{x\to 0^+} \frac{g(\lambda_{n+1}(1-\epsilon)x)}{g(x)}=(\lambda_{n+1}(1-\epsilon))^\beta.
\]
Since $\beta>1$ there exists $x_1(\epsilon)>0$ such that $x\in(0,x_1(\epsilon))$ implies
\[
g(\lambda_{n+1}(1-\epsilon)x)<(1-\epsilon)^{-(\beta-1)/2}(1-\epsilon)^\beta \lambda_{n+1}^\beta g(x), \quad 0<x<x_1(\epsilon).
\]
Since $G^{-1}(t)\to 0$ as $t\to\infty$, there is $T_0(\epsilon)$ such that $G^{-1}(t)<x_1(\epsilon)$ for all $t>T_0(\epsilon)$. Thus
\begin{equation} \label{eq.gest1qliminf}
g(\lambda_{n+1}(1-\epsilon)G^{-1}(t))<(1-\epsilon)^{-(\beta-1)/2}(1-\epsilon)^\beta \lambda_{n+1}^\beta g(G^{-1}(t)), \quad t>T_0(\epsilon).
\end{equation}
With $\varphi$ defined by \eqref{def.varphi}, by Lemma~\ref{lemma.liminfphiGinv} we have
\[
\liminf_{t\to\infty} \frac{\varphi(t)}{g(G^{-1}(t))}\geq b\lambda_n^\beta(1-q)^{-\beta/(\beta-1)}.
\]
Therefore for every $\epsilon\in(0,1)$ there exists $T_1(\varepsilon)>0$ such that
\begin{equation} \label{eq.gest2qliminf}
\varphi(t)>b(1-\epsilon)^{(\beta+1)/2} \lambda_n^\beta(1-q)^{-\beta/(\beta-1)}g(G^{-1}(t)), \quad t>T_1(\epsilon).
\end{equation}
Let $T_2(\epsilon)=\max(T_0(\epsilon),T_1(\epsilon))$. Since $\limsup_{t\to\infty} x(t)/G^{-1}(t)=\Lambda_1$, there exists
$T_3(\epsilon)>T_2(\epsilon)$ such that
\begin{equation} \label{eq.gest3qliminf}
x(T_3(\epsilon))>\Lambda_1(1-\epsilon)G^{-1}(T_3(\epsilon)).
\end{equation}
Since $x$ obeys \eqref{T1.c}, by \eqref{eq.gest2qliminf}, for $t>T_3(\epsilon)$ we have
\begin{equation} \label{eq.gest4qliminf}
x'(t)>-ag(x(t))+b(1-\epsilon)^{(\beta+1)/2}\lambda_n^\beta (1-q)^{-\beta/(\beta-1)}g(G^{-1}(t)), \quad t>T_3(\epsilon).
\end{equation}
We define
\begin{equation} \label{def.xLqliminf}
x_{L,\epsilon}(t)=\lambda_{n+1}(1-\epsilon)G^{-1}(t), \quad t\geq T_3(\epsilon).
\end{equation}
Since $\lambda_{n+1}\leq \Lambda_1$ by \eqref{def.xLqliminf} and \eqref{eq.gest3qliminf} we have
\[
x_{L,\epsilon}(T_3(\epsilon))=\lambda_{n+1}(1-\epsilon)G^{-1}(T_3(\epsilon))\leq \Lambda_1(1-\epsilon)G^{-1}(T_3(\epsilon))<x(T_3(\epsilon)).
\]
By \eqref{def.xLqliminf} for $t>T_3(\epsilon)$ we have $G(x_{L,\epsilon}(t)/\lambda_{n+1}(1-\epsilon))=t$, so
\[
-\frac{x_{L,\epsilon}'(t)}{g(G^{-1}(t))}=G'(G^{-1}(t)) x_{L,\epsilon}'(t) = \lambda_{n+1}(1-\epsilon), \quad t>T_3(\epsilon).
\]
Hence $x_{L,\epsilon}'(t)=- \lambda_{n+1}(1-\epsilon)g(G^{-1}(t))$ for $t>T_3(\epsilon)$. Thus by \eqref{eq.gest1qliminf} and \eqref{def.lambdaseq}
we get
\begin{align*}
\lefteqn{x_{L,\epsilon}'(t)+ag(x_{L,\epsilon}(t))-b(1-\epsilon)^{(\beta+1)/2}\lambda_n^\beta (1-q)^{-\beta/(\beta-1)}g(G^{-1}(t))}\\
&=- \lambda_{n+1}(1-\epsilon)g(G^{-1}(t))+ag(\lambda_{n+1}(1-\epsilon)G^{-1}(t))\\
&\qquad\qquad-b(1-\epsilon)^{(\beta+1)/2}\lambda_n^\beta (1-q)^{-\beta/(\beta-1)}g(G^{-1}(t))\\
&<- \lambda_{n+1}(1-\epsilon)g(G^{-1}(t))
+a
(1-\epsilon)^{-(\beta-1)/2}(1-\epsilon)^\beta \lambda_{n+1}^\beta g(G^{-1}(t))\\
&\qquad\qquad-b(1-\epsilon)^{(\beta+1)/2}\lambda_n^\beta (1-q)^{-\beta/(\beta-1)}g(G^{-1}(t))\\
&=\left\{- \lambda_{n+1}(1-\epsilon)
+
(1-\epsilon)^{(\beta+1)/2} \left(a\lambda_{n+1}^\beta
-b\lambda_n^\beta (1-q)^{-\beta/(\beta-1)}\right)\right\}g(G^{-1}(t))\\
&=\lambda_{n+1}(1-\epsilon)\left\{- 1
+
(1-\epsilon)^{(\beta-1)/2} \right\}g(G^{-1}(t)).
\end{align*}
Since $\beta>1$, we have
\[
x_{L,\epsilon}'(t)+ag(x_{L,\epsilon}(t))-b(1-\epsilon)^{(\beta+1)/2}\lambda_n^\beta (1-q)^{-\beta/(\beta-1)}g(G^{-1}(t))<0, \quad t>T_3(\epsilon).
\]
Using this inequality, \eqref{eq.gest3qliminf} and \eqref{eq.gest4qliminf}, the comparison principle implies that
\[
x(t)>x_{L,\epsilon}(t), \quad t\geq T_3(\epsilon).
\]
Therefore
\[
\liminf_{t\to\infty} \frac{x(t)}{G^{-1}(t)}\geq \liminf_{t\to\infty} \frac{x_{L,\epsilon}(t)}{G^{-1}(t)}=\lambda_{n+1}(1-\epsilon).
\]
Letting $\epsilon\to 0^+$ gives us the required limit.

\textbf{Proof of Statement (b).}
Since $\Lambda_1<\lambda_{n+1}$, we have $f_n(\lambda_{n+1})=0>f_n(\Lambda_1)$, we have
\[
a\Lambda_1^\beta-\Lambda_1+b\lambda_n^\beta(1-q)^{-\beta/(\beta-1)}<0.
\]
Therefore we may find  $\epsilon\in(0,1)$ sufficiently small that
\begin{equation}   \label{eq.lambda1epscondn}
a\Lambda_1^\beta(1-\epsilon)^{(\beta+1)/2}-\Lambda_1(1-\epsilon)
+b\lambda_n^\beta(1-q)^{-\beta/(\beta-1)}(1-\epsilon)^{(\beta+1)/2}<0.
\end{equation}
Arguing as above for every  $\epsilon\in(0,1)$ there exists $T_5(\epsilon)>0$ such that
\begin{equation} \label{eq.gest5qliminf}
g(\Lambda_{1}(1-\epsilon)G^{-1}(t))<(1-\epsilon)^{-(\beta-1)/2}(1-\epsilon)^\beta \Lambda_{1}^\beta g(G^{-1}(t)), \quad t>T_5(\epsilon),
\end{equation}
and there is  $T_6(\epsilon)>0$ such that
\begin{equation} \label{eq.gest6qliminf}
\varphi(t)>b(1-\epsilon)^{(\beta+1)/2} \lambda_n^\beta(1-q)^{-\beta/(\beta-1)}g(G^{-1}(t)), \quad t>T_6(\epsilon).
\end{equation}
Let $T_7(\epsilon)=\max(T_5,T_6)$. Since $\limsup_{t\to\infty} x(t)/G^{-1}(t)=\Lambda_1$, there exists
$T_8(\epsilon)>T_7(\epsilon)$ such that
\begin{equation} \label{eq.gest8qliminf}
x(T_8(\epsilon))>\Lambda_1(1-\epsilon)G^{-1}(T_8(\epsilon)).
\end{equation}
We define
\begin{equation} \label{def.xLqliminfLambda1}
x_{L,\epsilon}(t)=\Lambda_{1}(1-\epsilon)G^{-1}(t), \quad t\geq T_8(\epsilon).
\end{equation}
By \eqref{def.xLqliminfLambda1} and \eqref{eq.gest8qliminf} we have
\begin{equation}\label{eq.xLqliminfic}
x_{L,\epsilon}(T_8(\epsilon))=\Lambda_{1}(1-\epsilon)G^{-1}(T_8(\epsilon))<x(T_8(\epsilon)).
\end{equation}
By \eqref{def.xLqliminfLambda1} for $t>T_8(\epsilon)$ we have $x_{L,\epsilon}'(t)=- \Lambda_{1}(1-\epsilon)g(G^{-1}(t))$ for $t>T_8(\epsilon)$.
Therefore by \eqref{eq.gest5qliminf}, for $t>T_8(\epsilon)$ we have
\begin{align*}
\lefteqn{x_{L,\epsilon}'(t)+ag(x_{L,\epsilon}(t))-b(1-\epsilon)^{(\beta+1)/2}\lambda_n^\beta (1-q)^{-\beta/(\beta-1)}g(G^{-1}(t))}\\
&=- \Lambda_{1}(1-\epsilon)g(G^{-1}(t))+ag(\Lambda_{1}(1-\epsilon)G^{-1}(t))\\
&\qquad\qquad-b(1-\epsilon)^{(\beta+1)/2}\lambda_n^\beta (1-q)^{-\beta/(\beta-1)}g(G^{-1}(t))\\
&<- \Lambda_{1}(1-\epsilon)g(G^{-1}(t))
+a(1-\epsilon)^{-(\beta-1)/2}(1-\epsilon)^\beta \Lambda_{1}^\beta g(G^{-1}(t))\\
&\qquad\qquad-b(1-\epsilon)^{(\beta+1)/2}\lambda_n^\beta (1-q)^{-\beta/(\beta-1)}g(G^{-1}(t))\\
&=\left\{- \Lambda_{1}(1-\epsilon)
+
(1-\epsilon)^{(\beta+1)/2} \left(a\Lambda_{1}^\beta
-b\lambda_n^\beta (1-q)^{-\beta/(\beta-1)}\right)\right\}g(G^{-1}(t))\\
&<0,
\end{align*}
where we have used \eqref{eq.lambda1epscondn} at the last step. Therefore we have
\[
x_{L,\epsilon}'(t)<-ag(x_{L,\epsilon}(t))+b(1-\epsilon)^{(\beta+1)/2}\lambda_n^\beta (1-q)^{-\beta/(\beta-1)}g(G^{-1}(t)), \quad t>T_8(\epsilon).
\]
By \eqref{eq.gest6qliminf} we have
\[
x'(t)>-ag(x(t))+b(1-\epsilon)^{(\beta+1)/2}\lambda_n^\beta (1-q)^{-\beta/(\beta-1)}g(G^{-1}(t)), \quad t>T_8(\epsilon).
\]
Using these two inequalities and \eqref{eq.xLqliminfic} by the comparison principle we get
\[
x(t)>x_{L,\epsilon}(t)=\Lambda_1(1-\epsilon)G^{-1}(t), \quad t\geq T_8(\epsilon).
\]
Therefore we have
\[
\liminf_{t\to\infty} \frac{x(t)}{G^{-1}(t)}\geq \Lambda_1(1-\epsilon).
\]
Letting $\epsilon\in(0,1)$ gives
\[
\liminf_{t\to\infty} \frac{x(t)}{G^{-1}(t)}\geq \Lambda_1.
\]
Since we also have $\limsup_{t\to\infty} x(t)/G^{-1}(t)=\Lambda_1$, we get $\lim_{t\to\infty} x(t)/G^{-1}(t)=\Lambda_1$, proving statement (b).

We now come to the question of determining a lower bound on the liminf. Scrutiny of the proof of statement (a) above reveals that the argument
is still valid in the case when $\Lambda_1\geq \Lambda$. We still have
\[
\liminf_{t\to\infty} \frac{x(t)}{G^{-1}(t)}\geq a^{-1/(\beta-1)}=\lambda_1,
\]
and $\lambda_1<\Lambda\leq \Lambda_1$. Now since $\Lambda_1\geq \Lambda$ we have by Lemma~\ref{lemma.existlambda} that $\lambda_n<\Lambda_1$
for all $n\in\mathbb{N}$, and that $\lambda_n\to \Lambda$ as $n\to\infty$. Therefore, we may apply statement (a) arbitrarily many times:
since $\lambda_2<\Lambda_1$, and
\[
\liminf_{t\to\infty} \frac{x(t)}{G^{-1}(t)}\geq \lambda_1,
\]
by statement (a) we have
\[
\liminf_{t\to\infty} \frac{x(t)}{G^{-1}(t)}\geq \lambda_2.
\]
Since $\lambda_3<\Lambda_1$, by statement (a) we have
\[
\liminf_{t\to\infty} \frac{x(t)}{G^{-1}(t)}\geq \lambda_3.
\]
Continuing in this manner we find that
\[
\liminf_{t\to\infty} \frac{x(t)}{G^{-1}(t)}\geq \lambda_n, \quad \text{for all $n\in\mathbb{N}$}.
\]
Letting $n\to\infty$ we get \eqref{eq.xpropdelayasy1liminf}.
\end{proof}

\section{Proof of Results for Linear Delay}
\subsection{Proof of Lemma~\ref{lemma.sigtauasy}}
We first consider the case when $\lambda$ in \eqref{eq.sigmatlambda} obeys $\lambda\in(0,\infty)$. Then for every $\varepsilon\in(0,1)$
there exists $T_0(\varepsilon)>0$ such that
\[
\frac{1}{\lambda}\frac{1}{1+\varepsilon}\frac{1}{t} < \frac{1}{\sigma(t)} <
\frac{1}{\lambda}\frac{1}{1-\varepsilon}\frac{1}{t}, \quad t>T_0(\varepsilon).
\]
Since $\tau$ obeys \eqref{eq.tminustautoinfty}, there exists $T_1(\varepsilon)>T_0(\varepsilon)$ such that $t-\tau(t)>T_0(\varepsilon)$
for all $t>T_1(\varepsilon)$. Therefore for $t>T_1(\varepsilon)$ we have
\[
\frac{1}{\lambda}\frac{1}{1+\varepsilon}\int_{t-\tau(t)}^t \frac{1}{s}\,ds \leq \int_{t-\tau(t)}^t\frac{1}{\sigma(s)}\,ds \leq 
\frac{1}{\lambda}\frac{1}{1-\varepsilon}\int_{t-\tau(t)}^t \frac{1}{s}\,ds.
\]
Therefore by \eqref{eq.t3} we have
\[
\limsup_{t\to\infty} \frac{1}{\lambda}\frac{1}{1+\varepsilon} \log\left(\frac{t}{t-\tau(t)}\right) \leq 1,
\quad
\liminf_{t\to\infty} \frac{1}{\lambda}\frac{1}{1+\varepsilon} \log\left(\frac{t}{t-\tau(t)}\right) \geq 1.
\]
This implies
\[
\lim_{t\to\infty}  \log\left(\frac{t}{t-\tau(t)}\right) = \lambda.
\]
Rearranging and using the continuity of $\tau$ we have \eqref{eq.tauasytot}.

Next we consider the case when $\lambda=0$. Then for every $\varepsilon\in(0,1)$
there exists $T_0(\varepsilon)>0$ such that $\sigma(t)< \varepsilon t$ for all $t>T_0(\varepsilon)$.
Since $\tau$ obeys \eqref{eq.tminustautoinfty}, there exists $T_1(\varepsilon)>T_0(\varepsilon)$ such that $t-\tau(t)>T_0(\varepsilon)$
for all $t>T_1(\varepsilon)$. Therefore for $t>T_1(\varepsilon)$ we have
\[
\frac{1}{\varepsilon} \log\left(\frac{t}{t-\tau(t)}\right) = \frac{1}{\varepsilon}\int_{t-\tau(t)}^t \frac{1}{s}\,ds \leq \int_{t-\tau(t)}^t\frac{1}{\sigma(s)}\,ds.
\]
Therefore by \eqref{eq.t3} we have
\[
\limsup_{t\to\infty}  \log\left(\frac{t}{t-\tau(t)}\right) \leq \varepsilon.
\]
Since $\tau(t)>0$ for all $t$ sufficiently large we have $t-\tau(t)<t$. Therefore $\log (t/(t-\tau(t)))\geq 0$ for all $t$ sufficiently large.
This implies that
\[
0\leq
\limsup_{t\to\infty}  \log\left(\frac{t}{t-\tau(t)}\right) \leq \varepsilon,
\]
and so
\[
\lim_{t\to\infty} \log\left(\frac{t}{t-\tau(t)}\right) =0.
\]
This rearranges to give $\tau(t)/t\to 0$ as $t\to\infty$, which for $\lambda=0$ implies \eqref{eq.tauasytot}.

Finally, we consider the case when $\lambda=\infty$. Then for every $M>0$
there exists $T_0(M)>0$ such that $\sigma(t)> M t$ for all $t>T_0(M)$.
Since $\tau$ obeys \eqref{eq.tminustautoinfty}, there exists $T_1(M)>T_0(M)$ such that $t-\tau(t)>T_0(M)$
for all $t>T_1(M)$. Therefore for $t>T_1(M)$ we have
\[
\frac{1}{M} \log\left(\frac{t}{t-\tau(t)}\right) = \frac{1}{M}\int_{t-\tau(t)}^t \frac{1}{s}\,ds \geq \int_{t-\tau(t)}^t\frac{1}{\sigma(s)}\,ds.
\]
Therefore by \eqref{eq.t3} we have
\[
\liminf_{t\to\infty}  \log\left(\frac{t}{t-\tau(t)}\right) \geq M.
\]
Since $M>0$ is arbitrary, we have
\[
\lim_{t\to\infty} \log\left(\frac{t}{t-\tau(t)}\right) =\infty.
\]
This rearranges to give $\tau(t)/t\to 1$ as $t\to\infty$, which for $\lambda=\infty$ implies \eqref{eq.tauasytot},
if we interpret $1-e^{-\lambda}=1$ in this case.

\subsection{Proof of Theorem~\ref{thm3.1}}

By \eqref{eq:g3}, we note that there is a $\delta_0>0$ such that $g'(x)>0$ for all $x\in(0,\delta_0)$. Hence $g^{-1}:[0,g(\delta_0)]\to [0,\delta_0]$. Clearly there is a $\delta_1>0$ such that $\delta_1<g(\delta_0)$, or $g^{-1}(\delta_1)<\delta_0$. Then for $x\in(0,\delta_1)$, we may define $\Gamma_1:(0,\delta_1)\to(0,\infty)$ by
\[
\Gamma_1(x)=g'(g^{-1}(x)), \quad x\in[0,\delta_1).
\]
Since $g'\in \text{RV}_0(\beta-1)$ and $g(0)=0$, we have that $g\in \text{RV}_0(\beta)$.
Therefore $g^{-1}\in\text{RV}_0(1/\beta)$.
Hence $\Gamma_1\in \text{RV}_0((\beta-1)/\beta)$. Therefore we may apply Theorem~\ref{thm3.1gen} to get the desired result.

\subsection{Proof of Theorem~\ref{thm3.1q}}
Let $\lambda=\log(1/(1-q))>0$. Define $\sigma(t)=\lambda(t+\bar{\tau}+1)$ for $t\geq -\bar{\tau}$.
Since $\tau(t)/t\to q\in(0,1)$ as $t\to\infty$, we have that $\sigma$ obeys \eqref{eq.t1}, \eqref{eq.t2} and \eqref{eq.t3}.

By \eqref{eq:g3}, we note that there is a $\delta_0>0$ such that $g'(x)>0$ for all $x\in(0,\delta_0)$. Hence $g^{-1}:[0,g(\delta_0)]\to [0,\delta_0]$. Clearly there is a $\delta_1>0$ such that $\delta_1<g(\delta_0)$, or $g^{-1}(\delta_1)<\delta_0$. Then for $x\in(0,\delta_1)$, we may define $\Gamma_1:(0,\delta_1)\to(0,\infty)$ by
\[
\Gamma_1(x)=g'(g^{-1}(x)), \quad x\in[0,\delta_1).
\]
Since $g'\in \text{RV}_0(\beta-1)$ and $g(0)=0$, we have that $g\in \text{RV}_0(\beta)$.
Therefore $g^{-1}\in\text{RV}_0(1/\beta)$. Hence $\Gamma_1\in \text{RV}_0((\beta-1)/\beta)$, which implies
that $g$ obeys \eqref{eq:galt6} with $\gamma=\beta/(\beta-1)$. We note that \eqref{eq.abqbig} implies \eqref{eq.abqbigsuperhyp}, so all the hypotheses of Theorem~\ref{thm3.1superhypq} hold, and so
\[
\lim_{t\to\infty} \frac{\log g(x(t))}{\int_0^t \frac{1}{\sigma(s)}\,ds}=-\log\left( \frac{a}{b}\right).
\]
Since $g\in\text{RV}_0(\beta)$ and $x(t)\to 0$ as $t\to\infty$, and
\[
\lim_{t\to\infty} \frac{\int_0^t \frac{1}{\sigma(s)}\,ds}{\log t}=\frac{1}{\lambda}
\]
we have
\[
\lim_{t\to\infty} \frac{\log x(t)}{\log t}
=
\lim_{t\to\infty} \frac{\log x(t)}{\log g(x(t))}\cdot\frac {\log g(x(t))}{\int_0^t \frac{1}{\sigma(s)}\,ds}\cdot\frac{\int_0^t \frac{1}{\sigma(s)}\,ds}{\log t}
=
\frac{1}{\beta}\cdot -\log(a/b) \cdot \lambda.
\]
Since $\lambda=\log(1/(1-q))$ we have
\[
\lim_{t\to\infty} \frac{\log x(t)}{\log t}=-\frac{1}{\beta}\frac{1}{\log(1/(1-q))}\log(a/b),
\]
as required.

\section{Proofs for Max--Type Equation}
\subsection{Proof of Theorem~\ref{thm.odedecaymax}}
We first need to prove that
\begin{equation}  \label{eq.GxgtLambdatmax}
G(x(t))\geq \Lambda t, \quad t\geq1.
\end{equation}

Since $a>b$ we may fix $\epsilon\in(0,1)$ so small that \eqref{eq.condabepsa} holds.
Define $G_0$ by \eqref{def.G0}. Then $G_0$ is decreasing on $(0,\infty)$ and therefore invertible.
Since $x(t)>0$ for all $t\geq 0$, $b>0$ and $g(x(t-\tau(t)))>0$ for all $t\geq 0$ we have
\[
x'(t)\geq -ag(x(t)), \quad t>0.
\]
This gives \eqref{eq.Gxleqat}, as required.
Define $\Gamma:[0,\infty)\to (0,\infty)$ by \eqref{def.Gamma}.
Therefore for $\epsilon\in(0,1)$, there is a $x_1(\epsilon)>0$ such that $\Gamma$ obeys \eqref{eq.gammarvepsa}.
There also exists $\delta(\epsilon)>0$ defined in terms of $G_0$, $\epsilon$, $x_1(\epsilon)$ which obeys \eqref{eq.epsGdelxa}, and for $\epsilon\in(0,1)$, because $x(t)\to 0$ as $t\to\infty$ and $\tau(t)/t\to 0$ as $\to\infty$, there is a $T_1(\epsilon)>0$ such that
\begin{equation} \label{eq.T1defmax}
x(t)\leq \delta(\epsilon), \quad \tau(t)< \epsilon t, \quad t>T_1(\epsilon).
\end{equation}
Next define
\begin{equation} \label{def.lambdaepsmax}
0<\lambda(\epsilon)=\frac{\epsilon(1-\epsilon)G_0(\delta(\epsilon))}{T_1(\epsilon)}.
\end{equation}
Define $x_{U,\epsilon}$ by
\begin{equation} \label{def.xUodemax}
x_{U,\epsilon}(t)=G_0^{-1}(\lambda(\epsilon)t),\quad t\geq T_1(\epsilon).
\end{equation}
Define $T_2(\epsilon):=T_1(\epsilon)/(1-\epsilon)>T_1(\epsilon)$.
For $t\in[T_1(\epsilon), T_2(\epsilon)]$ we can prove as in Theorem~\ref{thm.odedecay} that
\begin{equation} \label{eq.xUicodemax}
x_{U,\epsilon}(t)>\delta(\epsilon)\geq x(t), \quad t\in[T_1(\epsilon),T_2(\epsilon)].
\end{equation}
and also that
\begin{equation} \label{eq.xUinmonotoneregionmax}
x_{U,\epsilon}(t)<\delta_1, \quad t\geq T_1(\epsilon).
\end{equation}
Proceeding as in the proof of Theorem~\ref{thm.odedecay} it can then be shown that
\begin{equation} \label{eq.xUdiffineq2max}
x_{U,\epsilon}'(t)>-ag(x_{U,\epsilon}(t))+b\max_{t-\tau(t)\leq s\leq t}g(x_{U,\epsilon}(s)), \quad t>T_2(\epsilon),
\end{equation}
using the fact that $x_{U,\epsilon}$ is decreasing on $[T_1(\epsilon),\infty)$ to simplify the righthand side of \eqref{eq.xUdiffineq2max}, because
\[
\max_{t-\tau(t)\leq s\leq t}g(x_{U,\epsilon}(s))=g(x_{U,\epsilon}(t-\tau(t))).
\]
By \eqref{eq.xUicodemax} and \eqref{eq.xUdiffineq2max}, by the comparison principle we have $x(t)<x_{U,\epsilon}(t)$ for all $t\geq T_1(\epsilon)$.
Therefore $x(t)<G_0^{-1}(\lambda(\epsilon)t)$ for $t\geq T_1(\epsilon)$. Hence $G_0(x(t))>\lambda(\epsilon)t$ for $t\geq T_1(\epsilon)$.
Since $\epsilon\in(0,1)$ obeying \eqref{eq.condabepsa} is fixed we have \eqref{eq.GxgtLambdatmax}.

Armed with \eqref{eq.GxgtLambdatmax}, we now prove \eqref{eq.odedecaymax}.

Since $t-\tau(t)\to\infty$ as $t\to\infty$, there exists $T_0>1$
such that $t-\tau(t)\geq 1$ for all $t\geq T_0$. Recalling that $x'(t)\geq -ag(x(t))$ for all $t\geq 0$, for all $t\geq T_0$ and $s\in[t-\tau(t),t]$ we have
\[
G_0(x(s)))-G_0(x(t))=\int_{x(s)}^{x(t)} \frac{1}{g(u)}\,du = \int_{s}^t \frac{x'(v)}{g(x(v))}\,dv \geq -a(t-s),
\]
so
\begin{equation} \label{eq.G0tG0ttaumax}
G_0(x(t))\leq a(t-s)+G_0(x(s))\leq a\tau(t)+G_0(x(s)), \quad t\geq T_0, \quad s\in[t-\tau(t),t].
\end{equation}
Define $G_1(x)=1/G_0(x)$ for $x>0$. Then $G_1\in \text{RV}_0(\beta-1)$ and $G_1$ is increasing. Since $\beta>1$ we have that  $G_1^{-1}\in\text{RV}_0(1/(\beta-1))$. Rearranging \eqref{eq.G0tG0ttaumax} we have
\[
\frac{G_1(x(s))}{G_1(x(t))}\leq 1+a\tau(t)G_1(x(s)), \quad t\geq T_0, \quad s\in[t-\tau(t),t].
\]
Therefore for $t\geq T_0$ we have for $s\in[t-\tau(t),t]$
\begin{align*}
\frac{G_1(x(s))}{G_1(x(t))}&\leq 1+a\frac{\tau(t)}{t}\cdot\frac{t}{s}\cdot\frac{s}{G_0(x(s))}\\
&\leq 1+a\frac{\tau(t)}{t}\cdot\frac{t}{t-\tau(t)}\cdot\sup_{s\geq t-\tau(t)}\frac{s}{G_0(x(s))}.
\end{align*}
Therefore
\[
\frac{\max_{t-\tau(t)\leq s\leq t} G_1(x(s))}{G_1(x(t))}
\leq 1+a\frac{\tau(t)}{t}\cdot\frac{t}{t-\tau(t)}\cdot\sup_{s\geq t-\tau(t)}\frac{s}{G_0(x(s))}.
\]
Since $G_1$ is increasing we have
\[
\max_{t-\tau(t)\leq s\leq t} G_1(x(s))=G_1(\max_{t-\tau(t)\leq s\leq t} x(s)).
\]
Therefore as $\tau(t)/t\to 0$ as $t\to\infty$ and \eqref{eq.GxgtLambdatmax} holds we have
\begin{equation*} 
\limsup_{t\to\infty} \frac{G_1(\max_{t-\tau(t)\leq s\leq t} x(s))}{G_1(x(t))}\leq 1.
\end{equation*}
Therefore for every $\epsilon\in(0,1)$ there is $T_1(\epsilon)>0$ such that
\[
G_1(\max_{t-\tau(t)\leq s\leq t} x(s))\leq (1+\epsilon)G_1(x(t)), \quad t\geq T_1(\epsilon).
\]
Therefore as $G_1^{-1}$ is increasing, we have
\[
\frac{\max_{t-\tau(t)\leq s\leq t} x(s)}{x(t)}\leq \frac{G_1^{-1}((1+\epsilon)G_1(x(t)))}{G_1^{-1}(G_1(x(t)))}, \quad t\geq T_1(\epsilon).
\]
Now $G_1(x(t))\to 0$ as $t\to\infty$ and $G_1^{-1}\in \text{RV}_0(1/(\beta-1))$, so
\[
\limsup_{t\to\infty}
\frac{\max_{t-\tau(t)\leq s\leq t} x(s)}{x(t)}
\leq
\lim_{t\to\infty} \frac{G_1^{-1}((1+\epsilon)G_1(x(t)))}{G_1^{-1}(G_1(x(t)))}
=
(1+\epsilon)^{1/(\beta-1)}.
\]
Letting $\epsilon\to 0^+$ gives
\[
\limsup_{t\to\infty} \frac{\max_{t-\tau(t)\leq s\leq t} x(s)}{x(t)}\leq 1.
\]
Therefore there exists $T_2(\epsilon)>0$ such that $\max_{t-\tau(t)\leq s\leq t} x(s)\leq (1+\epsilon)x(t)$ for all $t\geq T_2(\epsilon)$.
Since $g$ is increasing $g(\max_{t-\tau(t)\leq s\leq t} x(s))\leq g((1+\epsilon)x(t))$ for all $t\geq T_2(\epsilon)$. Therefore as $g\in\text{RV}_0(\beta)$ and
$x(t)\to 0$ as $t\to\infty$ we have
\[
\limsup_{t\to\infty} \frac{g(\max_{t-\tau(t)\leq s\leq t} x(s)))}{g(x(t))}\leq \limsup_{t\to\infty}\frac{g((1+\epsilon)x(t))}{g(x(t))}=(1+\epsilon)^\beta.
\]
Letting $\epsilon\to 0^+$ we have
\[
\limsup_{t\to\infty} \frac{g(\max_{t-\tau(t)\leq s\leq t} x(s))}{g(x(t))}\leq 1.
\]
Since $g$ is increasing, we have $g(\max_{t-\tau(t)\leq s\leq t} x(s))=\max_{t-\tau(t)\leq s\leq t} g(x(s))$ so
\[
\limsup_{t\to\infty} \frac{\max_{t-\tau(t)\leq s\leq t} g(x(s))}{g(x(t))}\leq 1.
\]
On the other hand it is trivially true that $\max_{t-\tau(t)\leq s\leq t} g(x(s))\geq g(x(t))$. Therefore
\[
\liminf_{t\to\infty} \frac{\max_{t-\tau(t)\leq s\leq t} g(x(s))}{g(x(t))}\geq 1.
\]
Combining these limits gives
\[
\lim_{t\to\infty} \frac{\max_{t-\tau(t)\leq s\leq t} g(x(s))}{g(x(t))}=1,
\]
which implies
\[
\lim_{t\to\infty} \frac{x'(t)}{g(x(t))}=\lim_{t\to\infty} -a + b \frac{\max_{t-\tau(t)\leq s\leq t} g(x(s))}{g(x(t))}=-(a-b).
\]
Proceeding as we did at the end of the proof of Theorem~\ref{thm.odedecay}, we
obtain \eqref{eq.odedecaymax}, as required.

\section{Proofs of Lemmata and Justification of Claims from Section~\ref{sec.nonrv}}
\subsection{Proof of Lemma~\ref{lemma.ex1GinvGammaasy}}
Then
\[
\lim_{x\to 0^+} \frac{G(x)}{e^{1/x^\alpha} x^{\alpha+1}}
=\lim_{x\to 0^+} \frac{\int_1^{1/x} \frac{1}{g(1/v)}v^{-2}\,dv}{e^{1/x^\alpha} x^{\alpha+1}}  
=\lim_{y\to\infty} \frac{\int_1^y \frac{1}{g(1/v)}v^{-2}\,dv}{e^{y^\alpha} y^{-(\alpha+1)}}.
\]
Therefore
\[
\lim_{x\to 0^+} \frac{G(x)}{e^{1/x^\alpha} x^{\alpha+1}}
=
\lim_{y\to\infty} \frac{\frac{1}{g(1/y)}y^{-2}}{-(\alpha+1) e^{y^\alpha} y^{-(\alpha+1)-1}+\alpha y^{\alpha-1} e^{y^\alpha} y^{-(\alpha+1)}}.
\]
Hence
\begin{equation} \label{eq.ex1nonrvGasy}
\lim_{x\to 0^+} \frac{G(x)}{e^{1/x^\alpha} x^{\alpha+1}}
=
\lim_{y\to\infty} \frac{y^{-2}}{-(\alpha+1)  y^{-\alpha-2} +\alpha y^{-2}}
=\frac{1}{\alpha}.
\end{equation}
Now this implies
\[
\lim_{y\to\infty} \frac{y}{e^{1/G^{-1}(y)^\alpha} G^{-1}(y)^{\alpha+1}}
=\frac{1}{\alpha},
\]
so
\[
\lim_{y\to\infty} \left\{ \log y - \frac{1}{G^{-1}(y)^\alpha} - (\alpha+1) \log G^{-1}(y)\right\}=\log(1/\alpha).
\]
Since $G^{-1}(y)\to\infty$ as $y\to\infty$ we have
\[
\lim_{y\to\infty} \frac{1}{G^{-1}(y)^\alpha} \left\{G^{-1}(y)^\alpha \log y - 1 - (\alpha+1)G^{-1}(y)^\alpha \log G^{-1}(y)\right\}=\log(1/\alpha),
\]
from which \eqref{eq.ex1nonrvGinvasy} can be inferred.

Now we turn to the asymptotic behaviour of $\Gamma(x)=g(G^{-1}(x))$. By definition $\Gamma(G(x))=g(x)$. Therefore by \eqref{eq.ex1nonrvgeeasy} we
have
\[
\lim_{x\to 0^+} \frac{\Gamma(G(x))}{e^{-1/x^\alpha}}=1.
\]
By \eqref{eq.ex1nonrvGasy} we have
\begin{equation} \label{eq.ex1exalpha}
\lim_{x\to 0^+}
\frac{e^{-1/x^\alpha}}{ x^{\alpha+1}/G(x)}=\frac{1}{\alpha},
\end{equation}
so therefore we have the limit
\[
\lim_{x\to 0^+} \frac{\Gamma(G(x))}{x^{\alpha+1}/G(x)}
=
\lim_{x\to 0^+} \frac{e^{-1/x^\alpha}}{x^{\alpha+1}/G(x)}\cdot \frac{\Gamma(G(x))}{e^{-1/x^\alpha}}
=\frac{1}{\alpha}.
\]
Therefore
\[
\lim_{y\to \infty} \frac{\Gamma(y)}{G^{-1}(y)^{\alpha+1}/y}
=
\lim_{x\to 0^+} \frac{\Gamma(G(x))}{x^{\alpha+1}/G(x)}
=\frac{1}{\alpha}.
\]
Combining this with \eqref{eq.ex1nonrvGinvasy} yields \eqref{eq.ex1nonrvGammaasy}.
%
%

\subsection{Proof of Lemma~\ref{lemma.ex1GinvGamma1asy}}
Since $g(0)=0$, integration yields
\begin{equation} \label{eq.ex1gasy2}
\lim_{x\to 0^+} \frac{g(x)}{e^{-1/x^\alpha}}
=
\lim_{x\to 0^+} \frac{\int_0^x g'(s)\,ds}{\int_0^x \alpha s^{-\alpha-1} e^{-1/s^\alpha}\,ds}=1.
\end{equation}
Therefore
\[
\lim_{x\to 0^+} \frac{x}{e^{-1/g^{-1}(x)^\alpha}}=1.
\]
Hence we have
\[
\lim_{x\to 0^+} -\log (1/x) + \frac{1}{g^{-1}(x)^\alpha}=0.
\]
This implies
\begin{equation} \label{eq.ex1ginvasya}
\lim_{x\to 0^+} \frac{g^{-1}(x)}{\log(1/x)^{-1/\alpha}} = 1.
\end{equation}
We have $\Gamma_1(x)=g'(g^{-1}(x))$, so $\Gamma_1(g(x))=g'(x)$. Therefore
\[
\lim_{x\to 0^+} \frac{\Gamma_1(g(x))}{\alpha x^{-\alpha-1} e^{-1/x^\alpha}}=1.
\]
Therefore by \eqref{eq.ex1gasy2}
\[
\lim_{x\to 0^+} \frac{\Gamma_1(g(x))}{\alpha x^{-(\alpha+1)} g(x)}=1,
\]
which implies
\[
\lim_{y\to 0^+} \frac{\Gamma_1(y)}{\alpha g^{-1}(y)^{-(\alpha+1)} y}=1.
\]
By \eqref{eq.ex1ginvasya} we have \eqref{eq.ex1nonrvGamma1asy}.
Since for any $\mu>0$ we have $\lim_{y\to 0^+} \Gamma_1(\mu y)/\Gamma_1(y)=\mu$, it follows that $\Gamma_1\in \text{RV}_0(1)$ as required.

\subsection{Proof of Lemma~\ref{lemma.ex2GinvGammaasy}}
First we have
\[
\lim_{x\to 0^+} \frac{G(x)}{\exp(e^{1/x})e^{-1/x}x^2}=\lim_{x\to 0^+} \frac{\int_x^1 \frac{1}{g(u)}\,du}{\exp(e^{1/x})e^{-1/x}x^2}
=\lim_{y\to \infty} \frac{\int_1^y \frac{1}{g(1/v)}v^{-2}\,dv}{\exp(e^y)e^{-y}y^{-2}}.
\]
Hence by L'H\^opital's rule, we have
\begin{align*}
\lim_{x\to 0^+} \frac{G(x)}{\exp(e^{1/x})e^{-1/x}x^2}
&=\lim_{y\to \infty} \frac{\frac{1}{g(1/y)}y^{-2}}{\exp(e^y)y^{-2}+\exp(e^y)\left(-2e^{-y}y^{-3} - e^{-y}y^{-2}\right)}\\
&=\lim_{y\to \infty} \frac{1}{1-e^{-y}(2y^{-1} +1)}=1.
\end{align*}
Therefore
\[
\lim_{y\to\infty} \frac{y}{\exp(e^{1/G^{-1}(y)})e^{-1/G^{-1}(y)}G^{-1}(y)^2}=1,
\]
which implies
\[
\lim_{y\to\infty} \log y - \left( e^{1/G^{-1}(y)} -1/G^{-1}(y)+2\log G^{-1}(y)\right)=0.
\]
This implies that there is a function $\theta$ with $\theta(y)\to 0$ as $t\to\infty$ such that
\[
\lim_{y\to\infty} e^{1/G^{-1}(y)}\left(\frac{\log y}{e^{1/G^{-1}(y)}}-1-\theta(y)\right)=0,
\]
from which we infer
\[
\lim_{y\to\infty} \frac{e^{1/G^{-1}(y)}}{\log y}=1.
\]
Taking logarithms and arguing in a similar manner we obtain \eqref{eq.ex2nonrvGinvasy}.

To determine the asymptotic behaviour of $\Gamma$, we note that $\Gamma(G(x))=g(x)$ so
\[
\lim_{x\to 0^+} \frac{\Gamma(G(x))}{\exp(-e^{1/x})}=1.
\]
Since
\[
\lim_{x\to 0^+} \frac{e^{1/x}x^{-2}G(x)}{\exp(e^{1/x})}=1,
\]
we have
\[
1
=\lim_{x\to 0^+} \Gamma(G(x)) e^{1/x}x^{-2}G(x)
=\lim_{y\to \infty} \Gamma(y) e^{1/G^{-1}(y)}G^{-1}(y)^{-2} y.
\]
Therefore
\[
1=\lim_{y\to \infty} \Gamma(y) e^{1/G^{-1}(y)}G^{-1}(y)^{-2} y
=\lim_{y\to \infty} \Gamma(y) \log y (\log_2 y)^{2} y,
\]
which is \eqref{eq.ex2nonrvGammaasy}.

\subsection{Proof of Lemma~\ref{lemma.ex2GinvGamma1asy}}
Since $g(0)=0$, integration yields
\begin{equation} \label{eq.ex2gasy2}
\lim_{x\to 0^+} \frac{g(x)}{\exp(-e^{1/x})}
=
\lim_{x\to 0^+} \frac{\int_0^x g'(s)\,ds}{\int_0^x s^{-2}e^{1/s}\exp(-e^{1/s})\,ds}=1.
\end{equation}
Therefore
\[
\lim_{x\to 0^+} \frac{x}{\exp(-e^{1/g^{-1}(x)})}=1.
\]
Hence we have
\[
\lim_{x\to 0^+} -\log (1/x) + e^{1/g^{-1}(x)}=0.
\]
This implies
\begin{equation} \label{eq.ex2ginvasya}
\lim_{x\to 0^+} \frac{e^{1/g^{-1}(x)}}{\log(1/x)} = 1,
\quad
\lim_{x\to 0^+} \frac{1/g^{-1}(x)}{\log_2(1/x)} = 1.
\end{equation}
We have $\Gamma_1(x)=g'(g^{-1}(x))$, so $\Gamma_1(g(x))=g'(x)$. Therefore
\[
\lim_{x\to 0^+} \frac{\Gamma_1(g(x))}{x^{-2}e^{1/x}\exp(-e^{1/x})}=1.
\]
Therefore by \eqref{eq.ex2gasy2}
\[
\lim_{x\to 0^+} \frac{\Gamma_1(g(x))}{x^{-2}e^{1/x} g(x)}=1,
\]
which implies
\[
\lim_{y\to 0^+} \frac{\Gamma_1(y)}{g^{-1}(y)^{-2} e^{1/g^{-1}(y)} y}=1.
\]
By \eqref{eq.ex2ginvasya} we have \eqref{eq.ex2nonrvGamma1asy}.
Since for any $\mu>0$ we have $\lim_{y\to 0^+} \Gamma_1(\mu y)/\Gamma_1(y)=\mu$, it follows that $\Gamma_1\in \text{RV}_0(1)$ as required.

\subsection{Justification of Example~\ref{example.asygeminuspoly}}
\subsubsection{Justification of part (b).}
Since $\lim_{t\to\infty} g(x(t))/e^{-1/x(t)^\alpha}=1$, we have
so
\[
\lim_{t\to\infty} \left\{\log g(x(t)) +1/x(t)^\alpha\right\}=0,
\]
By Theorem~\ref{thm3.1superhypqgamma1} we have
\[
\lim_{t\to\infty} \frac{\log g(x(t))}{\log t} = -\frac{\log(a/b)}{\log(1/(1-q))},
\]
so
\[
\lim_{t\to\infty} \frac{1}{x(t)^\alpha\log t} = \frac{\log(a/b)}{\log(1/(1-q))},
\]
and the result follows.

\subsubsection{Justification of part (c).}
It can be shown that the function $\sigma(t)=\gamma(t+2\bar{\tau}+e^2)\log_2(t+2\bar{\tau}+e^2)$ for $t\geq -\bar{\tau}$ obeys \eqref{eq.t1}--\eqref{eq.t4} with
\[
\lim_{t\to\infty} \frac{\int_0^t \frac{1}{\sigma(s)}\,ds}{\frac{1}{\gamma}\log t/\log_2 t}=1.
\]
Therefore by Theorem~\ref{thm3.1nonrv} we have
\[
\lim_{t\to\infty} \frac{-1/x(t)^\alpha}{\frac{1}{\gamma}\log t/\log_2 t}=-\log\left(\frac{a}{b}\right),
\]
from which the result follows.

\subsubsection{Justification of part (d).}
It can be shown that the function $\sigma$ defined by $\sigma(t)=\log(1/\gamma)(t+2\bar{\tau}+1)\log(t+2\bar{\tau}+1)$ for $t\geq -\bar{\tau}$ obeys \eqref{eq.t1}--\eqref{eq.t4} with
\[
\lim_{t\to\infty} \frac{\int_0^t \frac{1}{\sigma(s)}\,ds}{\frac{1}{\log(1/\gamma)}\log_2 t}=1.
\]
Therefore by Theorem~\ref{thm3.1nonrv} we have
\[
\lim_{t\to\infty} \frac{-1/x(t)^\alpha}{\frac{1}{\log(1/\gamma)}\log_2 t}=-\log\left(\frac{a}{b}\right),
\]
from which the result follows.

\subsection{Justification of Example~\ref{example.asygeminusexp}}
\subsubsection{Justification of part (b).}
Since $\lim_{t\to\infty} g(x(t))/\exp(-e^{1/x(t)})=1$, we have
so
\[
\lim_{t\to\infty} \left\{\log g(x(t)) +e^{1/x(t)}\right\}=0.
\]
By Theorem~\ref{thm3.1superhypqgamma1} we have
\[
\lim_{t\to\infty} \frac{\log g(x(t))}{\log t} = -\frac{\log(a/b)}{\log(1/(1-q))},
\]
so
\[
\lim_{t\to\infty} \frac{e^{1/x(t)}}{\log t} = \frac{\log(a/b)}{\log(1/(1-q))},
\]
and the result follows.

\subsubsection{Justification of part (c).}
It can be shown that the function $\sigma(t)=\gamma(t+2\bar{\tau}+e^2)\log_2(t+2\bar{\tau}+e^2)$ for $t\geq -\bar{\tau}$ obeys \eqref{eq.t1}--\eqref{eq.t4} with
\[
\lim_{t\to\infty} \frac{\int_0^t \frac{1}{\sigma(s)}\,ds}{\frac{1}{\gamma}\log t/\log_2 t}=1.
\]
Therefore by Theorem~\ref{thm3.1nonrv} we have
\[
\lim_{t\to\infty} \frac{e^{1/x(t)}}{\log t/\log_2 t}=\frac{1}{\gamma}\log\left(\frac{a}{b}\right).
\]
Hence
\[
\lim_{t\to\infty} 1/x(t) - \log \left(\log t/\log_2 t\right)=\log\left(\frac{1}{\gamma}\log\left(\frac{a}{b}\right)\right),
\]
from which the result follows.

\subsubsection{Justification of part (d).}
It can be shown that the function $\sigma$ defined by $\sigma(t)=\log(1/\gamma)(t+2\bar{\tau}+1)\log(t+2\bar{\tau}+1)$ for $t\geq -\bar{\tau}$ obeys \eqref{eq.t1}--\eqref{eq.t4} with
\[
\lim_{t\to\infty} \frac{\int_0^t \frac{1}{\sigma(s)}\,ds}{\frac{1}{\log(1/\gamma)}\log_2 t}=1.
\]
Therefore by Theorem~\ref{thm3.1nonrv} we have
\[
\lim_{t\to\infty} \frac{e^{1/x(t)}}{\log_2 t}=\frac{1}{\log(1/\gamma)}\log\left(\frac{a}{b}\right).
\]
Hence
\[
\lim_{t\to\infty} 1/x(t)- \log_3 t=\log\left(\frac{1}{\log(1/\gamma)}\log\left(\frac{a}{b}\right)\right),
\]
from which the result follows.

\end{document}